\documentclass[12pt]{article}

\usepackage{amsmath}
\usepackage{amssymb}
\usepackage{amsfonts}
\usepackage{amsthm} 
\usepackage{enumerate}

\newtheorem{theorem}{Theorem}[section]
\newtheorem{proposition}[theorem]{Proposition}
\newtheorem{lemma}[theorem]{Lemma}
\newtheorem{definition}[theorem]{Definition}
\newtheorem{corollary}[theorem]{Corollary}

\numberwithin{equation}{section}

\newcommand{\rr}{{\mathbb R}}
\newcommand{\zz}{{\mathbb Z}}
\newcommand{\nn}{{\mathbb N}}

\newcommand{\ee}{{\mathbb E}}

\newcommand{\pp}{{\mathbb P}}

\newcommand{\cals}{\mathcal{S}}

\newcommand{\calm}{{\cal M}}
\newcommand{\calf}{{\cal F}}
\newcommand{\calee}{{\cal E}}

\newcommand{\loc}{{\rm loc}}
\newcommand{\diam}{{\rm diam}}
\newcommand{\one}{{\bf 1}}
\newcommand{\half}{{\frac{1}{2}}}

\newcommand{\supp}{\hbox{supp\,}}

\newcommand{\aint}{{\mathfrak A}_{\text{int}}}
\newcommand{\atr}{{\mathfrak A}_{\text{tr}}}

\begin{document}

\title{Maximal operators and differentiation theorems for sparse sets}
\author{Izabella {\L}aba and Malabika Pramanik}
\date{July 21, 2010}
\maketitle


\begin{abstract}

We study maximal averages associated with singular measures
on $\rr$. Our main result is a construction of singular Cantor-type measures supported on sets of Hausdorff dimension $1 - \epsilon$ with $0 \leq \epsilon < \frac{1}{3}$ for
which the corresponding maximal operators are bounded on $L^p(\mathbb R)$ for
$p > (1 + \epsilon)/(1 - \epsilon)$.  
As a consequence, we are able to answer a question
of Aversa and Preiss on density and differentiation theorems for singular measures in one dimension.
Our proof combines probabilistic techniques with the methods developed in 
multidimensional Euclidean harmonic analysis, in particular there are strong
similarities to Bourgain's proof of the circular maximal theorem in two dimensions. 

\end{abstract}


\allowdisplaybreaks{

\section{Introduction}
\subsection{Maximal operators}\label{sec-statement}

Let $\{S_k : k \geq 1 \}$ be a decreasing sequence of subsets of $\rr$. We define
the maximal operator associated with this sequence by
\begin{align}\label{max-e2}
\tilde\calm f(x)&:=\sup_{r>0, k\geq 1} \frac{1}{|S_k|} \int_{S_k} |f(x+ry)|dy. 
\end{align}
While the definition (\ref{max-e2}) is quite general, we will focus on 
cases where the sequence $\{S_k\}$ arises from a Cantor-type iteration, so that in particular
each $S_k$ is a union of finitely many intervals.  We will further assume that 
$|S_k|\to 0$ as $k\to\infty$.

Under mild conditions on the Cantor iteration process, the densities $\phi_k=\frac{1}{|S_k|}
\one_{S_k}$ converge weakly to a probability measure $\mu$ supported on the set
$S = \bigcap_{k=1}^{\infty} S_k$.  We then define the maximal operator with respect to 
$\mu$:
\begin{equation} \label{max-e1}
\tilde{\mathfrak M} f(x) :=  \sup_{r > 0} \int \left| f(x + ry) \right| d\mu(y).
\end{equation}

We will be interested in the $L^p$ mapping properties of  $\tilde{\calm}$.
Since $\tilde{\mathfrak M}$ is clearly dominated by $\tilde{\calm}$, similar estimates will
follow for  $\tilde{\mathfrak M}$ with the same range of exponents. 

We will also be concerned with $L^p\to L^q$ maximal estimates with $p<q$.
For this purpose, it is necessary to define the modified maximal operators
\begin{align}\label{max-e102}
\tilde\calm^a f(x)&:=\sup_{r>0,\ k\geq 1} r^a \int |f(x+ry)|\phi_k(y)dy\ , 
\end{align}
\begin{equation} \label{max-e104}
\tilde{\mathfrak M}^a f(x) :=  \sup_{r > 0} r^a \int \left|f(x + ry) \right| d\mu(y)\  ,
\end{equation}
where the exponent $a=\frac{1}{p}-\frac{1}{q}$ accounts for the appropriate scaling
correction.  Note that $\tilde\calm^0 =\tilde\calm$ and $\tilde{\mathfrak M}^0=\tilde{\mathfrak M}$.

Finally, we will need the restricted maximal operators 
\begin{equation}\label{max-e3}
\calm f(x) :=\sup_{1<r<2,k\geq 1} \frac{1}{|S_k|} \int_{S_k} |f(x+ry)|dy\ ,
\end{equation}
\begin{equation} \label{max-e103}
\mathfrak M f(x) :=  \sup_{1<r<2} \int \left|f(x + ry) \right| d\mu(y)\  ,
\end{equation}
where the range of the dilation factor $r$ is limited to a single scale. These operators will play a critical role in the proofs of the unrestricted maximal estimates.

\subsection{The main results}
\begin{theorem}\label{thm-main}
There is a decreasing sequence of sets $S_k \subseteq [1,2]$ with the following properties: 
\begin{enumerate}[(a)]
\item Each $S_k$ is a disjoint union of finitely many intervals.
\item $|S_k| \searrow 0$ as $k \rightarrow \infty$.
\item The weak-$\ast$ limit  $\mu$ of the densities $\mathbf 1_{S_k}/|S_k|$ exists.
\item The restricted maximal operators $\mathcal M$ and $\mathfrak M$ defined in (\ref{max-e3}) and (\ref{max-e103}) are bounded from $L^p[0,1]$ to $L^q(\rr)$ for any $p,q\in(1,\infty)$, and from $L^p(\rr)$ to $L^q(\rr)$ for any $1<p\leq q<\infty$. \label{thm-main-restricted}
\item The unrestricted maximal operators $\tilde{\mathcal M}^a$ and $\tilde{\mathfrak M}^a$ defined in (\ref{max-e2}) and (\ref{max-e1}) are bounded from $L^p(\rr)$ to
$L^q(\rr)$ whenever $1<p\leq q<\infty$, with $a=\frac{1}{p}-\frac{1}{q}$.  In particular, $\tilde{\mathcal M}$ and $\tilde{\mathfrak M}$ are bounded on $L^p(\rr)$ for $p>1$. \label{thm-main-unrestricted1} 
\end{enumerate}
\end{theorem}

As a corollary, we obtain a differentiation theorem for averages on $S_k$ that answers a 
question of Aversa and Preiss \cite{AP2} (see \S\ref{Aversa-Preiss-description} for more details).

\begin{theorem}\label{AP-diff}
Let $\{S_k: k \geq 1 \}$ be the sequence of sets given by Theorem \ref{thm-main}, with the
limiting measure $\mu$. Then 
for every $f\in L^p(\mathbb R)$ with $p\in (1,\infty)$ we have
\begin{equation}\label{diff-e10}
\lim_{r\to 0}\sup_{k}\left|\frac{1}{r|S_k|}\int_{x+rS_k}f(y)dy- f(x) \right| = 0 \hbox{ for a.e. }x\in\rr,
\hbox{ and }
\end{equation}
\begin{equation}\label{diff-e-meas}
\lim_{r\to 0}\left| \int f(x+ry)d\mu(y)- f(x) \right| = 0 \hbox{ for a.e. }x\in\rr.
\end{equation}
\end{theorem}

The limiting set $S=\bigcap_{k=1}^\infty  S_k$ constructed in our proof of 
Theorem \ref{thm-main} has Hausdorff dimension 1.  However, we are
also able to prove similar maximal estimates for sequences of sets whose limit has Hausdorff 
dimension $1 - \epsilon$ with $\epsilon > 0$, provided that the range of exponents is
adjusted accordingly.

\begin{theorem}\label{thm-main3}
For any $0 < \epsilon < \frac{1}{3}$, there is a decreasing sequence of sets $S_k\subset
[1,2]$ obeying the conditions (a)--(c) of Theorem \ref{thm-main}
and such that: 
\begin{enumerate}[(a)]
\item $S = \bigcap_{k=1}^{\infty} S_k$ has Hausdorff dimension $1 - \epsilon$.
\item The restricted maximal operators $\mathcal M$ and $\mathfrak M$ are bounded 
from $L^p[0,1]$ to $L^q(\rr)$ for any $p,q$ such that \label{thm-main3-restricted} 
\begin{equation}\label{zz-e11}
\frac{1+\epsilon}{1-\epsilon}<p<\infty \text{ and }1<q<\frac{1-\epsilon}{2\epsilon}p,
\end{equation}
and from $L^p(\rr)$ to $L^q(\rr)$ for any $p,q$ such that $p\leq q$ and (\ref{zz-e11}) holds.
\item The unrestricted maximal operators $\tilde{\mathcal M}^a$ and $\tilde{\mathfrak M}^a$ 
are bounded from $L^p(\rr)$ to $L^q(\rr)$ with  $a=\frac{1}{p}-\frac{1}{q}$ for any $p,q$ such that $p \leq q$ and (\ref{zz-e11}) holds. In particular, $\tilde{\mathcal M}$ and $\tilde{\mathfrak M}$ are bounded on $L^p(\rr)$ for $p > \frac{1 + \epsilon}{1 - \epsilon}$. \label{thm-main3-unrestricted1} 
\item The family of sets $\mathcal S = \{rS_k : k \geq 1 \}$ and the measure $\mu$ differentiate $L^p(\mathbb R)$ in the 
sense of (\ref{diff-e10}) and (\ref{diff-e-meas}) for all $p >\frac{1+\epsilon}{1 - \epsilon}$.    
\label{lower-dim-diff}
\end{enumerate}
\end{theorem}


\noindent\textit{Remarks.}  
\begin{enumerate}[1.]
\item \label{remark1}
It is possible to use the ideas of \cite{lp} to
modify the construction of the sequence of sets $S_k$ so that,
in addition to all the conclusions of
Theorems \ref{thm-main} and \ref{thm-main3}, the limiting set 
$S=\bigcap_{k=1}^\infty S_k$ is a {\em Salem set}.    
See \S\ref{fourier} for the definitions and more details.

\item It may be of greater interest that the correlation condition
(\ref{transverse-cond}) used to prove Theorems \ref{thm-main} and \ref{thm-main3} 
already implies that $S$
has positive {\em Fourier dimension}, provided that the $\epsilon$ in Theorem
\ref{thm-main3} is small enough ($\epsilon<\frac{1}{5}$ will suffice).
We hope to address this issue at length in a subsequent paper.

\item An argument due to David Preiss, included here in Subsection \ref{preiss-example},
shows that Theorem \ref{AP-diff} (hence also Theorem \ref{thm-main}(e)) 
cannot hold with $p=1$.  On the other hand, we do not know whether the range of 
$\epsilon$ or the exponents $p,q$ in  
Theorem \ref{thm-main3} is optimal.  

\end{enumerate}

\subsection{Motivation}
The motivation for the study of the maximal operators introduced in this article comes from
two different directions.  On the one hand, our maximal operators provide a 
one-dimensional analogue of higher dimensional Euclidean phenomena that have
been studied extensively in harmonic analysis in the context of hypersurfaces
and singular measures on $\rr^d$.  On the other hand, they arise naturally in the
consideration of density and differentiation theorems
for averages on sparse sets. We describe these below.     

\subsubsection{Analogues of maximal operators over submanifolds of $\rr^d$} 
\label{sphere-sub}
There is a vast literature on maximal and averaging operators over families of
lower-dimensional submanifolds of $\rr^d$. We will focus here on the case of 
maximal operators over rescaled copies of a single submanifold. 
Assuming that the submanifold in question is sufficiently smooth, 
the main issue turns out to be the curvature: roughly speaking, curved submanifolds admit
nontrivial maximal estimates, whereas flat submanifolds do not. 
A fundamental and representative positive result is the \textit{spherical
maximal theorem}, due to E.M. Stein \cite{stein-76} for $d\geq 3$ and Bourgain
\cite{bourg-86} for $d=2$.  We state it here for future reference.

\begin{theorem}\label{thm-spherical} (Stein \cite{stein-76}, Bourgain \cite{bourg-86})
Let $d\geq 2$. Recall the spherical maximal operator in $\rr^d$:
\begin{equation}\label{ave-e21}
\tilde{\mathfrak M}_{\mathbb S^{d-1}}f(x)=\sup_{r>0}\int_{\mathbb S^{d-1}}|f(x+ry)|d\sigma(y),
\end{equation}
where $\sigma$ is the normalized Lebesgue measure on the unit sphere
$\mathbb S^{d-1}$.  Then
$\tilde{\mathfrak M}_{\mathbb S^{d-1}}$ is bounded on $L^p(\rr^d)$ for 
$p>\textstyle{\frac{d}{d-1}}$,
and this range of $p$ is optimal.
\end{theorem}

It is not difficult to see that Theorem \ref{thm-spherical} fails in all dimensions 
$d\geq 2$ if the sphere $S^{d-1}$ is replaced by a polygonal line or surface. The latter, while 
still piecewise smooth, do not have any curvature. In intermediate cases, such as 
conical surfaces which are flat along the light rays but curved in other directions, maximal 
estimates may still be available but with weaker exponents. 
Many results of this type are known under varying
smoothness and curvature conditions.  We refer the reader to \cite{stein-ha},
\cite{CNSW}, \cite{phong-stein1}, \cite{phong-stein2} for an introduction to this area of 
research and further references.

Stein's proof of the spherical maximal theorem for $d\geq 3$, as well as most of the other results
just mentioned, exploit curvature via the decay of the Fourier transform of the surface measure 
on the manifold. In the case of the sphere, the Fourier transform decays like $|\xi|^{-\frac{d-1}{2}}$
at infinity; similar estimates hold for other convex hypersurfaces of codimension 1 with non-vanishing 
Gaussian curvature. The decay estimates are weaker for manifolds with flat directions, which 
is reflected in the range of exponents in maximal and averaging estimates.

The Fourier decay arguments are not sufficient for this purpose in dimension 2,
essentially because the decay is not fast enough. Instead, Bourgain's 
proof of the circular maximal theorem for $d=2$ relies more directly on the geometry involved. 
The relevant geometric information concerns intersections of pairs of thin annuli.
In an arrangement of many such annuli, most pairwise intersections are smaller by an order
of magnitude than each annulus itself; while larger intersections are possible, it can be shown that
they do not occur frequently. (For more details, see Subsection \ref{proofs-outline}.)

No similar theory has been developed so far in one dimension.  Indeed, it is not clear 
\textit{a priori} what such a theory might look like, given that the real line has no nontrivial 
lower-dimensional submanifolds. However, given any $\epsilon > 0$, there are
many singular measures on $\mathbb R$ supported on sets of Hausdorff
dimension $1 - \epsilon$. Viewing $\epsilon$ as an analogue of ``codimension'',
it is natural to ask whether by imposing additional structure on these sets that
would assume the role of curvature, one might obtain $L^p$ estimates similar to those in Theorem
\ref{thm-spherical} for the associated maximal operators and for
a range $p > p_{\epsilon}$, where $p_{\epsilon} \searrow 1$ as $\epsilon \rightarrow 0$.
Theorem \ref{thm-main3} provides an affirmative answer to this question. Theorem \ref{thm-main} may be interpreted as the limiting situation as $\epsilon \rightarrow 0$ (compare with Theorem \ref{thm-spherical} as $d \rightarrow \infty$) where the maximal range $(1, \infty]$ of $p$ is achieved for a single set $S$ of zero Lebesgue measure.       

It turns out that the proofs of Theorems \ref{thm-main} and \ref{thm-main3} do not use
any Fourier decay conditions.  Instead, our proofs rely on
geometrical arguments more akin to those in Bourgain's proof of the circular maximal
theorem. The right substitute for the Fourier decay (see (\ref{e-rdf}) below)  turns out to be
the correlation condition (\ref{transverse-cond}), providing the needed bound on 
the size of multiple intersections analogous to those arising in Bourgain's argument. Readers
familiar with the proof of Theorem \ref{thm-spherical} for $d=2$ or other similar results
will recognize the correlation condition as a bound on the integrand (interpreted
as the correlation function) in the expression for the $L^n$ norm of the dual
linearized and discretized maximal operator.
The goal will be to minimize this integrand whenever possible by
employing randomization arguments.

\subsubsection{Maximal averages via Fourier decay estimates}\label{fourier}
We now turn to the study of maximal operators $\tilde{\mathfrak M}$ 
defined as in (\ref{max-e1}) with $\mu$ obeying appropriate Fourier decay conditions.
It turns out that such conditions alone may often be substituted for the geometric assumptions 
of \S\ref{sphere-sub} (see e.g. \cite{duo-rdf86}, \cite{rdf} and the references therein). 
From this perspective, our result may be viewed as an extension of the following result by Rubio de Francia \cite{rdf}.  
We write $\widehat{\mu}(\xi)=\int e^{-2\pi i\xi x}d\mu(x)$.

\begin{theorem}(Rubio de Francia \cite{rdf})
Suppose that $\sigma$ is a compactly supported Borel measure on $\mathbb R^d$, $d \geq 1$, such that 
\begin{equation}\label{e-rdf}
|\widehat{\sigma}(\xi)| \leq C(1 + |\xi|)^{-a} 
\end{equation}
for some $a>\half$.  Then the maximal operator $\tilde{\mathfrak M}_{\sigma}$, defined as in (\ref{max-e1}) but with
$\mu$ replaced by $\sigma$, is bounded on $L^p(\mathbb R^d)$ for $p > (2a+1)/(2a)$.   
\label{thm-rdf}\end{theorem}

Theorem \ref{thm-rdf} implies Theorem \ref{thm-spherical} for $d \geq 3$, 
since then the surface measure $\sigma$ on the sphere obeys the above assumption with
$a = \frac{d-1}{2} > \frac{1}{2}$, but it fails to capture the circular maximal estimate in $\rr^2$ 
for which $a = \frac{1}{2}$ just misses the stated range. We also observe that the range 
of $p$ in Theorem \ref{thm-rdf} is independent of the dimension $d$; rather, it is given in 
terms of the Fourier decay exponent $a$. 
 
It is not possible for a singular measure $\sigma$ on $\rr$ to obey (\ref{e-rdf}) with 
$a>\half$ (see \cite{salem}).  In particular, Theorem \ref{thm-rdf} does not apply in 
this case.
On the other hand, there are many such measures obeying (\ref{e-rdf}) with a smaller exponent.  
Recall that the \textit{Fourier dimension} of a compact set $S\subset\rr$ is defined by
$$
\dim_{\mathbb F}(S)=\sup\{0\leq \beta\leq 1:\ \exists\hbox{ a probability measure $\sigma$ 
supported on $S$}$$
\vskip-8mm
$$
\hbox{ such that }|\widehat{\sigma}(\xi)|\leq C(1+|\xi|)^{-\beta/2}\hbox{ for all }\xi\in\rr \}.
$$
It is well known  that 
$\dim_{\mathbb F}(S)\leq\dim_{\mathbb H}(S)$ for all compact $S\subset \rr$,
and that the inequality is often strict (\cite{mattila-book}, \cite{falconer}).  However,
there are also many examples of sets with $\dim_{\mathbb F}(S)=\dim_{\mathbb H}(S)$,
see e.g.  \cite{salem},  \cite{kaufman},  \cite{bluhm-1}, \cite{bluhm-2}, \cite{kahane},  \cite{lp}.  
Such sets are known as \textit{Salem sets}.  
It is of interest to ask whether there is an analogue of Theorem \ref{thm-rdf}  
that might apply to singular measures supported on Salem sets and obeying (\ref{e-rdf}), 
possibly with additional assumptions.

We have already mentioned that our proofs of Theorems \ref{thm-main} and \ref{thm-main3} 
do not rely explicitly on any Fourier decay conditions, invoking the correlation condition
(\ref{transverse-cond}) instead. There is, however, an indirect connection which may be 
observed as follows. The condition (\ref{transverse-cond}) may be
viewed in terms similar to those of additive combinatorics:
if (\ref{e-rdf}) is interpreted as indicative of
\textit{linear uniformity} of $S$ (see \cite{lp}), then (\ref{transverse-cond}) 
is closer to higher-order uniformity conditions in additive
combinatorics (cf. \cite{gowers-szem}, \cite{gt-1}).  The latter are known to imply, and in fact
be strictly stronger than, Fourier-analytic estimates.  A continuous version of a standard calculation
in additive combinatorics shows that the correlation condition (\ref{transverse-cond})
implies Fourier decay estimates of the form (\ref{e-rdf}). In particular, it follows that
the sets we construct must have positive Fourier dimension provided that the $\epsilon$ 
in Theorem \ref{thm-main3} is sufficiently small ($\epsilon<\frac{1}{5}$ will do).
The rate of decay
obtained in this manner is not necessarily optimal.  In the case of the set $S$ of dimension $1$
given by Theorem \ref{thm-main}, our current methods yield (\ref{e-rdf}) for all $a<\frac{1}{8}$,
whereas the optimal range would be $a\leq\half$.  Note that the range of $p$ in Theorems
\ref{thm-main} and \ref{thm-main3} is better than what would follow from the numerology of
Theorem \ref{thm-rdf} with that value of $a$. We do not know
whether it is possible to prove maximal estimates such as those in Theorems \ref{thm-main}
or \ref{thm-main3} based solely on Fourier decay with $a<\half$.

With some additional effort, it is possible to construct sequences of
sets $S_k$ obeying all conditions of Theorems \ref{thm-main} and \ref{thm-main3}, 
respectively, such that $S$ is also a Salem set.  This can be done 
by adding the
appropriate Fourier-analytic conditions to Theorem \ref{random-mainprop} and proving
them along the same lines as in \cite[Section 6]{lp}. However, the optimal Fourier decay is not
needed in the proofs of any of our theorems.

\subsubsection{Density theorems and differentiation of integrals} \label{Aversa-Preiss-description}

There are several natural questions concerning density 
and differentiation theorems in one dimension 
that suggest the directions we pursue here.
We do not attempt to survey the vast literature on density theorems and differentiation
of integrals (see
\cite{bruckner}, \cite{de-guzman} for more information) and focus only on the
specific problems relevant to the present discussion.

The following question was raised and investigated by Preiss \cite{preiss} and Aversa-Preiss
\cite{AP1}, \cite{AP2}: to what extent can the Lebesgue density theorem be 
viewed as  ``canonical" in $\rr$, in the sense that any other density theorem that
takes into account the affine structure of the reals must follow from the Lebesgue
density theorem?  

Let us clarify and motivate this statement.  Consider a family $\cals$ of measurable 
subsets of $\rr$.  We will say that $\cals$ has the \textit{translational density property}
if for every measurable set $E\subset\rr$ we have
\begin{equation}\label{e-density}
\lim_{S\in\cals, \diam(S\cup\{0\})\to 0} \frac{|(x+S)\cap E|}{|S|}=1\hbox{ for a.e. }x\in E.
\end{equation}
Here and below, we use $x+S$ to denote the translated set 
$\{x+y:\  y\in S\}$. 

It follows from the Lebesgue density theorem that the collection of intervals $\{(-r,r): r>0\}$ has this property.  A moment's thought shows that collections such as $\{(0,r):\ r>0\}$
or $\{(\frac{r}{2},r):\ r>0\}$ also have it, simply because the intervals in question
occupy at least a fixed positive proportion of $(-r,r)$.

Consider now the family of intervals $\cals=\{I_k\}_{k=1}^\infty$, where 
$I_k=(\frac{k}{(k+1)!}, \frac{1}{k!})$.  We have $|I_k|=\frac{1}{(k+1)!}$ and 
$\diam(I_k\cup\{0\})=\frac{1}{k!}$, hence the last argument no longer applies.
In other words, the Lebesgue density theorem does not imply any density properties
of $\cals$.
Nonetheless, $\cals$ does have the translational density property, courtesy of
the \textit{hearts density theorem} of Preiss \cite{preiss} and Aversa-Preiss \cite{AP1}
(see also \cite{csornyei} for an alternative proof).

The collection $\cals$ in the last example does not generate an
\textit{affine invariant density system}:
if we let $I_k=(\frac{k}{(k+1)!}, \frac{1}{k!})$ as before and define $\cals'=\{rS_k:\ r>0,k\in\nn\}$,
then (\ref{e-density}) does not hold with $\cals$ replaced by $\cals'$.
(Note that the limit in (\ref{e-density}) is now being taken
over the two parameters $k$ and $r$.)  In fact, Aversa-Preiss prove in \cite{AP1} that no 
sequence of intervals $I_k$ can generate an affine invariant density 
system unless $\lim\inf_{k\to\infty}|I_k|/\diam(I_k \cup\{0\})
>0$, in which case the density property in question follows  from the Lebesgue theorem
as explained above.  

On the other hand, if we drop the requirement that $\cals$ be
a family of intervals, it is possible for $\cals$ to generate an affine invariant density system
independently of the Lebesgue density theorem.
This was announced by Aversa and Preiss in \cite{AP1} and proved in \cite{AP2}.

\begin{theorem}\label{AP-density} 
(Aversa-Preiss \cite{AP1}, \cite{AP2})
There is a sequence $\{S_k\}$ of compact sets of positive measure such that $|S_k|\to 0$ and:

(a)  $0$ is a Lebesgue density point for $\rr\setminus\bigcup S_k$, and in particular
we have
$$\lim_{n\to\infty} \frac{|S_k|}{\diam(S_k\cup\{0\})}=0;$$

(b) the family $\{rS_k:\ r>0,k\in\nn\}$ has the affine density property.
\end{theorem}

This essentially settles the matter for density theorems, except that constructing an
explicit example of sets $S_k$ as in Theorem \ref{AP-density} is still an open
problem.  (The Aversa-Preiss construction is probabilistic, and so is ours below.) However, the analogous question for $L^p$ differentiation theorems remained unanswered.

We will say that $\cals$ 
\textit{differentiates}\footnote{
This is a slight abuse of the standard terminology, which would require us to
say instead that the family $\{S+x\}_{x\in\rr}$ differentiates $L^p_\loc(\rr)$. 
}
$L^p_\loc (\rr)$ for some  $1\leq p\leq\infty$ if for every $f\in L^p_\loc (\rr)$ we have
\begin{equation}\label{e-diff}
\lim_{S\in\cals, \diam(S\cup\{0\})\to 0} \frac{1}{|S|}\int_{x+S}f(y)dy=f(x)\hbox{ for a.e. }x\in \mathbb R.
\end{equation}
For instance, the Lebesgue differentiation theorem states that the collection $\{(-r,r):\ r>0\}$
differentiates $L^1_\loc(\rr)$. Note that the differentiation property (\ref{e-diff}) implies the density property
(\ref{e-density}), by letting $f$ range over characteristic functions of measurable sets.  
There is no reason, though, why the converse implication should automatically hold.

While density theorems (such as Theorem \ref{AP-density} or the hearts density theorem 
mentioned earlier) can often be proved using purely geometrical considerations, differentiation
theorems tend to require additional analytic input, usually in the form of maximal estimates.
A well-known and representative example is provided by the Hardy-Littlewood
maximal theorem \cite{hardy-littlewood2}, \cite{wiener},
which easily implies the Lebesgue differentiation theorem.

Aversa and Preiss conjectured in  \cite{AP2} that their Theorem \ref{AP-density} could be
strengthened to an $L^2$  differentiation theorem. Specifically, there should exist a 
sequence of sets $\{ S_k \}$
as in Theorem \ref{AP-density} such that the family $\{rS_k: \ r>0,\  k\in\nn\}$ differentiates
$L^2(\rr)$ in the sense of (\ref{e-diff}).  Our maximal estimates in Theorem \ref{thm-main}
imply the Aversa-Preiss conjecture along the lines of
the standard Hardy-Littlewood argument.  Our Theorem \ref{AP-diff} is in fact stronger, 
providing a family of sparse sets which differentiates $L^p(\rr)$ for all $p>1$.  Preiss's
argument in Subsection \ref{preiss-example} shows that this range is optimal.


\subsection{Outline of the proofs}
\label{proofs-outline}

The intuition behind the construction in Theorems \ref{thm-main} and \ref{thm-main3} is,
roughly, that such results might hold if the sets $S_k$ (hence also $S$) are
sufficiently randomly distributed throughout the interval $[1,2]$.
Thus the challenge is first to find appropriate pseudorandomness conditions that
guarantee the boundedness of our maximal operators, then to actually construct a family of
sets obeying such conditions.
Our arguments are largely inspired by considerations from multidimensional
harmonic analysis, in particular by Bourgain's proof of the circular maximal theorem 
\cite{bourg-86}. The probabilistic construction of $S_k$ is somewhat similar to that
in \cite[Section 6]{lp}, but significantly more complicated.

The sets $S_k$ will be constructed by randomizing a Cantor-type iteration whose general 
features are described in Section \ref{sec-general-cantor}.
The main task is to prove that $S_k$ may be chosen so that the restricted maximal operator
$\calm$ obeys $L^p\to L^q$ bounds as indicated in Theorems \ref{thm-main} and \ref{thm-main3}.
Once such bounds are available, the corresponding estimates on $\tilde\calm^a$ are 
obtained through the scaling analysis in Section \ref{sec-scales}, and the estimates on
$\mathfrak M$ and $\tilde{\mathfrak M}$ follow automatically provided that the limiting measure
$\mu$ exists.  The differentiation theorems (Theorem \ref{AP-diff} and \ref{thm-main3} 
(\ref{lower-dim-diff}))  are deduced in Section
\ref{AP-section}.

Our analysis of $\calm$ begins with several preliminary reductions carried out in
Section \ref{subsec-lin-disc}. Consider the auxiliary restricted maximal operators
\begin{equation}\label{ferret-e200}
\calm_k f(x)=\sup_{1<t<2} \Big| \int f(x+ty) \sigma_k(y)dy\Big| \ ,
\end{equation}
where
$\sigma_k=\phi_{k+1}-\phi_k$, and $\phi_k$ is the normalized Lebesgue density on $S_k$.
The bulk of the work is to prove appropriate $L^p\to L^q$ bounds on $\calm_k$;
this implies the bounds on $\calm$ upon summation in $k$.
We further replace each $\mathcal M_k$ by its discretized and linearized
counterpart $\Phi_k$, the discretization being in the space of affine transformations. 
By duality and interpolation, the desired $L^p$ estimates on $\Phi_k$ will follow 
from restricted strong-type estimates on the ``dual" operator $\Phi_k^*$.
These reductions are all well known in the harmonic analysis literature, 
even though the details are specific to the problem at hand.  We will follow
the approach of \cite{bourg-86}, \cite{schlag2}, and especially \cite{schlag-thesis}
with relatively minor modifications.

The main part of our argument is to prove the required estimates on $\Phi_k^*$.
Before we describe it in more detail, 
we pause for a moment to recall the analogous part of Bourgain's proof of the circular 
maximal theorem in \cite{bourg-86}.  In his context, the dual 
linearized operator $\Phi_k^*$ 
acting on characteristic functions $g=\one_\Omega$ has the form
$$
\Phi^*_k g(z) = \int_{\Omega} \frac{1}{|E_{x,k}|} \one_{E_{x,k}}(z) dx,
$$
where each $E_{x,k}$ is an annulus of thickness $2^{-k}$ and radius $r_x$ centered at $x$.
(Our setup is actually somewhat different in that our operators have built-in
cancellation: the maximal operators in (\ref{ferret-e200}) and their dual linearized counterparts 
$\Phi_k^*$ involve integration with density $\sigma_k$ rather than $\phi_k$. This distinction
will be important later on, for example we will need bounds on $\Phi_k^*$ that are summable
and not just bounded in $k$,
but we omit it for the purpose of the present discussion.)

The main task is to prove that $\Phi^*_k$ is bounded on $L^{p'}$ with $1\leq p'<2$.
The $L^1$ bound is trivial, and the proof would be complete if we could prove 
a similar bound on $L^2$.  We have
\begin{equation}\label{ferret-e3}
\begin{split}
\| \Phi^*_k g\|_2^2 
&= \int \int_{\Omega\times\Omega} \frac{1}{|E_{x,k}|\, |E_{y,k}| } \one_{E_{x,k}}(z) \one_{E_{y,k}}(z) dx\,dy\,dz\\
&= \int_{\Omega\times\Omega} \frac{1}{|E_{x,k}|\, |E_{y,k}| } |E_{x,k} \cap E_{y,k} | dx\,dy\, .\\
\end{split}
\end{equation}
If we had
\begin{equation}\label{ferret-e1}
|E_{x,k} \cap E_{y,k} | \leq  C_k |E_{x,k}|\, |E_{y,k}|,
\end{equation}
the needed $L^2$ bound would follow.  Unfortunately, (\ref{ferret-e1}) need not always hold. 
Specifically, if the two annuli are ``internally tangent" in a clamshell configuration, 
the area of the intersection on the left side of (\ref{ferret-e1}) can easily 
be much larger than $ |E_{x,k}|\, |E_{y,k}|\approx 2^{-2k}$.

Bourgain's key observation is that geometric considerations put a strict
limit on the size of the set of pairs $(x, y)\in\Omega^2$ for 
which the associated annuli are internally tangent. 
The remaining generic (or \textit{transverse})
intersections do have reduced area.  This allows
him to split the region of integration in two parts.  One of them involves only transverse 
intersections, hence there is a good $L^2$ bound as described above. The other part 
covers the internal tangencies; here the $L^2$ estimates are poor, but on the other hand
the $L^1$ estimates can be improved thanks to the small size of the region.  An interpolation
argument completes the proof.

Let us now try to apply a similar argument in our setting, with $p$ restricted for now
to the range $(2,\infty]$ so that $1\leq p'< 2$.  As in Bourgain's proof, the restricted weak $L^2$ bounds
for $\Phi^*_k$ are based on estimates on the size of the double intersections $(x+rS_k)\cap (y+sS_k)$ via
the appropriate analogue of (\ref{ferret-e3}).  
While we still expect that generic double intersections should be significantly smaller than
$|S_k|$, the task of actually estimating them turns out to be quite hard, due to the interplay
between the different scales in the Cantor iteration.  

To illustrate the problem, we consider the following somewhat simplified setting.  Suppose that
the $k$-th iteration $S_k$ of the Cantor set is given.  Subdivide each of the intervals of $S_k$
into $N_{k+1}$ subintervals of equal length, and choose $N_{k+1}^{1-\epsilon}$ of them within each
interval of $S_{k}$.  Given the translation and dilation parameters $x,y,r,s$, what is
the size of $(x+rS_{k+1})\cap (y+sS_{k+1})$?

We write the intersection in question as a union of sets 
\begin{equation}\label{ferret-e6}
(x+r(I\cap S_{k+1}))\cap (y+s(J\cap S_{k+1})),
\end{equation}
where $I$ and $J$ range over all intervals of $S_{k}$.  If $I\neq J$, the $S_{k+1}$-subintervals 
of $I$ and $J$ were chosen independently, hence (\ref{ferret-e6}) is expected to consist of
about $N_{k+1}^{1-2\epsilon}$ such subintervals.  In other words, we expect a substantial gain compared to the
size of each of the sets $I\cap S_{k+1}$ and $J\cap S_{k+1}$. On the other hand, this argument
does not apply to (\ref{ferret-e6}) with $I=J$, where we cannot expect to do better than
the trivial bound.

Following Bourgain, we will refer to the first type of intersections ((\ref{ferret-e6}) 
with $I\neq J$) as \textit{transverse intersections}, and to the second type (with $I=J$) 
as \textit{internal tangencies}. At each step $k$ of the iteration, a typical intersection 
of two affine copies of $S_k$ will consist of both transverse intersections and internal
tangencies.  If there are few internal tangencies, we expect an overall gain as described 
above. If on the other hand there are many internal tangencies, a geometrical argument 
shows that both $|x-y|$ and $|r-s|$ must be small relative to the current scale, 
which in turn restricts the relevant 
domain of $(x,y)$.  As in Bourgain's proof, we are able to combine these two observations 
to prove the desired maximal bound. To extend our bounds to $1<p\leq 2$ (hence 
$2\leq p'<\infty$), we consider the $L^n$ analogues 
of (\ref{ferret-e3}) which involve $n$-fold intersections of affine copies of $S_k$.  

The precise statement of the intersection bound we need is given by the \textit{transverse
correlation condition} (\ref{transverse-cond}). 
In Section \ref{section-maximal} we formulate the correlation condition and prove that it 
does indeed guarantee a restricted strong type estimate on $\Phi^*_k$. The correlation 
condition (\ref{transverse-cond}) may be viewed as a multiscale analogue of the higher order
uniformity conditions in additive combinatorics, see e.g. \cite{gowers-szem}, \cite{gt-1}.
It appears to be stronger than the pseudorandomness conditions considered so far in the
literature, due to the inclusion of the dilation factor and the interplay between different scales.

The random construction of sets $S_k$ obeying our correlation condition is carried
out in Section \ref{section-random}. This part of the proof contains the bulk of the technical
work and requires the full strength of our probabilistic machinery. The procedure
is based on a Cantor-type iteration as described in Section \ref{sec-general-cantor}, but 
now each $S_k$ is randomized subject to appropriate constraints on the parameters.
We then use
large deviation inequalities (specifically, Bernstein's inequality and Azuma's inequality)
to prove that at each step of the construction there is a positive probability that 
the set $S_k$ has the required properties including (\ref{transverse-cond}).
Finally, in Section \ref{sec-choose} we fix the parameters of the random construction and complete
the proof of our restricted maximal estimates.


\subsection{Acknowledgement}

We are grateful to Vincenzo Aversa and Nir Lev for bringing the Aversa-Preiss
conjecture to our attention, and to Nir Lev for a further
introduction to questions regarding density and differentiation theorems for sparse sets.

We thank Andreas Seeger for showing us the argument (together with relevant background \cite{nagel-stein-wainger78}, \cite{duo-rdf86}, \cite{christ88-unpublished}) that $\tilde{\mathcal M}$ and $\tilde{\mathfrak M}$ are bounded on all $L^p(\mathbb R)$ with $1<p < 2$ for which the restricted operator $\mathcal M$ is $L^p$-bounded. He also referred us to \cite{gsw99-pams}, the techniques of which combined with his proof led to the $L^p-L^q$ generalizations for $\tilde{\mathcal M}^a$ and $\tilde{\mathfrak M}^a$ stated in Theorems \ref{thm-main} (\ref{thm-main-unrestricted1}) and \ref{thm-main3} (\ref{thm-main3-unrestricted1}). We include this argument with his kind permission in Section \ref{sec-scales} (see Proposition \ref{p-less-2}).   

The proof that singular measures cannot differentiate $L^1(\rr)$ (see Subsection 
\ref{preiss-example}) is due to David Preiss.  We are indebted to him for allowing us
to include his argument in this article.

We thank the anonymous referees for their valuable suggestions.

This project was started at the Fields Institute in Toronto, where both authors
were participating in the Winter 2008 program on Harmonic Analysis. 
Both authors are supported by NSERC Discovery Grants.

\section{The general Cantor-type construction} \label{sec-general-cantor}

\subsection{Basic construction of the sets $\{S_k\}$} \label{subsec-basic-constr}
All the nested sequences of sets $\{S_k : k \geq 1\}$ considered in this paper
will be obtained using a Cantor-type construction, whose basic features
we now describe. The parameters in the construction are the following:  
\begin{enumerate}[(a)]
\item a nondecreasing sequence of positive integers $\{N_k : k \geq 1 \}$ with $\delta_k^{-1} = N_1 N_2 \cdots N_k$,  
\item certain sequences of $0$-s and $1$-s, 
\begin{align*} 
&\mathbf X_k := \{X_k({\mathbf i}) : \mathbf i = (i_1, \cdots, i_k),
1 \leq i_j \leq N_j, \, 1 \leq j \leq k \}, \text{ and } 
\\ & \mathbf Y_k := \{Y_k({\mathbf i}) : \mathbf i = (i_1, \cdots, i_k),
1 \leq i_j \leq N_j, \, 1 \leq j \leq k \} \; 
\text{ satisfying } \\ 
&X_{k+1}(\overline{\mathbf i}) := X_k({\mathbf i}) Y_{k+1}(\overline{\mathbf i}), \; \text{ where }
\overline{\mathbf i} = (i_1, \cdots, i_{k+1}). \end{align*}
\end{enumerate}
Given these quantities, we denote
\[ \mathbb I = \mathbb I_k =
\{\mathbf i = (i_1, \cdots, i_k) \in \mathbb Z^k : 1 \leq i_r \leq N_r, \;
1 \leq r \leq k \}, \]
and for every multi-index $\mathbf i = (i_1, \cdots, i_k) \in \mathbb I_k$,
\begin{align}
\alpha({\mathbf i}) &= \alpha_k(\mathbf i)
= 1 + \frac{i_1 - 1}{N_1} + \frac{i_2-1}{N_1 N_2} + \cdots + \frac{i_k - 1}{N_1 \cdots N_k},
\label{alphai-def}
\\ I_k({\mathbf i}) &= \left[\alpha(\mathbf i), \alpha(\mathbf i) + \delta_k \right],
\quad \text{so that}
\quad I_k(\mathbf i) = \bigcup_{i_{k+1}=1}^{N_{k+1}} I_{k+1}(\overline{\mathbf i}).
\end{align}
The argument $k$ will sometimes be suppressed if it is clear from the context. 
We also set for
$k \geq 1$, 
\[M_k = N_1 N_2 \cdots N_k (\text{so that } \delta_k = M_k^{-1}), \quad P_k =
\#\{\mathbf i : X_k(\mathbf i) = 1 \}. \]

The construction proceeds as follows. Starting with the interval
$[1,2]$ equipped with the Lebesgue measure, we subdivide it into $N_1$ intervals
$\{I_1(i): 1 \leq i \leq N_1 \}$ of equal length. We choose the $P_1$ intervals $I_1(i_1)$
for which $X_1(i_1) = 1$ and assign weight $P_1^{-1}$ to each one. At the second
step, we subdivide each of the intervals chosen at the first step into $N_2$
subintervals of equal length $\delta_2$, and choose from $I_1({i_1})$ the subintervals
$\{I_2({\mathbf i}), \; \mathbf i = (i_1, i_2)\}$ such that $Y_2({\mathbf i}) = 1$.
The total number of chosen subintervals at this stage is therefore $P_2$, and each one
is assigned a weight of $P_2^{-1}$. We continue to iterate the procedure, selecting
at the $(k+1)$-th stage subintervals of the intervals chosen at the $k$-th step,
based on the sequences $Y_{k+1}(\mathbf i)$. In summary, the sets $S_k$
are chosen according to the scheme 
\[S_0 = [1,2], \qquad 
S_k = \bigcup_{\mathbf i} \left\{I_k(\mathbf i): X_k(\mathbf i) = 1 \right\}.
\] 
We will always assume that $|S_k| \searrow 0$, i.e., $P_k \delta_k \rightarrow 0$.   


\subsection{The Hausdorff dimension of the set $S$}

We now investigate the Hausdorff dimension of the resulting set
$S = \bigcap_{k=1}^{\infty} S_k$ as a function of the parameters of the construction. 

\begin{lemma}
Let $\dim_{\mathbb H}(S)$ denote the Hausdorff dimension of $S$ constructed above. Then  
\begin{enumerate}[(a)]
\item $\dim_{\mathbb H}(S) \leq \liminf_{k \rightarrow \infty} \log(P_k)/\log(M_k)$. 
\item   
$\dim_{\mathbb H}(S) \geq s_0 := \liminf_{k \rightarrow \infty} \log(P_{k}/N_k)/\log(M_{k-1})$. 
\end{enumerate} \label{Hd-lemma}
\end{lemma}  
\begin{proof}
Part (a) follows immediately from Proposition 4.1 in \cite{falconer-book}.
For the proof of part (b), we follow an approach similar to Example 4.6 in
\cite{falconer-book}. The goal is to define a measure $\nu$ on $S$ such that for
any $s < s_0$, there exists a constant $C_s < \infty$ satisfying 
\begin{equation}
\nu(J) \leq C_s |J|^s \quad \text{for all intervals $J \subset \mathbb R$}.
\label{ball} \end{equation}   
The desired conclusion would then follow from Frostman's lemma (see e.g. Proposition 8.2
in \cite{wolff-lectures}).

In order to define $\nu$, we follow a standard procedure due to Caratheodory (see
Chapter 4, \cite{mattila-book}). 
Let $\mathcal B = \bigcup \mathcal B_k$, where $\mathcal B_0 = [1,2]$ and $\mathcal
B_k$ for $k \geq 1$ is the family of all basic intervals of $S_k$, i.e., intervals
of the form $\{I_k(\mathbf i) : X_k(\mathbf i)=1 \}$. For each interval
$I \in \mathcal B$, we define its weight $w(I)$ to be
\begin{equation}
w([1,2]) = 1, \qquad w(I) = P_{k}^{-1} \text{ if } I \in \mathcal B_{k},
\label{def-w} 
\end{equation} 
and a family of outer measures $\nu_k$ as follows, 
\begin{equation} 
\nu_k(F) := \inf \Bigl\{\sum_{i=1}^{\infty} w(J_i) : F \subseteq \bigcup_{i=1}^{\infty}
J_i,\; |J_i| \leq \delta_k, \; J_i \in \mathcal B \Bigr\}
\label{def-sigmak} \end{equation}
for all $F \subseteq S$. It is easy to see that $\nu_k$ is monotonic, so we can
define $\nu$ by
\begin{equation}
\nu(F) = \lim_{k \rightarrow \infty} \nu_k(F) = \sup_{k \geq 1}\nu_k(F).
\label{def-sigma} \end{equation}
Then $\nu$ is a non-negative regular Borel measure of unit mass on subsets of $S$ (Theorem 4.2, \cite{mattila-book}). 

To prove (\ref{ball}), let $J$ be an interval with $0 < |J| \leq \delta_1$. Given such a $J$, there is a unique $k = k(J)$ such that $\delta_{k+1} \leq |J| < \delta_k$. The number of basic intervals of $S_{k+1}$ that intersect $J$ is 
\begin{enumerate}[(i)]
\item at most $2N_{k+1}$ since $J$ intersects at most two intervals of $S_k$, and 
\item at most $|J|/\delta_{k+1}$, since the basic intervals comprising $S_{k+1}$ are of length $\delta_{k+1}$ and have disjoint interiors.  
\end{enumerate}      
It therefore follows from the definitions (\ref{def-w}) and (\ref{def-sigmak}) that  
\begin{align*}
\nu_{k+1}(J) &\leq P_{k+1}^{-1} \min \left[ 2N_{k+1}, \frac{|J|}{\delta_{k+1}}\right] \\
&\leq P_{k+1}^{-1} (2N_{k+1})^{1-s} \left( \frac{|J|}{\delta_{k+1}}\right)^{s}
\text{ for all } 0 \leq s \leq 1, \\ \text{i.e.,} \quad
\frac{\nu_{k+1}(J)}{|J|^s} &\leq \frac{2^{1-s} N_{k+1}^{1-s}}{P_{k+1} \delta_{k+1}^s}.   
\end{align*}
Letting $k \rightarrow \infty$ and recalling (\ref{def-sigma}), we find that the right hand side of the inequality above is bounded above by a constant provided that $s < s_0$. This completes the proof.
\end{proof}
{{\em Remark:}} In our applications, 
the sequences $\mathbf X_k, \mathbf Y_k$ of $0$-s and $1$-s will be chosen according to a random mechanism, to be described in Section \ref{section-random}. We will see in these instances that the upper and the lower bounds given by Lemma \ref{Hd-lemma} coincide, providing an exact value of the Hausdorff dimension.


\subsection{A limiting measure} 
\label{sub-measure-limit}

Although most of our results can be stated purely in terms of the maximal operators $\calm$ 
associated with the sequence of sets $\{S_k: k \geq 1\}$, it is often of interest to
know whether the normalized Lebesgue measures $\phi_k = \mathbf 1_{S_k}/|S_k|$ have
a nontrivial weak-$*$ limit $\mu$.  In this case, the maximal operator $\mathfrak M$ associated
with $\mu$ is bounded by $\mathcal M$. If each
interval in $S_k$ contains the same number of subintervals of $S_{k+1}$, it is easy
to see that $\mu$ exists and is identical to the measure $\nu$ defined in the
last subsection. Below we provide a sufficient condition for the existence of the
weak-$\ast$ limiting measure under a 
slightly weaker assumption that will be verified for certain constructions in the sequel.

\begin{lemma} \label{limit-measure-lemma}
Suppose that the distribution of the chosen subintervals $\{I_{\mathbf i}(k) : X_k(\mathbf i) = 1 \}$ within $S_{k-1}$ is approximately uniform in the following sense: 
\begin{equation} \sup_{k': k' \geq k} \sum_{\begin{subarray}{c}\mathbf i \\
X_k({\mathbf i}) = 1 \end{subarray}} \left| \int_{I_k({\mathbf i})} \left[\phi_{k'} - \phi_k \right](x) \, dx  \right| \rightarrow 0 \quad \text{ as } \quad k \rightarrow \infty.  \label{limit-measure} \end{equation}  
Then there exists a probability measure $\mu$ on $[1,2]$ such that $\phi_k \rightarrow \mu$ in the weak$-{\ast}$ topology, i.e., for all $f \in C[1,2]$
\[ \int f \phi_k \rightarrow \int fd\mu  \quad \text{ as } \quad k \rightarrow \infty. \] 
\end{lemma} 
\begin{proof}
It suffices to show that $\lim_{k \rightarrow \infty} \int f \phi_k$ exists for all continuous functions $f$ on $[1,2]$, i.e., that the sequence $\{\int f \phi_k : k \geq 1\}$ is Cauchy. Since $f$ is uniformly continuous, given any $\epsilon > 0$ there exists $\delta > 0$ such that \begin{equation} \label{uniform-cty}|f(x) - f(y)|< \frac{\epsilon}{4}\quad \text{ whenever } \quad |x-y|< \delta. \end{equation}
Fix $K \geq 1$ such that $\delta_K < \delta$ and 
\begin{equation} \sup_{k': k' \geq k} \sum_{\begin{subarray}{c}\mathbf i \\ 
X_k({\mathbf i}) = 1 \end{subarray}} \left| \int_{I_k({\mathbf i})}
 \left[\phi_{k'} - \phi_k \right](x) \, dx  \right| < \frac{\epsilon}{2 ||f||_{\infty}}
 \text{ for all } k \geq K.
\label{uniform} \end{equation} 
Let $\{x_k({\mathbf i}) : X_k({\mathbf i}) = 1\}$ be a collection of points in $[1,2]$ 
such that $x_k({\mathbf i}) \in I_k({\mathbf i})$. Then for all $k' \geq k \geq K$,    
\begin{align*}
&\left|\int f(x) \bigl( \phi_{k'}(x) - \phi_k(x)\bigr) \, dx \right| \\
&\leq \sum_{\begin{subarray}{c} \mathbf i \\ X_k({\mathbf i})=1 \end{subarray}} 
\int_{I_k({\mathbf i})} \Big| \left[f(x) - f(x_k({\mathbf i})) \right] 
\left(\phi_{k'} - \phi_k \right)(x)\Big| \, dx \\ 
& \hskip0.5in + \sum_{\begin{subarray}{c} \mathbf i \\ 
X_k({\mathbf i})=1 \end{subarray}} |
f(x_k({\mathbf i}))| \left| \int_{I_k({\mathbf i})} \left(\phi_{k'} - \phi_k \right)(x) \, dx \right| \\ 
&\leq \frac{\epsilon}{4} \int_{S_k} (\phi_{k'} + \phi_k)(x)\, dx 
+\|f\|_{\infty} \sum_{\begin{subarray}{c} \mathbf i \\ X_k({\mathbf i}) = 1 \end{subarray}} \left| \int_{I_k({\mathbf i})} \left(\phi_{k'} - \phi_k \right)(x) \, dx \right| \\ 
&\leq 2\frac{\epsilon}{4} + \frac{\epsilon}{2} = \epsilon,
\end{align*} 
where we have used (\ref{uniform-cty}) and (\ref{uniform}) at the last two steps.
\end{proof}


\subsection{Internal tangencies and transverse intersections} \label{sec-int-tgt}

An important ingredient in the derivation of the maximal estimates is the behavior 
of the intersections of a fixed number of affine copies of $S_k$. Obviously, 
much of our analysis will depend on the specific structure of $\{S_k\}$, 
which will be described in detail in Section 
\ref{section-random}. However, we also need certain general properties of the 
$n$-fold intersections of affine copies of sets $S_k$ constructed as in 
Subsection \ref{subsec-basic-constr}. The relevant results of this type 
are collected in this subsection. 

Fix $k \geq 1$, $r,s \in [1,2]$ and points $x,y$ in a fixed compact set, say $[-4,0]$ (the reason for this choice will be made clear in the next section). We will be interested in 
classifying pairs of multi-indices $(\mathbf i, \mathbf j)\in{\mathbb I}_k^2$ such that   
\begin{equation}\label{int-e1}
(x + r I_k({\mathbf i})) \cap (y + s I_k({\mathbf j})) \ne \emptyset.
\end{equation}

We will need to distinguish between the situations where $|\alpha_k(\mathbf i) - 
\alpha_k(\mathbf j)|$ is ``small'' or ``large''. The first case will be referred
to as an {\it internal tangency} and the second as a {\it transverse intersection}.
In view of subsequent applications, we give the precise definitions of these notions
for general $n$-fold intersections of intervals. However, the main ideas are already contained
in the case $n=2$, which we encourage the reader to investigate first.    

\begin{definition}
For integers $k \geq 1, n \geq 2$ and any set \[ \mathbf A_n = \{(c_{\ell},r_{\ell}) : 1 \leq \ell \leq n,\; c_{\ell} \in [-4,0],  r_{\ell} \in [1,2]  \} \] of $n$ translation-dilation pairs, we define a set $\mathbb F= \mathbb F[n, k; \mathbf A_n]$ and $n$ projection maps $\pi_{\ell}= \pi_{\ell}[n, k; \mathbf A_n](\mathbf i_1, \cdots, \mathbf i_n) : \mathbb F \rightarrow \mathbb I_k$ as follows,  
\begin{equation} 
\mathbb F  = \Bigl\{(\mathbf i_1, \cdots, \mathbf i_n) \in \mathbb I_k^n: 
\bigcap_{\ell = 1}^n \bigl(c_{\ell} + r_{\ell} I_k(\mathbf i_{\ell})\bigr) \ne \emptyset \Bigr\},
\label{def-F} \end{equation}
\[ \pi_{\ell}(\mathbf i_1, \cdots, \mathbf i_n)  = \mathbf i_{\ell}. \]  
\end{definition}

\noindent{\em{Remarks:}} 
\begin{enumerate}[1.] 
\item We emphasize that $\mathbb F$ consists of {\em{all}} tuples $(\mathbf i_1, \cdots, 
\mathbf i_n) \in \mathbb I_k^n$ such that (\ref{def-F}) holds, regardless of the actual 
choice of the sets $S_k$. Thus $\mathbb F$ depends only on the parameters $n,k,N_j$, 
and on the choice of $\mathbf A_n$.

\item Eventually, our translation and dilation parameters $c_{\ell}$ and $r_{\ell}$ 
will be chosen from discrete subsets ${\cal C}, {\cal R}$ of the respective spaces 
$[-4,0]$ and $[1,2]$.  Then the total number of possible collections $\mathbb F$ 
cannot exceed $|{\cal C}|^n|{\cal R}|^n$, again irrespective of the choice of the 
sets $S_k$.    
\end{enumerate}

The next lemma is an easy observation concerning the ``almost injectivity'' of the 
projections $\pi_{\ell}$.

\begin{lemma}
For any $1 \leq \ell \leq n$ and any fixed choice of multi-indices $(\mathbf i_{\ell'} : 1 \leq \ell' \leq n, \ell' \ne \ell) \in \mathbb I_k^{n-1}$,
\begin{equation} 
\max_{{\mathbf i_{\ell}} :\ (\mathbf i_1, \cdots, \mathbf i_n) \in \mathbb F}
\alpha_k({\mathbf i_{\ell}}) 
-\min_{{\mathbf i_{\ell}} :\ (\mathbf i_1, \cdots, \mathbf i_n) \in \mathbb F}
\alpha_k({\mathbf i_{\ell}}) 
\leq 4 \delta_k. \label{max-min} \end{equation}
In particular, for any $1 \leq \ell \leq n$ the map $\pi_{\ell}$ is at most four-to-one, i.e.,  
\begin{equation} \sup_{\mathbf i_{\ell} \in \mathbb I_k} \#\bigl(\pi_{\ell}^{-1}(\mathbf i_{\ell})\bigr) \leq 4. \label{421} \end{equation}      \label{prelim-lemma}
\end{lemma} 
\begin{proof}
The second part of the lemma follows from the first. The inequality in (\ref{max-min}) is essentially a fact about two-fold intersections. Fix $\ell' \ne \ell$ and $(\mathbf i_1, \cdots, \mathbf i_n) \in \mathbb F$, so that by definition (\ref{def-F})
\[ (x_{\ell'} \cap r_{\ell'} I_k(\mathbf i_{\ell'})) \cap (x_{\ell} \cap r_{\ell} I_k(\mathbf i_{\ell})) \ne \emptyset. \]  Since $r_{\ell}, r_{\ell'}  \in [1,2]$ any interval of the form $x_{\ell'} + r_{\ell'} I_k(\mathbf i_{\ell'})$ can intersect at most four intervals of the form $x_{\ell} + r_{\ell} I_k(\mathbf i_{\ell})$ and these intervals must necessarily be adjacent. The claim follows.   
\end{proof} 
\begin{corollary}\label{prelim-cor1}
There exists a decomposition of $\mathbb F$ into at most $4^{n-1}$ subsets so that all the projection maps $\pi_{\ell}$ restricted to each subset are injective.  
\end{corollary}
\begin{proof}
The proof is an easy induction on $n$ combined with (\ref{421}), and is left to the interested reader. 
\end{proof}  
The lemma above motivates the following definition. Setting $\mathbf i_{\ell} =
(\mathbf i_{\ell}', i_{\ell k}) \in \mathbb I_{k-1} \times \{1,2, \cdots, N_k\}$,
we find that each $\mathbb F = \mathbb F[n,k;\mathbf A_n]$ decomposes as \begin{align*} \mathbb F &= \mathbb F_{\text{int}} \cup \mathbb F_{\text{tr}}, \quad \text{ where } \quad  
\mathbb F_{\text{int}} := \bigcup_{1 \leq \ell \neq \ell' \leq n} \mathbb F_{\text{int}}(\ell, \ell'), \text{ with } \\ \mathbb F_{\text{int}}(\ell, \ell') &:= \{(\mathbf i_1, \cdots, \mathbf i_n) \in \mathbb F : \mathbf i_{\ell}' = \mathbf i'_{\ell'},\; |i_{\ell k} - i_{\ell' k}| \leq 4 \}, \text{ and } \\
\mathbb F_{\text{tr}} &:= \mathbb F \setminus \mathbb F_{\text{int}}.
\end{align*} 
Note that in view of (\ref{alphai-def}), 
\begin{equation} (\mathbf i_1, \cdots, \mathbf i_{n}) \in \mathbb F_{\text{int}}(\ell, \ell') \quad \text{ implies }  \quad |\alpha_k(\mathbf i_{\ell}) - \alpha_k(\mathbf i_{\ell'})| \leq 4 \delta_k. 
\label{int-cond} \end{equation}  
\begin{definition}
The collections $\mathbb F_{\text{int}}$ and $\mathbb F_{\text{tr}}$, which depend only on $n, k, \{N_j : 1 \leq j \leq k \}$ and $\mathbf A_n = \{(c_{\ell},r_{\ell}) : 1 \leq \ell \leq n \}$, are referred to as the classes of internal tangencies and transverse intersections respectively.
\end{definition}
A large number of internal tangencies forces a relation between the translation (and hence dilation) parameters, in a sense made precise by the next lemma.  (A similar observation was made 
by Aversa and Preiss in \cite{AP2}.)

\begin{lemma} \label{internal-tangency}
Suppose $\#(\mathbb F_{\text{int}}) \geq L$. Then \[\min \{|c_{\ell}-c_{\ell'}| : 1 \leq \ell \neq \ell' \leq n \} \leq \min\left(4, 80n(n-1)/L \right). \] 
\end{lemma} 
\begin{proof}
Since the translation parameters all lie in $[-4,0]$, we may assume without loss of generality that $L > 20n(n-1)$. Using the definition of $\mathbb F_{\text{int}}$ and pigeonholing we can find indices $\ell \neq \ell'$ such that $\#(\mathbb F_{\text{int}}(\ell, \ell')) \geq \frac{2L}{n(n-1)}$. By Lemma \ref{prelim-lemma}, there exists a further subset $\mathbb F^{\ast} \subseteq \mathbb F_{\text{int}}(\ell, \ell')$ such that \begin{equation} \#(\mathbb F^{\ast})\geq \frac{1}{4}\#(\mathbb F_{\text{int}}(\ell, \ell')) \geq \frac{L}{2n(n-1)}, \text{ and } \mathbf \pi_{\ell}\Bigl|_{\mathbb F^{\ast}} \text{ is injective }. \label{Fstar} \end{equation}

Let $(\mathbf i_1, \cdots \mathbf i_n), (\mathbf j_1, \cdots, \mathbf j_n) \in \mathbb F$. Since $r_{\ell}, r_{\ell'} \in [1,2]$, it follows from the definition (\ref{def-F}) that  
\begin{equation} \begin{aligned} 
\left|\bigl(c_{\ell} + r_{\ell} \alpha_k(\mathbf i_{\ell}) \bigr) - 
\bigl(c_{\ell'} + r_{\ell'} \alpha_k(\mathbf i_{\ell'}) \bigr) \right| 
\leq \max(r_{\ell},r_{\ell'}) \delta_k &\leq 2\delta_k,  \\ \text{ and similarly } \left|\bigl(c_{\ell} + r_{\ell} \alpha_k(\mathbf j_{\ell}) \bigr) - 
\bigl( c_{\ell'} + r_{\ell'} \alpha_k(\mathbf j_{\ell'}) \bigr) \right| &\leq 2 \delta_k.  \end{aligned} \label{est1} \end{equation}
If further $(\mathbf i_1, \cdots, \mathbf i_n), (\mathbf j_1, \cdots, \mathbf j_n) \in \mathbb F_{\text{int}}(\ell, \ell')$, then (\ref{est1}) and (\ref{int-cond}) imply that 
\begin{align*} 
\bigl|(c_{\ell}- c_{\ell'}) + (r_{\ell} - r_{\ell'})\alpha_k(\mathbf i_{\ell}) \bigr| 
\leq 2 \delta_k + r_{\ell'} |\alpha_k(\mathbf i_{\ell'}) - \alpha_k(\mathbf i_{\ell})| 
&\leq 10 \delta_k, \\ 
\bigl|(c_{\ell}- c_{\ell'}) + (r_{\ell} - r_{\ell'})\alpha_k(\mathbf j_{\ell}) \bigr| 
\leq 2 \delta_k + r_{\ell'} |\alpha_k(\mathbf j_{\ell'}) - \alpha_k(\mathbf j_{\ell})| 
&\leq 10 \delta_k.   \end{align*}
Eliminating $(r_{\ell}- r_{\ell'})$ from the two inequalities above we find that 
\[ |c_{\ell}- c_{\ell'}| \bigl|\alpha_k(\mathbf i_{\ell}) - \alpha_k(\mathbf j_{\ell}) \bigr| \leq 40 \delta_k. \] 
If we now choose $(\mathbf i_1, \cdots, \mathbf i_{n}), (\mathbf j_1, \cdots, \mathbf j_n) \in \mathbb F^{\ast}$ so that $\bigl|\alpha_k(\mathbf i_{\ell}) - \alpha_k(\mathbf j_{\ell})\bigr|$ is maximal in this class, it follows from (\ref{Fstar}) that $|\alpha_k(\mathbf i_{\ell}) - \alpha_k(\mathbf j_{\ell})| \geq \frac{L \delta_k}{2n(n-1)}$, from which the desired conclusion follows.    
\end{proof} 

We end this section by applying these definitions to the intersections of the sets $S_k$.  
Fix $k \geq 1$, and suppose that the sets $S_1,\dots,S_k$ have been chosen.
Recalling from Subsection \ref{subsec-basic-constr} that $S_k =
\bigcup_{X_k(\mathbf i)=1} I_k(\mathbf i)$ and restricting the scale
factors $r, s \in [1,2]$, we find that any intersection of the form $(x+rS_k)
\cap (y+s S_k)$ is nonempty if and only if there exists at least one pair of
multi-indices $(\mathbf i, \mathbf j)$ such that   
$X_k(\mathbf i) = X_k(\mathbf j) = 1$ and (\ref{int-e1}) holds.
In general, there may be many such pairs $(\mathbf i, \mathbf j)$. 
Given two affine copies of $S_k$ with a large intersection, one of two
cases must arise: either there will be a strong match, in the sense that the number
of internal tangencies will be large, or else all but a few such pairs will be
transverse intersections.  We will need to treat these two situations differently.
As before, the exact definitions are stated for
general $n$-fold intersections of affine copies of $S_k$.

\begin{definition}
Let $\{S_k : k \geq 1\}$ be a sequence of sets constructed as in Subsection \ref{subsec-basic-constr}. Given $\mathbf A_n = \{(c_{\ell}, r_{\ell}) : 1 \leq \ell \leq n \} \subseteq [-4,0] \times [1,2]$, the sets $x_{\ell} + r_{\ell} S_k$ are said to have $L$ internal tangencies (respectively transverse intersections) if 
\[ \# \{(\mathbf i_1, \cdots, \mathbf i_n) \in \mathbb F_{\text{int}} \text{ (resp. } \mathbb F_{\text{tr}}) : X_k(\mathbf i_1) = \cdots =  X_k(\mathbf i_n) = 1 \} = L. \]  The total number of intersections among $x_{\ell} + r_{\ell} S_k$ is defined to be the sum of the numbers of internal tangencies and transverse intersections.    
\end{definition}

A large number of internal tangencies among $c_{\ell} + r_{\ell}S_k$ implies a lower bound on $\#(\mathbb F_{\text{int}})$, which in light of Lemma 2.6 (and regardless of what $S_k$ may be) provides a gain in the form of relative proximity of the translation parameters $\{c_{\ell}\}$. On the other hand, controlling the transverse intersections will be possible only under certain additional assumptions on $S_k$. We take up this issue in Sections \ref{section-maximal} and \ref{section-random}.


\section{Preliminary reductions}
We now begin our analysis of the restricted maximal operator $\mathcal M$ defined in (\ref{max-e3}). In this section, we decompose $\mathcal M$ as a sum of auxiliary restricted maximal operators $\calm_k$, each of which is then replaced by a linearized and discretized operator $\Phi_k$. We will subsequently investigate the $L^p\to L^q$ mapping properties of $\Phi_k$ when acting on functions supported in a fixed compact set. While these reductions 
are well known and have been used extensively in the literature, it is not entirely
straightforward to adapt them to the specific situation at hand, hence we include them
for completeness. 
\subsection{Spatial restriction} \label{subsec-spatial} 
\begin{lemma} \label{spatial-lemma}
Suppose that there are exponents $(p,q)$ with $1 \leq p \leq q < \infty$ and a constant $A > 0$ such that $\mathcal M$ as in (\ref{max-e3}) satisfies 
\begin{equation}  \|\mathcal M f\|_{q} \leq A \|f\|_p \quad \text{ for all } f \in L^p[0,1]. \label{z-e10} \end{equation}  
Then the inequality in (\ref{z-e10}) continues to hold for all $f \in L^p(\mathbb R)$, with the constant $A$ replaced by $4^{\frac{1}{q}}A$. 
\end{lemma} 
\begin{proof}
It suffices to prove the assertion for functions $f \in L^p(\rr)$ of arbitrary compact support. Given any such $f$, we can find an integer $R$ such that $f = \sum_{i=-R}^{R} f_i$, where $f_i$ is supported in $[i, i+1]$. Observe that the support of $\mathcal Mf_i$ is contained in $[i-4, i]$, which implies 
\begin{equation}\label{scales-e7}
\begin{split}
\|\mathcal M f \|_q^q =\Big\| \mathcal M \bigl(\sum_i f_i \bigr) \Big\|_q^q 
&\leq \Big\| \sum_i \mathcal M f_i \Big\|_q^q\\ 
&\leq 4 \sum_i \| \mathcal M f_i \|_q^q \leq 4 \sum_{i=-R}^R A^q \|f_i\|_p^q\ .\\
\end{split}
\end{equation}  
In the second line we have used the finitely overlapping supports for $\mathcal M f_i$, and then applied (\ref{z-e10}) to each $f_i$. If $p\leq q$, we estimate the last sum in (\ref{scales-e7}) by
\begin{equation*}
4A^q \sum_{i=-R}^R \big(\|f_i\|_p^p\big)^{\frac{q}{p}}
\leq 4 A^q \Big[\sum_{i=-R}^R \|f_i\|_p^p\Big]^{\frac{q}{p}}
\leq 4A^q \|f\|_p^q\ .
\end{equation*}
\end{proof} 
We will henceforth assume that all functions are supported on $[0,1]$, so that $\mathcal M$ is supported within the fixed compact set $[-4,0]$.

\subsection{Linearization and discretization} \label{subsec-lin-disc}

Define the auxiliary restricted maximal operators
\begin{equation}\label{max-e101}
\calm_k f(x):=\sup_{1<r<2} \Big| \int f(x+ry) \sigma_k(y)dy\Big| \quad \text{ where } 
\quad \sigma_k=\phi_{k+1}-\phi_k.
\end{equation} 
Then
\[ \mathcal Mf \leq \mathcal N f+ \sum_{k=1}^{\infty} \mathcal M_k|f|, \]
where $\mathcal N f(x)=\sup_{1<r<2} \int |f(x+ry)| \phi_1(y)dy$. It is an easy exercise to
deduce from H\"older's inequality that $\|\mathcal N f\|_q\leq 4^{1/q}\|\mathcal N f\|_\infty
\leq 4^{1/q}\|\phi_1\|_{p'} \|f\|_p$ for any $p,q\in [1,\infty]$; 
the main task is to estimate $\calm_k$ with $k\geq 1$. 
We begin by discretizing each $\mathcal M_k$ in the space of affine transformations. Specifically, we decompose the spaces of translations $x$ and dilations $r$ (i.e. the intervals $[-4,0]$ and $[1,2]$) into disjoint intervals $\{Q_i\}$ and $\{R_i\}$ respectively, of length $\delta_{k+1}^L$, where $L$ is an integer to 
be fixed at the end of this subsection. The centers of $Q_i$ 
and $R_i$ are denoted by $c_i$ and $r_i$ respectively. Let 
\[ \mathcal C = \{ c_i : 1 \leq i \leq 4\delta_{k+1}^{-L} \}, \quad \mathcal R = \{r_i : 1 \leq i \leq \delta_{k+1}^{-L} \}. \]

\begin{proposition}\label{lin-prop1}
Fix $1< p<\infty$.  Then there is a large integer $L = L(p)$ and a small constant
$\eta=\eta(p)>0$ such that the following conclusions hold:
\begin{enumerate}[(a)] 
\item For every $f\in C_c[0,1]$, there are measurable functions $c(x)$ and $r(x)$ depending
on $f$ and taking values in the discrete sets $\mathcal C$ and $\mathcal  R$ respectively, such that
\begin{equation}\label{otter-e1}
\calm_kf(x)\leq 4 |\Phi_k f(x)|+\mathcal E_k f(x),
\end{equation}
where 
$$
\Phi_k f(x)= \int f(z)  V_{k,x}(z)dz, \quad \text{ with } \quad
V_{k,x}(z)=\sigma_k\left(\frac{z-c(x)}{r(x)}\right).$$
\item Both $\Phi_kf$ and $\mathcal E_kf$ are supported on $[-4,0]$. 
\item For every $q\geq 1$ there is a constant $C_{p,q}$ such that
\begin{equation}\label{lin-e2}
\|\mathcal E_k f\|_q\leq C_{p,q} 2^{-k\eta} \|f\|_p.
\end{equation}
\end{enumerate}
\end{proposition}

\begin{proof}
Fix a function $f\in C_c[0,1]$.  Since $f$ is bounded, so is $\calm_k f(x)$.
For $x \in Q_i$, there exists $\tilde r_i(x) \in[1,2]$ (depending on $f$ and measurable in $x$) such that
\begin{equation}\label{lin-e10}
\begin{split}
\calm_k f(x)&\leq 2\left|\int f(x+\tilde r_i(x) y)\sigma_k(y)dy\right|\\
&\leq 4\left|\int f(z)\sigma_k\left(\frac{z-x}{\tilde r_i(x)}\right)dz\right|\\
&\leq 4\left|\int f(z)\sigma_k\left(\frac{z-c_i}{r_{j(i)}}\right)dz\right|+\calee_k f(x),\\
\end{split}
\end{equation}
where
\begin{equation}\label{lin-e88}
\calee_k f(x)=4\left|\int f(z)\left[\sigma_k\left(\frac{z-x}{\tilde r_i(x)}\right)
-\sigma_k\left(\frac{z-c_i}{r_{j(i)}}\right)\right] dz\right|
\end{equation}
and the index $j(i) = j(i, x, f)$ is chosen so that $\tilde r_i(x)\in R_{j(i)}$.  Note that $|\tilde r_i(x)
-r_{j(i)}(x)|\leq \delta_{k+1}^L$. Thus (\ref{otter-e1}) holds with $c(x) = c_i$ and $r(x) = r_{j(i)}$ for $x \in Q_i$. We obtain below a pointwise bound for the operator $\mathcal E_k$ which eventually leads to (\ref{lin-e2}). For this estimate, we will not need to use the fact that $\tilde{r}_i$ and $r_{j(i)}$ are functions of $x$, and suppress this dependence in the rest of the proof.  

Since each $\mathcal M_k f$ is supported on $[-4,0]$, it is obvious from (\ref{lin-e10}) and (\ref{lin-e88}) that so are 
$\Phi_kf$ and $\calee_kf$.  It remains to prove (\ref{lin-e2}). For this we observe that 
\begin{equation}\label{lin-e11}
\begin{split}
|\calee_k f(x)|\leq &
4\left|\int f(z)\left[\phi_{k+1}\left(\frac{z-x}{\tilde r_i}\right)
-\phi_{k+1}\left(\frac{z-c_i}{r_{j(i)}}\right)\right] dz\right|\\
&+ 
4\left|\int f(z)\left[\phi_{k}\left(\frac{z-x}{\tilde r_i}\right)
-\phi_{k}\left(\frac{z-c_i}{r_{j(i)}}\right)\right] dz\right|. 
\end{split}
\end{equation}
By H\"older's inequality, the first term on the right side of (\ref{lin-e11}) is bounded by
\begin{align*}
&\|f\|_p\,\left\|\phi_{k+1}\left(\frac{ \cdot -x_i}{\tilde r_i}\right)
-\phi_{k+1}\left(\frac{ \cdot -c_i}{r_{j(i)}}\right)\right\|_{L^{p'}(\cdot)}\\
&=\frac{1}{P_{k+1}\delta_{k+1}}\|f\|_p\,\Big\|
\sum_m(\one_{x_i+\tilde r_iI_m^{(k+1)}}-\one_{c_i+ r_{j(i)}I_m^{(k+1)}})\Big\|_{p'}\\
&\leq \frac{2^{1/p}}{P_{k+1}\delta_{k+1}}\|f\|_p\,\Big\|
\sum_m(\one_{x_i+\tilde r_iI_m^{(k+1)}}-\one_{c_i+ r_{j(i)}I_m^{(k+1)}})\Big\|_{1}^{1/p'}\\
&\leq \frac{2^{1/p}}{P_{k+1}\delta_{k+1}}\|f\|_p\cdot \left(
\sum_m|(x_i+\tilde r_iI_m^{(k+1)})\triangle(c_i+ r_{j(i)}I_m^{(k+1)})|\right)^{1/p'}. 
\end{align*}
By Lemma \ref{lin-lemma1} below, each symmetric difference 
$(x_i+\tilde r_iI_m^{(k+1)})\triangle(c_i+ r_{j(i)}I_m^{(k+1)})$ has measure bounded by 
$3\delta_{k+1}^L$.  Hence the last expression is bounded by
\begin{align*}
 \frac{2^{1/p}}{P_{k+1}\delta_{k+1}}\|f\|_p \Big(P_{k+1}\delta_{k+1}^L\Big)^{1/p'}
 &\leq \frac{2^{1/p}\delta_{k+1}^{(L-1)/p'}}{(P_{k+1}\delta_{k+1})^{1/p}}\|f\|_p\\
&\leq 2^{1/p}\delta_{k+1}^{\frac{L}{p'}-1}\|f\|_p \leq C 2^{-(k+1)\eta}\|f\|_{p},
\end{align*}
where $\eta=\frac{L}{p'}-1$ is positive for large enough $L$ whenever $p>1$.
We have used the trivial bounds $P_{k+1}\geq 1$ and $N_k\geq 2$.
The second term in (\ref{lin-e11}) is bounded similarly, with $P_{k+1},\delta_{k+1}$ 
replaced by $P_k,\delta_k$.
Finally, (\ref{lin-e2}) follows from the pointwise bound above and the fact 
that $\calee_k$ are supported on the bounded interval $[-4,0]$.
\end{proof}

\begin{lemma}\label{lin-lemma1}
Let $0<t<1$, $\frac{1}{2}<r,s<2$.  Then for any $x,y\in\rr$ we have
$$
\left|[x,x+rt]\triangle [y,y+st]\right|\leq 3\eta
$$
whenever $\eta<t/2$ and $|x-y|<\eta$, $|r-s|<\eta$.
\end{lemma}

\begin{proof}
We may assume without loss of generality that $x\leq y$. Observe first that the two intervals
cannot be disjoint, since $y-x<\eta<\frac{t}{2}<rt$.  Hence we must have either
$x\leq y\leq x+rt\leq y+st$ or $x\leq y\leq y+st\leq x+rt$. In the first case, the
symmetric difference has measure $(y-x)+(y+st-x-rt)=2(y-x)+t(r-s)\leq 3\eta$.
In the second case, its measure is $(y-x)+(x+rt-y-st)=(r-s)t\leq\eta$.

\end{proof}


\subsection{The interpolation argument}\label{sub-interpolation}

We now turn to the question of proving $L^p \rightarrow L^q$ bounds for $\Phi_k$. 
In the next lemma we show how such bounds follow from a restricted strong-type 
estimate for the ``adjoint'' operator $\Phi_k^{\ast}$ given by
\begin{equation} 
\Phi_{k}^{\ast}g(z) = \int g(x) V_{k,x}(z) \, dx. \label{adjoint} 
\end{equation}
Although similar interpolation arguments are ubiquitous in the literature,
the sequence of steps in the proof is somewhat more complicated than usual, due to the
additional challenge of keeping track of the dependence of the operator norm 
of $\Phi_k^{\ast}$ on $k$. 
     
\begin{lemma} \label{interpolation-lemma}
Let $\Phi_k^{\ast}$ be as in (\ref{adjoint}) and $q_0 \geq 2$. Suppose that $\Phi_k^{\ast}$ 
obeys the restricted strong-type estimate 
\begin{equation} || \Phi_k^{\ast} \mathbf 1_{\Omega} ||_{q_0} \leq 2^{-k \eta_0}|\Omega|^{\frac{q_0-1}{q_0}} \quad \text{ for all sets } \Omega \subseteq [0,1]  \label{pseudo-restricted} \end{equation}  
with some $\eta_0>0$.
Then for any $p > \frac{q_0}{q_0-1}$ there is an $\eta(p)>0$ such that $\Phi_k$ is bounded from $L^p[0,1]$ to $L^{p(q_0-1)}[-4,0]$ with operator norm bounded by $2^{-k\eta(p)}$. 
\end{lemma} 
  
\begin{proof}
The operator $\Phi_k^{\ast}$ satisfies a trivial $L^1 \rightarrow L^1$ bound, with operator norm bounded by a constant independent of $k$. On one hand, by a standard interpolation theorem for operators satisfying restricted weak-type endpoint bounds (Chapter 4, Theorem 5.5, \cite{BS}), $\Phi_k^{\ast}$ is bounded from $L^p \rightarrow L^q$
for all $(p,q)$ satisfying $p' = q_0/\theta$ and $q' = q_0/(\theta (q_0-1))$, $0 < \theta < 1$,
with norm bounded uniformly in $k$ but not necessarily decaying as $k\to\infty$.
On the other hand, by H\"older's inequality 
\begin{align*} 
||\Phi_k^{\ast}\mathbf 1_{\Omega}||_q &\leq ||\Phi_k^{\ast}\mathbf 1_{\Omega}||_{q_0}^{\theta} ||\Phi_k^{\ast}\mathbf 1_{\Omega}||_1^{1 - \theta} \leq C2^{-k\eta_0 \theta} |\Omega|^{\frac{1}{p}}.
\end{align*} 
By Theorem 5.3 of \cite[Chapter 4]{BS}), the last two statements imply that the weak-type $(p,q)$ norm of $\Phi_k^{\ast}$ is bounded by $C2^{-k\eta_0 \theta}$ (possibly with a different 
constant). Note
that $p\leq q$, hence we may apply the Marcinkiewicz interpolation theorem (Theorem 4.13
and Corollary 4.14, Chapter 4, \cite{BS}) to two such pairs $(p,q)$ to get the desired
strong-type Lebesgue mapping properties on all the intermediate spaces and with
the operator norms decaying exponentially in $k$. The statement for $\Phi_k$ follows by duality.  
\end{proof} 
Combining Lemma \ref{interpolation-lemma} with Proposition \ref{lin-prop1} and Lemma
\ref{spatial-lemma}, we arrive at the following corollary. 

\begin{corollary} \label{interpolation-cor}
Assume that (\ref{pseudo-restricted}) holds. Then for every $\frac{q_0}{q_0-1} < p < \infty$, 
there is an $\eta(p)>0$ such that 
\[\|\mathcal M_k f\|_{(q_0-1)p} \leq 2^{-k \eta(p)} \|f\|_p \]
for all $f \in L^p[0,1]$. Moreover, the restricted maximal operator $\mathcal M$ is 
bounded from $L^p(\mathbb R)$ to $L^{(q_0-1)p}(\mathbb R)$.   
\end{corollary} 

\section{Transverse correlations} \label{section-maximal}

We now come to the main part of our proof.  The first step, to be accomplished in this 
section, is to reduce the problem of deriving restricted strong-type $L^{\frac{n}{n-1}}
\rightarrow L^n$ bounds on $\Phi_k^{\ast}$ to estimating $n$-fold correlations
between affine copies of $S_k$ with few internal tangencies.  
The construction of a sequence of sets $S_k$ that will meet the correlation condition 
in question will be addressed in Section \ref{section-random}. 
We start by setting up the
notation for such $n$-fold correlations and giving a precise statement of our correlation
criterion. 

Throughout this section, $n\geq 2$ will be a fixed even integer.
We will use $\mathfrak A = \mathfrak A[n,k,L]$ to denote the finite collection of all 
$n$-tuples of translation-dilation pairs that arise from
the $\delta_{k+1}^L$ discretization procedure in Section \ref{subsec-lin-disc}:
\[ 
\mathfrak A := \{\mathbf A_n : \mathbf A_n = \{(c_{\ell}, r_{\ell}) : 1 \leq \ell \leq n \}, \; c_{\ell} \in \mathcal C, r_{\ell} \in \mathcal R \}. 
\]
In particular, we have $\#(\mathfrak A) \leq 4 \delta_{k+1}^{-2Ln}$. We will also use 
${\mathfrak A}_{\text{tr}}$ to denote the subcollection of those $n$-tuples which 
have few internal tangencies:
$$
{\mathfrak A}_{\text{tr}}=\{\mathbf A_n \in \mathfrak A: \ 
\#(\mathbb F_{\text{int}}[n,k;\mathbf A_n]) < P_k^{1 - \epsilon_0}\},
$$
where $\epsilon_0\in (0,1)$ is a fixed constant (eventually, we will 
let $\epsilon_0=\frac{1}{2}$).
We write ${\mathfrak A}_{\text{int}}={\mathfrak A}\setminus {\mathfrak A}_{\text{tr}}$.

\begin{definition}\label{correlation-condition}
Let ${\mathbf A}_n\in {\mathfrak A}$, and let $f_1,\dots,f_n$ be functions on $\rr$.  We define
the {\em $n$-fold correlation of $f_1,\dots,f_n$ according to ${\mathbf A}_n$} as follows:
\begin{equation}\label{corr-e1}
\Lambda({\mathbf A}_n;f_1,\dots,f_n)=
\int \prod_{\ell = 1}^{n} f_{\ell} \Bigl(\frac{z-c_{\ell}}{r_{\ell}} \Bigr) \, dz.
\end{equation}
\end{definition}
If $f_1=\dots=f_n=f$, we will write $\Lambda({\mathbf A}_n;f,\dots,f)=
\Lambda({\mathbf A}_n;f)$.

The main result in this section is the following. 
\begin{proposition} \label{prop-transverse}
Suppose that for some positive even integer $n \geq 1$ and small constant $\epsilon_0 > 0$, the following transverse correlation condition holds: 
\begin{equation} 
\sup_{{\mathbf A}_n\in{\mathfrak A}_{\text{tr}}}|\Lambda({\mathbf A}_n;\sigma_k) |
\leq C_0(k,n, \epsilon_0). 
\label{transverse-cond}  \end{equation} 
Then the operator $\Phi_k^{\ast}$ defined in (\ref{adjoint}) satisfies the
restricted strong-type estimate  
\begin{equation}\label{critter-weak}
\sup_{\Omega \subseteq [0,1]} \frac{\|\Phi_k^{\ast}\mathbf 1_{\Omega}\|_{n}}{|\Omega|^{\frac{n-1}{n}}} \leq C \left[\max \left(
\frac{2^n n^4 P_k^{\epsilon_0-1}}{(P_{k+1} \delta_{k+1})^{n-1}}, C_0(k,n, \epsilon_0)\right) \right]^{\frac{1}{n}},   
\end{equation}  
where $C > 0$ is an absolute constant independent of $n$, $k$ and $\epsilon_0$.
\end{proposition}

\noindent{\em{Remarks:}} 
\begin{enumerate}
\item Our goal will be to construct sets $S_k$ for which $C_0(k,n,\epsilon_0)$, and indeed
the right hand side of (\ref{critter-weak}), decay exponentially in $k$. It will then 
follow from Corollary \ref{interpolation-cor} that $\mathcal M$ is bounded from 
$L^p(\mathbb R) \rightarrow L^{(n-1)p}(\mathbb R)$ for all $p > \frac{n}{n-1}$. 

\item The heuristic reason why (\ref{transverse-cond}) should hold is that, essentially,
$\sigma_k$ are highly oscillating random functions with $\int\sigma_k=0$, so that two affine 
copies of $\sigma_k$ with generic translation and scaling parameters should be close to
orthogonal.  In other words, there should be a lot of cancellation in the integral
defining $\Lambda({\mathbf A}_n;\sigma_k)$.  The only exception to this is when relatively
close correlations between two or more such copies are forced by a large number of internal
tangencies.

\end{enumerate}

\medskip

In the proof of the proposition we will need the following
trivial bound (ignoring all cancellation) on $\Lambda({\mathbf A}_n;\sigma_k)$.

\begin{lemma} \label{trivial-estimate-lemma}
For all $k \geq 1$ and ${\mathbf A}_n\in{\mathfrak A}$, we have 
\begin{equation} 
|\Lambda({\mathbf A}_n;\sigma_k)|\leq 
\frac{2^{n+1}}{(P_{k+1} \delta_{k+1})^{n-1}}. 
\label{trivial-estimate} \end{equation}  
\end{lemma}  
\begin{proof} Recalling that $\sigma_k = \phi_{k+1} - \phi_k$, and expanding the product 
in $\Lambda({\mathbf A}_n;\sigma_k)$, we arrive at the expression  
\begin{equation}
\begin{aligned} 
|\Lambda({\mathbf A}_n;\sigma_k)|
\leq  \sum_{\pmb{\lambda} \in \{0,1\}^n} 
|\Lambda({\mathbf A}_n;\phi_{k+\lambda_1},\dots, \phi_{k+\lambda_{n}}  )|,
 \end{aligned} \label{trest-step1}
\end{equation}
where $\pmb{\lambda} = (\lambda_1, \cdots, \lambda_n)$. We treat each summand separately. 
Suppose first that $\lambda_{\ell_0} = 1$ for some $\ell_0$. Since 
$\phi_{k+1} = (P_{k+1} \delta_{k+1})^{-1} \mathbf 1_{S_{k+1}}$, we may estimate all
factors pointwise by $(P_{k+1} \delta_{k+1})^{-1}$, 
so that
\begin{equation} \label{trest-step2}
\begin{aligned} 
|\Lambda({\mathbf A}_n;\phi_{k+\lambda_1},\dots, \phi_{k+\lambda_{n}}  )|
&\leq \frac{1}{(P_{k+1} \delta_{k+1})^n} \int \mathbf 1_{S_{k+1}}\Bigl(
\frac{z - c_{\ell_0}}{r_{\ell_0}} \Bigr) \, dz  \\ 
&\leq \frac{2P_{k+1}\delta_{k+1}}{(P_{k+1} \delta_{k+1})^n} 
= \frac{2}{(P_{k+1} \delta_{k+1})^{n-1}}. 
\end{aligned} \end{equation} 
If on the other hand $\lambda_{\ell} = 0$ for all $\ell$, we have
\begin{equation} \label{trest-step3}
\begin{aligned} 
|\Lambda({\mathbf A}_n;\phi_{k+\lambda_1},\dots, \phi_{k+\lambda_{n}}  )|
&\leq  \frac{1}{(P_k \delta_k)^n}\int \mathbf 1_{S_{k}}\Bigl( \frac{z - c_{1}}{r_1} \Bigr)\, dz
\\ 
&\leq \frac{2P_k \delta_k}{(P_k \delta_k)^n} = \frac{2}{(P_{k} \delta_{k})^{n-1}} \leq \frac{2}{(P_{k+1} \delta_{k+1})^{n-1}},
\end{aligned} \end{equation} 
where the last step uses the fact that the sequence $\{P_{k} \delta_k\}$ is monotone 
decreasing. Combining (\ref{trest-step1}), (\ref{trest-step2}) and (\ref{trest-step3}) 
yields the desired conclusion. 
\end{proof}

\noindent\textit{Proof of Proposition \ref{prop-transverse}.}
For $x_1,x_2,\dots,x_n\in [0,1]^n$, let
$$
{\mathbf A}(x_1,\dots,x_n)=\{(c(x_{\ell}), r(x_{\ell})) : 1 \leq \ell \leq n \},
$$
where $c(x_{\ell}), r(x_{\ell})$ are chosen as in Section \ref{subsec-lin-disc}.
Thus ${\mathbf A}(x_1,\dots,x_n)\in{\mathfrak A}$.
Let $\Omega \subseteq [0,1]$, then
\begin{align*}
\| \Phi_k^{\ast} \mathbf 1_{\Omega}\|_{n}^n 
&= \left\| \int_{\Omega} V_{k,x}(\cdot) dx \right\|_n^n \\ 
&= \int \prod_{j=1}^{n} \left[ \int_{\Omega} V_{k,x_j}(z)\, dx_j \right] dz \\ 
&= \int_{\Omega^n} \left[ \int \prod_{j=1}^{n} V_{k,x_j}(z) \, dz \right] dx_1 \cdots dx_n \\ 
&= \int_{\Omega^n} \Lambda({\mathbf A}(x_1,\dots,x_n);\sigma_k)\, dx_1\dots dx_n \\
&= \left[\int_{\Theta_1} + \int_{\Theta_2} \right]
\Lambda({\mathbf A}(x_1,\dots,x_n);\sigma_k)\, dx_1\dots dx_n,
\end{align*} 
where 
\begin{align*}
\Theta_1 &= \left\{(x_1, \cdots, x_n) \in \Omega^n : \ 
{\mathbf A}(x_1,\dots,x_n)\in\aint \right\},  \\ 
\Theta_2 &= \left\{(x_1, \cdots, x_n) \in \Omega^n : \ 
{\mathbf A}(x_1,\dots,x_n)\in\atr \right\}.  \\ 
\end{align*}
We first estimate the integral on $\Theta_1$. While the high order of internal 
tangency does not allow a better estimate than (\ref{trivial-estimate}) on the
integrand, the domain of the integration is restricted to a small set. Specifically, 
by Lemma \ref{internal-tangency} we have
\begin{align*} 
\Theta_1 &\subseteq \bigcup_{1 \leq \ell \ne \ell' \leq n}
\Bigl\{(x_1, \cdots, x_n) \in \Omega^n : 
|c(x_{\ell}) - c(x_{\ell'})| \leq \frac{80n(n-1)}{P_k^{1-\epsilon_0}} \Bigr\} \\ 
&\subseteq \bigcup_{1 \leq \ell \ne \ell' \leq n} \Bigl\{(x_1, \cdots, x_n) \in \Omega^n :
 |x_{\ell} - x_{\ell'}| \leq \frac{160n(n-1)}{P_k^{1-\epsilon_0}} \Bigr\}, 
\end{align*} 
where we used that $|x_{\ell}-c(x_\ell)|\leq \delta_{k+1}^{L}\leq \delta_k^{L}\leq P_k^{-1+\epsilon_0}$.
Combining this with Lemma \ref{trivial-estimate-lemma} we obtain   
\begin{equation} 
\begin{aligned}
\int_{\Theta_1} &
\Lambda({\mathbf A}(x_1,\dots,x_n);\sigma_k)\, dx_1\dots dx_n \\  
&\quad \leq \frac{2^{n+1}}{(P_{k+1} \delta_{k+1})^{n-1}} \sum_{1 \leq \ell \ne \ell' \leq n} 
 \int_{\Omega^{n-1}}\Biggl[\int_{|x_{\ell} - x_{\ell'}| 
 \leq \frac{160n(n-1)}{P_k^{1 - \epsilon_0}}} dx_{\ell} \Biggr] \prod_{j \ne \ell} dx_j \\ 
& \quad \leq \frac{2^{n+1} 160 n^2 (n-1)^2}{(P_{k+1} \delta_{k+1})^{n-1}P_k^{1-\epsilon_0}}
 |\Omega|^{n-1} \leq 320 \frac{2^n n^4 P_k^{\epsilon_0-1}}{(P_{k+1} \delta_{k+1})^{n-1}} |\Omega|^{n-1}.
\end{aligned} \label{int-Theta1-est}
\end{equation} 
On the other hand, the desired estimate on the integral on $\Theta_2$ follows directly
from (\ref{transverse-cond}):  
\[
 \int_{\Theta_2} \Lambda({\mathbf A}(x_1,\dots,x_n);\sigma_k)\, dx_1\dots dx_n
 \leq C_0(k,n, \epsilon_0) |\Omega|^n \leq C_0(k,n,\epsilon_0) |\Omega|^{n-1}. \]
The conclusion follows upon combining this with (\ref{int-Theta1-est}).
 
\qed


\section{The random construction} \label{section-random}

\subsection{Selection of the sets $\{S_k\}$}
We are now ready to describe the probabilistic construction of the sets $\{ S_k \}$ 
satisfying the transverse correlation 
condition (\ref{transverse-cond}) with acceptable constants $C_0(k,n,\epsilon_0)$. 
The basic procedure 
is as in Subsection \ref{subsec-basic-constr}, with the crucial additional point that
the sequences $\mathbf X_k, \mathbf Y_k$ are now randomized. 

Here and in the sequel, 
$\{ \epsilon_k : k \geq 1\}$ will be a sequence of small constants with $0<\epsilon_k<\half$,
and $\{N_k: k\geq 1\}$ will be a nondecreasing sequence of large constants with 
$N_1$ large enough.
Specific choices of both sequences will be made in the next section.  
Let $\mathbf X_1 = \{ X_1(i) : 1 \leq i \leq N_1 \}$ be a sequence of independent 
and identically distributed Bernoulli random variables: 
\[ X_1(i) = \begin{cases} 1 &\text{ with probability } p_1 = N_1^{-\epsilon_1}, \\ 
0 &\text{ with probability } 1-p_1. \end{cases} \]
Each realization of the Bernoulli sequence generates a possible candidate for $S_1$:
\[ S_1=S_1(\mathbf X_1) = \bigcup_{\begin{subarray}{c} 1 \leq i \leq N_1 \\ X_1(i) = 1  \end{subarray}}[\alpha_1(i), \alpha_1(i+1)]. \]
In general, at the end of the $k$-th step, we will have selected a realization of 
$S_1, S_2, \cdots, S_k$. At step $k+1$, we will consider an iid Bernoulli 
sequence $\mathbf Y_{k+1}=\{ Y_{k+1}(\overline{\mathbf i}) : \overline{\mathbf i} = (\mathbf i, i_{k+1}) 
\in \mathbb I_{k+1}\}$ with success probability $p_{k+1} = N_{k+1}^{-\epsilon_{k+1}}$, 
and set 
\begin{equation} \label{various-def} 
\left\{
\begin{aligned} 
&\mathbf X_{k+1} = \{X_{k+1}(\overline{\mathbf i}) : \overline{\mathbf i} \in \mathbb I_{k+1} \}, \quad 
X_{k+1}(\overline{\mathbf i}) = X_k(\mathbf i) Y_{k+1}(\overline{\mathbf i}), \\ 
&P_{k+1} = P_{k+1}(\mathbf X_{k+1}) = \sum_{\mathbf i} X_{k+1}(\overline{\mathbf i}), \\
&Q_1 = N_1 p_1, \; Q_{k+1} = P_{k} N_{k+1} p_{k+1} = P_k N_{k+1}^{1- \epsilon_{k+1}}, \\  
&S_{k+1}=S_{k+1}(\mathbf X_{k+1}) = \bigcup_{X_{k+1}(\overline{\mathbf i}) = 1}
 \left[\alpha_{k+1}(\overline{\mathbf i}), \alpha_{k+1}(\overline{\mathbf i}) + \delta_{k+1} \right].
\end{aligned} \right\} \end{equation}        
At step $k+1$, the only random variables are the entries of the sequence $\mathbf Y_{k+1}$
(and hence $\mathbf X_{k+1}$), the sequence $\mathbf X_k$ having already been fixed at the
previous step. Thus at step $k+1$, $P_{k+1}$ is a random variable, whereas $Q_{k+1}$ is not. In fact, $Q_{k+1}$ is the expected value of $P_{k+1}$ conditional on $\mathbf X_{k}$ having been fixed
at the previous step.

For every $k \geq 1$, we have a large sample space of possible choices for $S_k$. The goal of this section is to show that at every stage of the construction a selection can be made that satisfies a specified list of criteria, eventually leading up to (\ref{transverse-cond}). The main result in this section is the following.

\begin{theorem} \label{random-mainprop}
Let $B>0$ be an absolute constant, independent of $k$ and $n$ ($B=10$ will work). Then
there exists a sequence of sets $\{S_k\}$ constructed as described above
(for some realization of the Bernoulli sequences $\mathbf X_1$, $\mathbf Y_k$)
such that all of the following conditions hold: 
\begin{enumerate}[(a)]
\item $2^{-k} \prod_{j=1}^{k} N_j^{1 - \epsilon_j} \leq P_k \leq 2^k \prod_{j=1}^{k} N_j^{1 - \epsilon_j}$. \label{parta}
\item $|P_k - Q_k| \leq B \sqrt{Q_k}$.  \label{partb}
\item The transverse correlation condition (\ref{transverse-cond}) holds with $\epsilon_0=\half$ and \label{partc}
\begin{multline} \label{critter}
C_0(k,n, \half) = 
4^{n+2} n!\, B 2^{k(n+\frac{3}{2})} \\ \times \Bigl[
\prod_{j=1}^k N_j^{-\half+\epsilon_j(n-\half)}\Bigr] N_{k+1}^{n\epsilon_{k+1}} 
\Bigl[ \ln(4^n n!\, B \prod_{j=1}^{k+1}N_j^{2Ln}) \Bigr]^{1/2}.
\end{multline}
\item
$\displaystyle{\sup_{\mathbf i : X_{k}(\mathbf i) = 1} \Bigl|\sum_{i_{k+1}=1}^{N_{k+1}} \bigl(X_{k+1}(\overline{\mathbf i}) - p_{k+1} \bigr) \Bigr| \leq \bigl[8 N_{k+1}^{1-\epsilon_{k+1}} \ln(4BP_k) \bigr]^{\frac{1}{2}}}$. \label{lim-lemma-parta}

\end{enumerate}
\end{theorem} 

\begin{corollary}\label{critter-corollary}
Let $\{S_k\}$ be the sequence of sets given by Theorem \ref{random-mainprop}.  Then:
\begin{enumerate}[(a)]
\item The associated operators $\Phi_k^{\ast}$ defined in (\ref{adjoint}) satisfy the
restricted strong-type estimate  
\begin{equation}\label{critter-cor}
\begin{split}
\sup_{\Omega \subseteq [0,1]} \frac{\|\Phi_k^{\ast}\mathbf 1_{\Omega}\|_{n}}{|\Omega|^{\frac{n-1}{n}}}
\leq 
C (n!\, B)^{1/n} 2^{k(1+\frac{3}{2n})}&\Bigl[
\prod_{j=1}^k N_j^{-\half+\epsilon_j(n-\half)}\Bigr]^{1/n} N_{k+1}^{\epsilon_{k+1}}\\
& \times
\Bigl[ \ln(4^n n!\, B \prod_{j=1}^{k+1}N_j^{2Ln}) \Bigr]^{1/2n}, 
\end{split}
\end{equation}  
where $C > 0$ is an absolute constant independent of $n$ and $k$.
\item  
Assume that the parameters $N_k,\epsilon_k$ have been set so that
\begin{equation}\label{boundedness}
\sup_{k \geq 1}\frac{2^{(5+\gamma)k}\ln(M_k)}{N_{k+1}^{1 - \epsilon_{k+1}}} \leq \frac{1}{32}. \end{equation}
for some $\gamma>0$. 
Then we further have
\begin{equation} \sup_{k': k' \geq k} \sum_{\begin{subarray}{c}\mathbf i : X_k(\mathbf i) = 1 \end{subarray}} \Bigl| \int_{I_k(\mathbf i)} \bigl(\phi_{k'} - \phi_k \bigr)dx \Bigr| 
\leq \frac{2B}{1-2^{-\gamma/2}} 2^{-k\gamma/2}.
\label{lim-cond-ineq} \end{equation} 
Consequently, the densities $\phi_k$ converge weakly to a 
probability measure $\mu$ supported on $S=\bigcap_{k=1}^\infty S_k$.
\end{enumerate}
\end{corollary}

\begin{proof}
Part (a) follows from Proposition \ref{prop-transverse}.
By Theorem \ref{random-mainprop}(a), we have
$$
\frac{2^n n^4 P_k^{-1/2}}{(P_{k+1}\delta_{k+1})^{n-1}}
\leq 2^{2n-1} n^4 2^{k(n-\frac{1}{2})}\Bigl[
\prod_{j=1}^k N_j^{-\half+\epsilon_j(n-\half)}\Bigr] N_{k+1}^{(n-1)\epsilon_{k+1}}.
$$
Plugging this together with (\ref{critter}) into (\ref{critter-weak}), we get 
(\ref{critter-cor}).
The inequality (\ref{lim-cond-ineq}) follows from Theorem \ref{random-mainprop}(d);
we defer the proof of this to Subsection \ref{weasel-sub3}. Since (\ref{lim-cond-ineq})
implies in particular that (\ref{limit-measure}) holds, the convergence statement
 follows from Lemma \ref{limit-measure-lemma}.
\end{proof}

The proof of Theorem \ref{random-mainprop} is arranged as follows. 
Note that parts (a)--(b) concern the set $S_k$, whereas (c)--(d) are properties of 
$S_{k+1}$; accordingly, we will say
that $S_k$ obeys (a)--(b) if (a)--(b) hold as stated above, and that $S_k$ obeys (c)--(d) if
(c)--(d) hold with $k$ replaced by $k-1$.
Fix $B$ as in the statement of the theorem,
and choose $N_1$ sufficiently large relative to $B$. 
To initialize, we prove that $S_1$ obeys (a)--(b) with probability
at least $1-B^{-1}$, in particular there exists a choice of $S_1$ with these properties.  
Assume now that we have already chosen $S_1,\dots,S_k$ obeying (a)--(d)
(where (c)--(d) hold vacuously for $S_1$), 
and consider the space of all possible choices of $S_{k+1}$.  
We will prove in Subsections \ref{weasel-sub1}--\ref{weasel-sub3} that
each of (a)-(b) and (d) fails to hold for $S_{k+1}$ with probability at most $B^{-1}$,
and the event that (a)-(b) hold but (c) fails has probability at most $B^{-1}$.
Thus there is a probability of at least $1-4B^{-1}$ that $S_{k+1}$ obeys all of (a)--(d).
Fix this choice of $S_{k+1}$, and continue by induction.

We emphasize here that we do not attempt to randomize the entire sequence of steps simultaneously. By the $(k+1)$-th stage of the iteration we have restricted attention to a {\em{deterministic}} sequence $\mathbf X_k$, with the probabilistic machinery being applied to the random sequence $\mathbf X_{k+1}$ conditional on the previously obtained $\mathbf X_k$. As a consequence, we ensure the existence of {\em{some}} sequence of desirable sets, but (in contrast 
to e.g. Salem's construction in \cite{salem}) we can make no claim as to its frequency of occurrence among all possible iterative constructions subject to the given parameters.


\subsection{Two large deviation inequalities}
\label{sub-prob}

In this subsection, we record two large deviation inequalities widely used in 
probability theory that will play a key role in the sequel. 
The first one is a version of Bernstein's inequality. We will use it here much as it was used in \cite{lp}
and  \cite{green-sumsets}.

\begin{lemma}[Bernstein's inequality]\label{ex-lemma2}
Let $Z_1,\dots, Z_m$ be independent random variables with $|Z_j|\leq 1$,
$\ee Z_j=0$ and $\ee |Z_j|^2=\sigma_j^2$.  Let $\sum \sigma_j^2\leq\sigma^2$,
and assume that $\sigma^2\geq 6m\lambda$.  Then 
\begin{equation}\label{ex-6}
\pp\Big(\Big|\sum_1^m Z_j\Big|\geq m\lambda \Big)\leq 4e^{-m^2\lambda^2/8\sigma^2}.
\end{equation}
\end{lemma}

We will also need a similar inequality for random variables which are not independent,
but instead are allowed to interact with one another to a limited extent.  The exact
statement that we need is contained in Lemma \ref{azuma} below.  Recall that a sequence
$U_1,U_2,\dots$ of random variables is a \textit{martingale} if ${\mathbb E}|U_j|<\infty$
for all $j$ and 
$$
{\mathbb E}(U_{m+1}|U_1,\dots,U_{m})=U_{m},\ m=1,2,\dots.
$$

\begin{lemma}[Azuma's inequality, \cite{wellner-notes} or \cite{a&s-book}, p.\ 95] \label{azuma}
Suppose that $\{U_k : k=0,1,2, \cdots \}$ is a martingale and $\{c_k : k \geq 0\}$ 
is a sequence of positive numbers such that $|U_{k+1} - U_k| \leq c_k$ a.s. Then for 
all integers $m \geq 1$ and all $\lambda \in \mathbb R$, 
\[ \mathbb P(|U_m - U_0| \geq \lambda) \leq 2 \exp \left({-\frac{\lambda^2}{2 \sum_{k=1}^{m} c_k^2}}\right) . \] 
\end{lemma}


\subsection{Proof of Theorem \ref{random-mainprop} (a)-(b)}
\label{weasel-sub1}

For $k=1$, let $N_1$ be chosen 
so that $6B \leq N_1^{(1 - \epsilon_1)/2}$. By Bernstein's inequality (Lemma \ref{ex-lemma2})
with $Z_i = X_1(i) - p_1$, $m=N_1$, $\sigma^2 = N_1 p_1 = N_1^{1 - \epsilon_1}$ and 
$\lambda = B N_1^{-(1 + \epsilon_1)/2}$, we have
\begin{align*} \mathbb P(\bigl|P_1 - N_1 p_1 \bigr| > BN_1^{\frac{1 - \epsilon_1}{2}}) &= \mathbb P \Bigl( \Bigl| \sum_{i=1}^{N_1}\bigl[ X_1(i) - p_1 \bigr] \Bigr| > BN_1^{\frac{1 - \epsilon_1}{2}} \Bigr) \\ &\leq 4 e^{-\frac{B^2}{8}}. \end{align*} 
Since $Q_1 = N_1p_1 = N_1^{1 - \epsilon_1}$, this shows that the inequality in (\ref{partb}) holds for $k=1$ with probability $\geq 1 - 4e^{-{B^2}/{8}}$. Further, for any $\mathbf X_1$ that satisfies (\ref{partb}), the estimate  
\[ \frac{1}{2}N_1^{1 - \epsilon_1} \leq N_1^{1 - \epsilon_1} (1 - BN_1^{-\frac{1 - \epsilon_1}{2}})  \leq P_1 \leq N_1^{1 - \epsilon_1} (1 + BN_1^{-\frac{1 - \epsilon_1}{2}}) \leq 2N_1^{1 - \epsilon_1} \]   
holds. Assume now that $\mathbf X_k$ has been selected so that (\ref{parta}) and (\ref{partb}) hold for some $k \geq 1$.  The random variables 
\[ Z_{\mathbf i} = \frac{1}{N_{k+1}} \sum_{i_{k+1} = 1}^{N_{k+1}} \bigl[Y_{k+1}(\overline{\mathbf i}) - p_{k+1} \bigr] \ , \]
indexed by $\mathbf i \in \mathbb I_k$ with $X_k(\mathbf i) = 1$, 
are iid with mean zero and variance $p_{k+1}(1 - p_{k+1})/N_{k+1}$. Hence Lemma 
\ref{ex-lemma2} applies with $m=P_k$, $\sigma^2 = P_k p_{k+1}/N_{k+1}$, and 
$\lambda = B \sqrt{p_{k+1}/(P_k N_{k+1})}$, yielding 
\begin{align*}
\mathbb P \Bigl(\Bigl| P_{k+1} &- Q_{k+1 }\Bigr| > B \sqrt{Q_{k+1}}  \Bigr) \\ &= \mathbb P \Bigl( \Bigl| \sum_{\mathbf i \in \mathbb I_k} X_k(\mathbf i) \sum_{i_{k+1}=1}^{N_{k+1}} \bigl[Y_{k+1}(\overline{\mathbf i}) - p_{k+1} \bigr] \Bigr| > B \sqrt{Q_{k+1}} \Bigr)  \\ &= \mathbb P \Bigl( \Bigl| \sum_{X_k(\mathbf i) = 1} Z_{\mathbf i} \Bigr| >  P_k \lambda \Bigr) \\ &\leq 4e^{-\frac{B^2}{8}} < B^{-1}.
\end{align*}    
Thus with large probability, (\ref{partb}) holds with $k$ replaced by $k+1$. Further by induction hypothesis (a) and the definition of $Q$, 
\begin{equation} 2^{-k} \prod_{j=1}^{k+1} N_j^{1 - \epsilon_j} \leq Q_{k+1}
 \leq 2^k \prod_{j=1}^{k+1} N_j^{1 - \epsilon_j}, \label{Qk-ind} \end{equation} 
which in particular implies that $Q_{k+1} \geq 2^{-k} N_1^{(k+1)(1 - \epsilon_1)} \geq 4B^2$ if $N_1$ is chosen sufficiently large.  
Thus for any $\mathbf X_{k+1}$ satisfying (\ref{partb}),  
\[ \frac{1}{2} \leq 1 - \frac{B}{\sqrt{Q_{k+1}}} 
\leq \frac{P_{k+1}}{Q_{k+1}} \leq 1 + \frac{B}{\sqrt{Q_{k+1}}} \leq 2, \]
which coupled with (\ref{Qk-ind}) proves the inductive step for (\ref{parta}).


\subsection{Proof of Theorem \ref{random-mainprop}(\ref{partc})}
\label{weasel-sub2}

We now begin the proof of (c), which is substantially more difficult. The strategy of the proof is outlined in 
\S\ref{subsubsec-steps} below, the execution of the various steps being relegated to the later parts of this subsection.  
\subsubsection{Steps of the proof} \label{subsubsec-steps} 

Throughout this section we will assume that $S_k$ has been selected so as to obey Theorem \ref{random-mainprop}(a)-(b). 
We begin by replacing the measure $\sigma_k = \phi_{k+1} - \phi_k$ 
in (\ref{transverse-cond}) by $\overline{\sigma}_k$,
where 
\begin{equation} \overline{\sigma}_k(z) = \frac{1}{Q_{k+1} \delta_{k+1}} \mathbf 1_{S_{k+1}}(z) - \frac{1}{P_k \delta_k} \mathbf 1_{S_k}(z).  \label{phik-bar} \end{equation}
This renders the expression in (\ref{transverse-cond}) more amenable to the application 
of the large deviation inequalities from Subsection \ref{sub-prob}, at the expense of a 
harmless error term that we estimate below.

\begin{lemma}[{\bf{Step 1}}]
Assume that Theorem \ref{random-mainprop}(a)--(b) holds at step $k+1$. 
For any $\mathbf A_n = \{(c_{\ell}, r_{\ell}) : 1 \leq \ell \leq n \} \in \mathfrak A$, 
\begin{equation} 
|\Lambda({\mathbf A}_n;\sigma_k)|\leq |\Lambda({\mathbf A}_n;\overline{\sigma}_k)|
+ 2^{2n+1}B 2^{k(n + \frac{3}{2})} \Bigl[ \prod_{j=1}^{k+1} N_j^{-\frac{1}{2} + \epsilon_j(n-\frac{1}{2})}\Bigr]. \label{error-est} \end{equation} 
In particular, this means that for any $0 < \epsilon_0 < 1$, (\ref{transverse-cond}) holds with
\begin{equation} C_0(k,n,\epsilon_0)= \sup_{\mathbf A_n \in \mathfrak A_{\text{tr}}} \bigl| \Lambda(\mathbf A_n; \overline{\sigma}_k) \bigr| + 2^{2n+1}B2^{k + \frac{3}{2}}\Bigl[ \prod_{j=1}^{k+1} N_j^{-\frac{1}{2} + \epsilon_j(n-\frac{1}{2})}\Bigr]. \label{C0-est-step1} \end{equation} \label{lemma-error-est}
\end{lemma}

\begin{proposition}[{\bf{Step 2}}]\label{trans-criterion}
Suppose that there is a constant $C_1(k,n,\epsilon_0)$ such that for all $\mathbf A_n \in \atr$ the following estimate holds:
\begin{equation}\label{trans-crit}
\begin{split}
\Biggl|\sum_{\mathbf I \in \mathbb F_{\text{tr}}}  \prod_{\ell=1}^{n} X_k(\mathbf i_{\ell}) \sum_{\pmb{\iota}} \prod_{\ell=1}^{n} \Bigl( Y_{k+1}(\overline{\mathbf i}_{\ell}) - p_{k+1}\Bigr) \cdot \bigl| \bigcap_{\ell=1}^{n} \bigl(c_{\ell} + r_{\ell} &I_{k+1}(\overline{\mathbf i}_{\ell}) \bigr) \bigr| \Biggr| \\
& \leq C_1(k,n,\epsilon_0)
\end{split}
\end{equation} 
where $\mathbf I = (\mathbf i_1, \cdots, \mathbf i_n)$, $\overline{\mathbf i}_\ell=
({\mathbf i}_\ell, i_{k+1,\ell})$, and $\pmb{\iota}=(i_{k+1,1},\dots,i_{k+1,n})$ denotes the 
$n$-vector whose entries are the $(k+1)$-th entries of $\overline{\mathbf i}_1, 
\cdots, \overline{\mathbf i}_n$ respectively (thus $\pmb \iota$ ranges over the set 
$\{1,2 \cdots, N_{k+1} \}^n$). 
Then
\begin{equation} 
\begin{aligned}
\sup_{\mathbf A_n \in \atr}|\Lambda({\mathbf A}_n;\overline{\sigma}_k)|
\leq &C_1(k,n,\epsilon_0) 2^{kn}\Bigl[\prod_{j=1}^{k+1}N_j^{n\epsilon_j}\Bigr]\\
&+ 2^{k(n+1-\epsilon_0)+3} \Bigl[\prod_{j=1}^{k} N_j^{-\epsilon_0 +
\epsilon_j(n +\epsilon_0-1)) } \Bigr]N_{k+1}^{n\epsilon_{k+1}}.
\end{aligned}
\label{transverse-cond-C1}  \end{equation} 
\end{proposition}

\begin{proposition}[{\bf{Step 3}}] \label{prop-C0-final}
The event that (\ref{trans-crit}) holds with  
\begin{equation}
C_1(k,n,\epsilon_0) = 4^n n! \Bigl[\prod_{j=1}^{k} N_j^{-\frac{1 + \epsilon_j}{2}} \Bigr] \times \Bigl[\ln \bigl(4^n n! B \prod_{j=1}^{k+1} N_j^{2Ln} \bigr) \Bigr]^{\frac{1}{2}} \label{ran-e21}
\end{equation} 
has probability at least $1-B^{-1}$.
\end{proposition}

Assume for now the claims in steps 1--3.  
\begin{proof}[Conclusion of the proof of Theorem \ref{random-mainprop} (\ref{partc})]
Of the three estimates (\ref{C0-est-step1}), (\ref{transverse-cond-C1}), and (\ref{ran-e21}), 
the first one holds with probability at least $1-B^{-1}$ (Subsection \ref{weasel-sub1}),
the second one holds always, and the third one holds with probability at least
$1-B^{-1}$ as indicated in the last proposition. Combining these estimates yields that
\begin{align*}
&|\Lambda({\mathbf A}_n;\sigma_k)|\\
&\leq 
2^{2n+1}B 2^{k(n + \frac{3}{2})} \Bigl[ \prod_{j=1}^{k+1} N_j^{-\frac{1}{2} + \epsilon_j(n-\frac{1}{2})}\Bigr]\\
&+2^{k(n+1-\epsilon_0)+3} \Bigl[\prod_{j=1}^{k} N_j^{-\epsilon_0 +
\epsilon_j(n +\epsilon_0-1) } \Bigr]N_{k+1}^{n\epsilon_{k+1}}\\
&+
4^n n!\, 2^{k(n+\half)+\half}\Bigl[
\prod_{j=1}^k N_j^{-\half+\epsilon_j(n-\half)}\Bigr] N_{k+1}^{n\epsilon_{k+1}}\times
\Bigl[ \ln(4^n n!\, B \prod_{j=1}^{k+1}N_j^{2Ln}) \Bigr]^{\frac{1}{2}},
\end{align*}
with probability at least $1-2B^{-1}$.
Plugging in $\epsilon_0=\half$, we see after some simple algebra that in this event
$|\Lambda({\mathbf A}_n;\sigma_k)|$ is bounded as indicated in (\ref{critter}).
\end{proof}

\subsubsection{Proof of Lemma \ref{lemma-error-est}}
It suffices to prove (\ref{error-est}), since (\ref{C0-est-step1}) follows directly from it.
We write $\sigma_k = \overline{\sigma}_k + e_k$, where $\overline{\sigma}_k$ is as in (\ref{phik-bar}) so that 
\[ e_{k}(z) =  \left[\frac{1}{P_{k+1} \delta_{k+1}} - \frac{1}{Q_{k+1} \delta_{k+1}} \right] \mathbf 1_{S_{k+1}}(z), \] 
Then
\begin{align*} &
\Lambda({\mathbf A}_n;\sigma_k)=\Lambda({\mathbf A}_n;\overline{\sigma}_k)+E_k,
\text{ where } \\
E_k &= \sum_{\begin{subarray}{c} \pmb{\lambda} \in \{0,1\}^n \\ \lambda_1 + \cdots + \lambda_n \geq 1 \end{subarray}} 
\Lambda({\mathbf A}_n; u_{\lambda_{1}},\dots,u_{\lambda_n}) \; \text{ with } \;
u_{\lambda} = \begin{cases} \overline{\sigma}_k &\text{ if } \lambda = 0, \\ e_k &\text{ if } \lambda = 1.  \end{cases}  
\end{align*} 
We need to show that $|E_k|$ is bounded by the quantity in (\ref{error-est}).

We observe that by the definition of $Q_k$ in (\ref{various-def}) and Theorem \ref{random-mainprop}(\ref{parta}) at step $k$,
\begin{align*}
|\overline{\sigma}_k(z)| &\leq \Biggl[ \frac{1}{Q_{k+1} \delta_{k+1}} + \frac{1}{P_k \delta_k} \Biggr] \mathbf 1_{S_k}(z) = \Biggl[ \frac{N_{k+1}^{\epsilon_{k+1}}}{P_k \delta_k} + \frac{1}{P_{k} \delta_{k}} \Biggr] \mathbf 1_{S_k}(z) \\ & \leq 2 \frac{N_{k+1}^{\epsilon_{k+1}}}{P_k \delta_k} \mathbf 1_{S_k}(z) 
\leq 2^{k+1} \prod_{j=1}^{k+1} N_j^{\epsilon_j} \mathbf 1_{S_k}(z),
\end{align*} 
whereas by Theorem \ref{random-mainprop}(\ref{partb}) at step $k+1$,
\begin{align*} |e_k(z)| &
\leq \frac{|Q_{k+1} - P_{k+1}|}{P_{k+1} Q_{k+1} \delta_{k+1}} 
\mathbf 1_{S_{k+1}}(z) \leq B \frac{1}{P_{k+1} \sqrt{Q_{k+1}} \delta_{k+1}} 
\mathbf 1_{S_{k+1}}(z) \\ &\leq B 2^{\frac{3k}{2}+1} \Bigl[\prod_{j=1}^{k+1} N_j^{-\frac{1}{2} + \frac{3 \epsilon_j}{2}} \Bigr] \mathbf 1_{S_{k+1}}(z). \end{align*}
Therefore for any $\pmb{\lambda} \in \{0,1\}^n$ with $\lambda_1 + \cdots + \lambda_n \geq 1$, there exists an index $1 \leq \ell_0 \leq n$ such that 
\[ \supp \Bigl[ \prod_{\ell=1}^{n} u_{\lambda_{\ell}} \Bigl( \frac{\cdot - c_{\ell}}{r_{\ell}} \Bigr) \Bigr] \subseteq c_{\ell_0} + r_{\ell_0} S_{k+1}. \]
Note also that the estimate on $|e_k|$ is better than the estimate on $|\overline{\sigma}_k|$
if $N_j \geq N$ and $N$ has been chosen large enough.  Hence
\begin{align*} 
&|\Lambda({\mathbf A}_n; u_{\lambda_{1}},\dots,u_{\lambda_n}) |\\
 &\quad \leq \Bigl(2^{k+1} \prod_{j=1}^{k+1} N_j^{\epsilon_j} \Bigr)^{n-1} 
 \Bigl(B 2^{\frac{3k}{2}+1} \prod_{j=1}^{k+1} N_j^{-\frac{1}{2} + \frac{3 \epsilon_j}{2}} 
 \Bigr) |\bigl(c_{\ell_0} + r_{\ell_0}S_{k+1} \bigr)| \\ 
 &\quad \leq 2^{n}B 2^{k(n+ \frac{1}{2})} \Bigl[\prod_{j=1}^{k+1} N_j^{-\frac{1}{2} 
 + \epsilon_j(n + \frac{1}{2})} \Bigr] P_{k+1} \delta_{k+1} \\ 
 &\quad \leq 2^{n+1}B 2^{k(n + \frac{3}{2})} \Bigl[ \prod_{j=1}^{k+1} N_j^{-\frac{1}{2} + \epsilon_j(n-\frac{1}{2})}\Bigr]. \end{align*}
Since the total number of terms in the sum representing $E_k$ is $2^n-1$, the desired conclusion follows.     
\qed

\subsubsection{Proof of Proposition \ref{trans-criterion}}

We need to estimate
\begin{equation} 
\Lambda({\mathbf A}_n;\overline{\sigma}_k)
=\int \prod_{\ell=1}^{n} \overline{\sigma}_k \Bigl( \frac{z - c_{\ell}}{r_{\ell}}\Bigr)  dz 
\label{sigma-bar-int} \end{equation}
for $\mathbf A_n = \{(c_{\ell}, r_{\ell}) : 1 \leq \ell \leq n \} \in \atr$. 
We start by rewriting $\overline{\sigma}_k$ as 
\begin{align*}
\overline{\sigma}_k(z) 
&= \frac{1}{Q_{k+1} \delta_{k+1}} \sum_{X_{k+1}(\overline{\mathbf i})=1} 
\mathbf 1_{I_{k+1}(\overline{\mathbf i})}(z) - \frac{1}{P_k \delta_k} 
\sum_{X_k(\mathbf i)=1} \mathbf 1_{I_k(\mathbf i)}(z) \\ 
&= \frac{1}{Q_{k+1} \delta_{k+1}} \sum_{X_k(\mathbf i)=1} \sum_{i_{k+1}=1}^{N_{k+1}}
 \Bigl(Y_{k+1}(\overline{\mathbf i}) - p_{k+1}\Bigr) \mathbf 1_{I_{k+1}(\overline{\mathbf i})}(z) \\ &= \frac{1}{Q_{k+1} \delta_{k+1}} \sum_{\mathbf i \in \mathbb I_k} X_k(\mathbf i) \sum_{i_{k+1} = 1}^{N_{k+1}} (Y_{k+1}(\overline{\mathbf i}) - p_{k+1}) \mathbf 1_{I_{k+1}(\overline{\mathbf i})}(z).
\end{align*} 
Hence
\begin{multline}
\prod_{\ell=1}^{n} \overline{\sigma}_k \Bigl( \frac{z - c_{\ell}}{r_{\ell}}\Bigr) 
= \frac{1}{(Q_{k+1} \delta_{k+1})^n} \sum_{\mathbf I \in \mathbb I_k^n}  
\Bigg[ \prod_{\ell=1}^{n}  X_k(\mathbf i_\ell) \\ \times \sum_{\pmb{\iota}} 
\Big(\prod_{\ell=1}^{n} \bigl( Y_{k+1}(\overline{\mathbf i}_\ell) - p_{k+1}\bigr)\Big)
\mathbf 1_{\bigcap_{\ell=1}^{n}(c_{\ell} + r_{\ell}I_{k+1}(\overline{\mathbf i}_\ell))}(z)\Bigg] \label{sigmak-product}
\end{multline} 
where $\mathbf I$ and $\pmb{\iota}$ are as in Proposition \ref{trans-criterion}.
Since    
\begin{equation*} \bigcap_{\ell=1}^{n} (c_{\ell} + r_{\ell} I_{k+1}(\overline{\mathbf i}_\ell))
 \subseteq \bigcap_{\ell=1}^{n} (c_{\ell} + r_{\ell} I_{k}(\mathbf i_\ell)), 
\end{equation*}  
a summand in (\ref{sigmak-product}) is nonzero only if the $n$-fold intersection on the 
right hand side above is nonempty, i.e., only if $\mathbf I \in \mathbb F = 
\mathbb F[n,k;\mathbf A_n]$. Splitting $\mathbb F$ further into $\mathbb F_{\text{int}}$ 
and $\mathbb F_{\text{tr}}$ as in Subsection \ref{sec-int-tgt}, we find that
\begin{align*} 
\Lambda({\mathbf A}_n;\overline{\sigma}_k)
  &=  \frac{1}{(Q_{k+1} \delta_{k+1})^n} \Bigl\{ \sum_{\mathbf I \in \mathbb F_{\text{int}}} + \sum_{\mathbf I \in \mathbb F_{\text{tr}}} \Bigr\} 
\Bigg[ \prod_{\ell=1}^{n} X_k(\mathbf i_\ell) \\ &\qquad \times \sum_{\pmb{\iota}} \Big(\prod_{\ell=1}^{n} \bigl( Y_{k+1}(\overline{\mathbf i}_\ell) - p_{k+1}\bigr) \Big) \Big|\bigcap_{\ell=1}^{n}(c_{\ell} + r_{\ell}I_{k+1}(\overline{\mathbf i}_\ell)) \Big| \Bigg] \\ &:= \Xi_{\text{int}} + \Xi_{\text{tr}}. \end{align*} 
We treat these two sums separately.

Since $\mathbf A_n \in \atr$, we have $\#(\mathbb F_{\text{int}}) < P_k^{1 - \epsilon_0}$, therefore  
\begin{align*}
\bigl|\Xi_{\text{int}} \bigr| &\leq \frac{1}{(Q_{k+1} \delta_{k+1})^n} \sum_{\mathbf I \in \mathbb F_{\text{int}}} \sum_{\pmb{\iota}} \prod_{\ell=1}^n \bigl|Y_{k+1}(\overline{\mathbf i}_\ell) - p_{k+1}\bigr| \times \bigl|\bigcap_{\ell=1}^{n} \bigl(c_{\ell} + r_{\ell} I_{k+1}(\overline{\mathbf i}_{\ell}) \bigr) \bigr|  \\ 
&\leq \frac{1}{(Q_{k+1} \delta_{k+1})^n}
\sum_{\mathbf I \in \mathbb F_{\text{int}}} \sum_{\pmb{\iota}}\bigl|\bigcap_{\ell=1}^{n} \bigl(c_{\ell} + r_{\ell} I_{k+1}(\overline{\mathbf i}_{\ell}) \bigr) \bigr| \\
&\leq \frac{4P_k^{1 - \epsilon_0} N_{k+1} \delta_{k+1}}{(Q_{k+1} \delta_{k+1})^n} \\
&\leq 2^{k(n+1 - \epsilon_0)+3} \Bigl[ \prod_{j=1}^{k} N_j^{-\epsilon_0 + \epsilon_j(n + \epsilon_0 - 1)} \Bigr] N_{k+1}^{n\epsilon_{k+1}},
\end{align*} 
where at the third step we have used Lemma \ref{ran-lemma3} below to estimate the number of non-zero summands in the inner sum on the second line by $4N_{k+1}$.

On the other hand, by (\ref{trans-crit})
\begin{equation*}
\bigl|\Xi_{\text{tr}}\bigr|
\leq \frac{C_1(k,n,\epsilon_0)}{(Q_{k+1}\delta_{k+1})^n}
\leq C_1(k,n,\epsilon_0)2^{kn}\Bigl[\prod_{j=1}^{k+1}N_j^{n\epsilon_j}\Bigr].
\end{equation*}
Combining the two estimates, we get (\ref{transverse-cond-C1}).
\qed

\begin{lemma}\label{ran-lemma3}
For each fixed $\mathbf I \in \mathbb I_k$, there are at most $4N_{k+1}$ distinct choices of 
$\pmb{\iota} = (i_{k+1,1}, \cdots, i_{k+1,n})$ such that 
\begin{equation}\label{ran-e40}
\bigcap_{\ell=1}^{n} 
\bigl(c_{\ell} + r_{\ell} I_{k+1}( \overline{\mathbf i}_{\ell})\bigr)
\neq\emptyset.
\end{equation}
\end{lemma}
\begin{proof}
Suppose that (\ref{ran-e40}) holds, then
\begin{equation}\label{ran-e32}
\big(\overline{\mathbf i}_{1}, \cdots, \overline{\mathbf i}_{n} \big)
\in {\mathbb F}[n,k+1,{\mathbf A}_n].
\end{equation}
If $i_{k+1,1}$ is fixed, this fixes $\overline{\mathbf i}_1$ and it follows from Lemma \ref{prelim-lemma} that the number of possible tuples $(\overline{\mathbf i}_2, \cdots, \overline{\mathbf i}_n)$ such that (\ref{ran-e32}) holds is at most 4. Hence the number of possible choices of $(i_{k+1,2},\dots,i_{k+1,n})$ is at most 4.  This proves the claim, since there are at most $N_{k+1}$ choices of $i_{k+1,1}$.
\end{proof}


\subsubsection{Proof of Proposition \ref{prop-C0-final}} 

The heart of the proof is a convenient re-indexing of the sum in (\ref{trans-crit}) that permits the application of Azuma's inequality from Subsection \ref{sub-prob}. The next lemma is a preparatory step for arranging this sum in the desired form. The lemma following it completes the verification of the martingale criterion.   
\begin{lemma}\label{ran-lemma1}
Fix ${\mathbf A}_n\in{\mathfrak A}$.  Then there is a decomposition of ${\mathbb F}_{\text{tr}}$
into at most $4^{n-1}n!$ subclasses such that
\begin{enumerate}[(a)]
\item For all $1\leq\ell\leq n$, $\pi_\ell$ is injective on each subclass. \label{decomposition-lemma-parta}
\item For each subclass, there is a permutation $\rho$ of $\{1,\dots,n\}$ such that
\begin{equation}\label{ran-e1}
\alpha_k({\mathbf i}_{\rho(1)}) < \dots <\alpha_k({\mathbf i}_{\rho(n)})
\end{equation}
for all ${\mathbf I}=({\mathbf i}_1,\dots,{\mathbf i}_n)$ in the subclass.
\end{enumerate}
\end{lemma}

\begin{proof}
Let ${\mathbf I}=({\mathbf i}_1,\dots,{\mathbf i}_n)\in {\mathbb F}_{\text{tr}}$, 
then for all $\ell\neq\ell'$
we have
\begin{equation}\label{ran-e4}
|\alpha_k({\mathbf i}_{\ell}) -\alpha_k({\mathbf i}_{\ell'})|>4\delta_k.
\end{equation}
Thus for every $\mathbf I$, all $\alpha_k({\mathbf i}_{\ell})$, $1\leq \ell\leq n$, 
are distinct, and in particular there is a permutation
$\rho=\rho({\mathbf I})$ such that (\ref{ran-e1}) holds for that ${\mathbf I}$.  Let
$\calf_\rho=\{{\mathbf I}: \ \rho({\mathbf I})=\rho\}$ for each such permutation.  
By Corollary \ref{prelim-cor1}, 
each $\calf_\rho$ can be decomposed further into at most $4^{n-1}$ subsets on which
all the projections $\pi_\ell$ are injective.
\end{proof}

By a slight abuse of notation, we will continue to use $\calf_\rho$ to denote a subclass
of ${\mathbb F}_{\text{tr}}$ such that both (i) and (ii) hold for the permutation $\rho$. In view of Lemma \ref{ran-lemma1}, it suffices to estimate
\begin{equation}\label{ran-e2}
\Biggl|\sum_{\mathbf I \in \calf_\rho}  \prod_{\ell=1}^{n} X_k(\mathbf i_{\ell}) \sum_{\pmb{\iota}} \prod_{\ell=1}^{n} \Bigl( Y_{k+1}(\overline{\mathbf i}_{\ell}) - p_{k+1}\Bigr) \cdot \bigl| \bigcap_{\ell=1}^{n} \bigl(c_{\ell} + r_{\ell} I_{k+1}(\overline{\mathbf i}_{\ell}) \bigr) \bigr| \Biggr| 
\end{equation}
for each such $\calf_\rho$.

Observe that by part (\ref{decomposition-lemma-parta}) of Lemma \ref{ran-lemma1}, the index $\mathbf I$ in the outer sum is in fact determined uniquely
by ${\mathbf i}_{\rho(n)}=\pi_{\rho(n)}({\mathbf I})$. In other words, the elements $\{ \alpha_k(\mathbf i_{\rho(n)}) : \mathbf I \in \mathcal F_{\rho} \}$ are all distinct. Furthermore, the only indices that
contribute to (\ref{ran-e2}) are those with $\prod_{\ell=1}^{n} X_k(\mathbf i_{\ell})=1$.  Accordingly, let
$$
{\mathcal J}=\Big\{{\mathbf I}=({\mathbf i}_{1}, \dots,{\mathbf i}_{n}) 
\in\calf_\rho:
\  \prod_{\ell=1}^{n} X_k(\mathbf i_{\ell})=1\Big\},
$$
and let us arrange the elements of ${\mathcal J}$ in a sequence $\{{\mathbf I}(j)= 
( {\mathbf i}_{1}(j),\dots, {\mathbf i}_{n}(j) ):\ 
j=1,\dots,T \}$
so that
\begin{equation}\label{ran-e3}
\alpha_k({\mathbf i}_{\rho(n)}(1))<\dots<\alpha_k({\mathbf i}_{\rho(n)}(T)).
\end{equation}
For $1 \leq j \leq T$, we define
\begin{equation}\label{ran-e31}
W_j=
\sum_{\pmb{\iota}} \prod_{\ell=1}^{n} 
\Bigl( Y_{k+1}({\mathbf i}_{\ell}(j),i_{k+1,\ell}) - p_{k+1}\Bigr) 
\Bigl| \bigcap_{\ell=1}^{n} 
\bigl(c_{\ell} + r_{\ell} I_{k+1}( {\mathbf i}_{\ell}(j),i_{k+1,\ell}  )
 \bigr) \Bigr|,
\end{equation}
where the summation index $\pmb{\iota}=(i_{k+1,1},\dots,i_{k+1,n} )$ is as in the 
statement of Proposition  \ref{trans-criterion}, hence
ranges over all vectors in $\{1,\dots,N_{k+1}\}^n$. We also let $W_0=0$. Then the sum in (\ref{ran-e2}) is simply $W_1 + \cdots + W_T$.  

\begin{lemma}\label{ran-lemma2}
$\{W_j:\ 0\leq j\leq T\}$ is a martingale difference sequence (i.e. the sequence
$\{W_1+\dots+W_m:\ 1\leq m\leq T\}$ is a martingale), with $|W_j| \leq 4 \delta_k$ for all $1 \leq j \leq T$. 
\end{lemma}

\begin{proof}
We need to prove that ${\mathbb E}(W_m|W_1,\dots,W_{m-1})=0$.
It suffices to demonstrate that the random variables
$Y_{k+1}({\mathbf i}_{\rho(n)}(m),\cdot)$ are
\begin{enumerate}[(i)]
\item independent of all $Y_{k+1}({\mathbf i}_{\rho(\ell)}(m),\cdot)$ with $\ell<n$,
\item independent of all $W_j$ with $j<m$.
\end{enumerate}
Once we have this, the desired conclusion follows by setting $\mathcal W$ to be the collection of random variables in (i) and (ii) above, and $\mathcal W' = \mathcal W \setminus \{W_1, \cdots, W_{m-1} \}$, so that  \begin{align*}
\mathbb E(W_m | W_1, \cdots W_{m-1}) &= \mathbb E_{\mathcal W'} \Bigl[\mathbb E\bigl(W_m \bigl| \mathcal W \bigr) \Bigr] \\ 
&= \mathbb E_{\mathcal W'} \left[ \sum_{\pmb \iota} F_{\pmb \iota, m}(\mathcal W) \mathbb E \left( Y_{k+1}(\mathbf i_{\rho(n)}(m), i_{k+1, \rho(n)}) - p_{k+1}\right)  \right] \\ &= 0.
\end{align*}
Here $\{F_{\pmb \iota,m}\}$ are measurable functions of $\mathcal W$ specified by the expression (\ref{ran-e31}) for $W_m$ but whose exact functional forms are unimportant.  

By (\ref{ran-e1}), we have
$$
\alpha_k({\mathbf i}_{\rho(\ell)}(m)) <\alpha_k({\mathbf i}_{\rho(n)}(m)),\ \ell<n,
$$
which implies immediately the first claim (i).
It remains to prove (ii). Observe that $W_j$ depends only on 
$Y_{k+1}({\mathbf i}_{\ell}(j),\cdot)$, $1\leq\ell\leq n$, hence it suffices 
to prove that
$$
{\mathbf i}_{\rho(n)}(m)\notin\Big\{{\mathbf i}_{\rho(\ell)}(j):\ 
1\leq \ell\leq n,\ 1\leq j<m\Big\}.
$$
But this follows from
$$
\alpha_k({\mathbf i}_{\rho(\ell)}(j)) \leq \alpha_k({\mathbf i}_{\rho(n)}(j))
<\alpha_k({\mathbf i}_{\rho(n)}(m)) ,\ \ell\leq n, \ j<m
$$
where we used (\ref{ran-e1}) again and then (\ref{ran-e3}).

It remains to prove the almost sure bound on $W_j$. Indeed, by Lemma \ref{ran-lemma3} the number of summands in (\ref{ran-e31}) that make 
a non-zero contribution to $W_j$ is bounded by $4N_{k+1}$.
Since the size of each summand is bounded by $\delta_{k+1}$, it follows that
$|W_j|\leq 4N_{k+1}\delta_{k+1}=4 \delta_k$, as claimed.

\end{proof}

\begin{proof}[Conclusion of the proof of Proposition \ref{prop-C0-final}]
In light of Lemma \ref{ran-lemma2}, we apply Azuma's inequality (Lemma \ref{azuma}) to the martingale sequence $U_j = W_1+\dots+W_j$, with $c_j=4\delta_k$ and 
$$\lambda=4 \delta_k\sqrt{2P_k}\sqrt{\ln(4^{n}n!B\delta_{k+1}^{-2Ln})},
$$ and obtain
\begin{equation*}
{\mathbb P}\Big((\ref{ran-e2})>\lambda\Big)
\leq 2\exp(-\frac{\lambda^2}{32\delta_k^2 T})
\leq 2\exp(-\frac{\lambda^2}{32 \delta_k^2P_k})
\leq \frac{\delta_{k+1}^{2Ln}}{4^{n-1}n!B}.
\end{equation*}
Since there are at most $4^{n-1}n!$ classes $\calf_\rho$, the probability that 
(\ref{ran-e2})$>\lambda$ for at least one of them is bounded by $B^{-1}\delta_{k+1}^{2Ln}$.
Summing over such classes, we see that
\begin{equation*}
{\mathbb P}\Big(\text{LHS of }(\ref{trans-crit})>4^{n-1}n!\,\lambda\Big)
\leq \frac{\delta_{k+1}^{2Ln}}{B}.
\end{equation*}
Finally, since $\#(\mathfrak A)=\delta_{k+1}^{-2Ln}$, there is a probability of at least
$1-\frac{1}{B}$ that (\ref{trans-crit}) holds for every ${\mathbf A}\in \atr$ with
$$
C_1(k,n,\epsilon_0)=4^{n-1}n!\,\lambda=4^{n}n!\, \delta_k \sqrt{2P_k}\sqrt{\ln(4^{n}n!B\delta_{k+1}^{-2Ln})}.
$$
By Theorem \ref{random-mainprop}(a) at step $k$,
$$C_1(k,n,\epsilon_0)
\leq 4^n n!\, 2^{\frac{k+1}{2}}\Bigl[
\prod_{j=1}^k N_j^{-\frac{1+\epsilon_j}{2}}\Bigr] \times
\Bigl[ \ln(4^n n!\, B \prod_{j=1}^{k+1}N_j^{2Ln}) \Bigr]^{\frac{1}{2}}.
$$
This completes the proof of the proposition. 
\end{proof}

\noindent {\bf{Remark:}} The following argument, communicated to us by an anonymous referee, provides an alternative strategy for estimating the expression in (\ref{ran-e2}). It involves a finer decomposition of the class $\mathcal F_{\rho}$ into at most $2^{n(n-1)/2}$ subclasses $\mathcal F_{\rho}'$ such that in addition to (a) and (b) of Lemma \ref{ran-lemma1}, each $\mathcal F_{\rho}'$ satisfies a stronger hypothesis on the projections $\pi_{\ell}$, namely
\begin{equation*}
\begin{aligned}  
\text{\em{ (c) For every }} &\mathbf I = (\mathbf i_1, \cdots, \mathbf i_n) \in \mathcal F_{\rho}'  \text{\em{ and }} \ell, m \in \{1, 2, \cdots, n \} \text{\em{ with }} \ell \ne m,  \\ 
&\mathbf i_{\ell} \not\in \pi_m(\mathcal F_{\rho}') \quad \text{\em{ and }} \quad \mathbf i_m \not\in \pi_{\ell}(\mathcal F_{\rho}').  
\end{aligned} 
\end{equation*} 
In other words, all the $k$-dimensional multi-indices occurring as the components of the elements of $\mathcal F_{\rho}'$ are distinct. 

With such a construction, the random variables $\{ W_j : 1 \leq j \leq T \}$ in (\ref{ran-e31}) are independent, hence Azuma's inequality can be replaced by classical Bernstein's. This results in a slight improvement of the bound of $C_1(k,n,\epsilon_0)$ in the proof of Proposition \ref{prop-C0-final} but does not improve the range of $p$ in the statement of the main theorems. 

The condition (c) above is achieved as follows. Fix $\ell, m \in \{1, 2, \cdots, n \}$, $\ell \neq m$. Define an $(\ell, m)$-chain of length $M$ in $\mathcal F_{\rho}$ to be a maximal sequence $\mathfrak C = \{\mathbf I_s = (\mathbf i_{1,s}, \cdots, \mathbf i_{n,s})  : 0 \leq s \leq M  \}$  with the property that 
\begin{equation} \label{finer-decomp}
\begin{aligned} 
\alpha_k(\mathbf i_{\ell, 0}) < \alpha_k(\mathbf i_{m,0}) &= \alpha_k(\mathbf i_{\ell, 1}) < \alpha_k(\mathbf i_{m, 1}) = \cdots < \alpha_k(\mathbf i_{m,M}), \text{ so that } \\
\mathbf i_{\ell, 0} &\not\in \pi_m(\mathcal F_{\rho}) \quad \text{ and } \quad \mathbf i_{m, L} \not\in \pi_\ell(\mathcal F_{\rho}).
\end{aligned} 
\end{equation}  
It is easy to see that $\mathcal F_{\rho}$ decomposes into a disjoint union of $(\ell, m)$-chains of various lengths. We now decompose 
\begin{align*} 
\mathcal F_{\rho} &= \mathcal F_{\rho}(\text{even}) \cup \mathcal F_{\rho}(\text{odd}), \text{ where } \\
\mathcal F_{\rho}(\text{even}) &= \left\{ \mathbf I_{2n}: n \in \mathbb Z_{\geq 0}, \mathbf I_{2n}\in \mathfrak C \text{ for some $(\ell, m)$-chain } \mathfrak C  \right\}, \\ 
\mathcal F_{\rho}(\text{odd}) &= \left\{ \mathbf I_{2n+1}: n \in \mathbb Z_{\geq 0}, \mathbf I_{2n}\in \mathfrak C \text{ for some $(\ell, m)$-chain } \mathfrak C  \right\}.
\end{align*}
Then (\ref{finer-decomp}) implies that (c) holds for the given choice of $(\ell, m)$. Performing the above splitting procedure iteratively for every pair $(\ell, m)$ results in a decomposition of $\mathcal F_{\rho}$ into at most $2^{n(n-1)/2}$ subclasses, each of which satisfies (a), (b) and (c).  

\subsection{Existence of the limiting measure}
\label{weasel-sub3}

\subsubsection{Proof of Theorem \ref{random-mainprop}(d)} 

Let $\mathbf i \in \mathbb I_k$ with 
$X_k(\mathbf i) = 1$. Applying Bernstein's inequality to the random
variables $X_{k+1}(\overline{\mathbf i})-p_{k+1} = Y_{k+1}(\overline{\mathbf i})
-p_{k+1}$,  with $\sigma^2 = N_{k+1}p_{k+1}$ and $\lambda = (8p_{k+1} \ln(4BP_k)/N_{k+1})^{\frac{1}{2}}$, we obtain
\[ \mathbb P \left( \Bigl| \sum_{i_{k+1}=1}^{N_{k+1}} \bigl[ Y_{k+1}(\overline{\mathbf i}) - p_{k+1} \bigr] \Bigr| > N_{k+1} \lambda \right) \leq 4 \exp \left[-\frac{N_{k+1}^2 \lambda^2}{8N_{k+1}p_{k+1}} \right] = \frac{1}{BP_k}. \]  
Since there are $P_k$-many such choices of $\mathbf i$, we find that Theorem \ref{random-mainprop}(\ref{lim-lemma-parta}) holds with probability at least $1 - \frac{1}{B}$, as claimed.

\subsubsection{Proof of Corollary \ref{critter-corollary}(b)} 

\begin{lemma} \label{prop-lim-cond}
Assume that (\ref{boundedness}) and Theorem \ref{random-mainprop}(\ref{lim-lemma-parta}) hold for all $k$. Then 
for all $k \geq 1$, $m \geq 0$ and every $\mathbf i \in \mathbb I_k$ with 
$X_k(\mathbf i) = 1$, 
\begin{equation} 2^{-m} \Bigl[\prod_{r=1}^{m} N_{k+r}^{1-\epsilon_{k+r}}\Bigr]
\leq \sum_{\mathbf j} X_{k+m}(\mathbf i, \mathbf j) 
\leq 2^m \Bigl[\prod_{r=1}^{m} N_{k+r}^{1-\epsilon_{k+r}} \Bigr], \label{lim-lemma-partb-eqn} \end{equation} 
where the sum is taken over all $m$-dimensional multi-indices $\mathbf j$ such that $(\mathbf i, \mathbf j) \in \mathbb I_{k+m}$. 
\end{lemma} 

\begin{proof}
This follows from Theorem \ref{random-mainprop}(\ref{lim-lemma-parta}) by induction on $m$. For $m=0$, (\ref{lim-lemma-partb-eqn}) holds trivially. Assuming that Theorem \ref{random-mainprop}(\ref{lim-lemma-parta}) holds for $m$ and summing over $\overline{\mathbf j} = (\mathbf j, j_{m+1})$, we arrive at the following estimate  
\begin{align*}
\Bigl|\sum_{\overline{\mathbf j}} X_{k+m+1}(\mathbf i, \overline{\mathbf j}) &- \sum_{\mathbf j} X_{k+m}(\mathbf i, \mathbf j) N_{k+m+1}^{1-\epsilon_{k+m+1}} \Bigr| \\
&\leq \sum_{\mathbf j} X_{k+m}(\mathbf i, \mathbf j)\Bigl[ N_{k+m+1}^{1-\epsilon_{k+m+1}} \ln(4BP_{k+m}) \Bigr]^{\frac{1}{2}}, 
\end{align*}  
so that 
\[\left| \frac{\sum_{\overline{\mathbf j}} X_{k+m+1}(\mathbf i, \overline{\mathbf j})}{N_{k+m+1}^{1-\epsilon_{k+m+1}}\sum_{\mathbf j}X_{k+m}(\mathbf i, \mathbf j)} - 1 \right|
 \leq  \sqrt{\frac{\ln(4BP_{k+m})}{N_{k+m+1}^{1-\epsilon_{k+m+1}}}}\leq\half, \] 
where at the last step we used (\ref{boundedness}). Thus  
\[ \frac{1}{2} \leq \frac{\sum_{\overline{\mathbf j}} X_{k+m+1}(\mathbf i, \overline{\mathbf j})}{N_{k+m+1}^{1-\epsilon_{k+m+1}}\sum_{\mathbf j}X_{k+m}(\mathbf i, \mathbf j)} \leq 2 \text{ for all } m \geq 1, \] 
which yields the desired result by induction.  
\end{proof}

\begin{proof}[Proof of Corollary \ref{critter-corollary}(b)]

Since 
\begin{align*} 
\sup_{k': k' \geq k} \sum_{\begin{subarray}{c}\mathbf i : X_k(\mathbf i) = 1 \end{subarray}} \Bigl| \int_{I_k(\mathbf i)} \bigl(\phi_{k'} - \phi_k \bigr)(x) \, dx \Bigr|  
&\leq \sum_{\begin{subarray}{c}\mathbf i : X_k(\mathbf i) = 1 \end{subarray}} \sum_{m=0}^{\infty} \Bigl| \int_{I_k(\mathbf i)} \sigma_{k+m}(x) \, dx \Bigr| \\ 
&\leq P_k \sup_{\mathbf i : X_k(\mathbf i)=1} \Biggl[ \sum_{m=0}^{\infty}\Bigl| 
\int_{I_k(\mathbf i)} \sigma_{k+m}(x) \, dx \Bigr| \Biggr],   
\end{align*} 
it suffices to prove that the quantity in the last line is bounded above by the right hand side of (\ref{lim-cond-ineq}).  
To this end, we fix an $m \geq 0$ and $\mathbf i \in \mathbb I_k$ with $X_k(\mathbf i) = 1$ and write 
\begin{align*}
P_k \int_{I_k(\mathbf i)} \sigma_{k+m}(x) \, dx 
&= \frac{P_k}{P_{k+m+1}} \sum_{\overline{\mathbf j}} X_{k+m+1}(\mathbf i, \overline{\mathbf j})
-  \frac{P_k}{P_{k+m}} \sum_{\mathbf j} X_{k+m}(\mathbf i, \mathbf j) \\
&= \Xi_1 + \Xi_2, \text{ where }
\end{align*} 
\begin{align*} 
\Xi_1 &:=  P_k \Bigl[\frac{1}{P_{k+m+1}} - \frac{1}{Q_{k+m+1}} \Bigr]\sum_{\overline{\mathbf j}} X_{k+m+1}(\mathbf i, \overline{\mathbf j}), \text{ and } \\  
\Xi_2 &:= \frac{P_k}{Q_{k+m+1}}\sum_{\mathbf j} X_{k+m}(\mathbf i, \mathbf j) \sum_{j_{m+1}} \bigl(Y_{k+m+1}(\mathbf i, \mathbf j, j_{m+1}) - p_{k+m+1} \bigr).
\end{align*} 
By Theorem \ref{random-mainprop}(\ref{parta}) and (\ref{lim-lemma-partb-eqn}), we have
\begin{equation} 
\begin{aligned}
|\Xi_1| &\leq  P_k  \frac{|Q_{k+m+1} - P_{k+m+1}|}{P_{k+m+1} Q_{k+m+1}}
 \sum_{\overline{\mathbf j}} X_{k+m+1}(\mathbf i, \overline{\mathbf j}) \\ 
&\leq \frac{BP_k}{P_{k+m+1} \sqrt{Q_{k+m+1}}} 2^{m+1} \Bigl[\prod_{j=1}^{m+1} N_{k+j}^{1-\epsilon_{k+j}}\Bigr] \\ 
&\leq B 2^{\frac{5}{2}(k+m+1)} \Bigl[ \prod_{j=1}^{k+m+1} N_j^{1 - \epsilon_j} \Bigr]^{-\frac{1}{2}}. 
\end{aligned} \label{est-Xi1} \end{equation}  
On the other hand, using both Theorem \ref{random-mainprop}(\ref{lim-lemma-parta}) and (\ref{lim-lemma-partb-eqn}),
\begin{equation} \begin{aligned}
|\Xi_2| 
&\leq \frac{P_k}{Q_{k+m+1}} \sum_{\mathbf j} X_{k+m}(\mathbf i, \mathbf j) 
\Bigl[ 8N_{k+m+1}^{1-\epsilon_{k+m+1}} \ln(BP_{k+m}) \Bigr]^{\frac{1}{2}} \\ 
&\leq \frac{P_k}{Q_{k+m+1}} 2^m \Bigl[\prod_{j=1}^{m} N_{k+j}^{1-\epsilon_{k+j}} \Bigr] 
\times \Bigl[ 8N_{k+m+1}^{1-\epsilon_{k+m+1}} \ln(BP_{k+m}) \Bigr]^{\frac{1}{2}} \\ 
&\leq 2^{2(k+m)} \sqrt{8} \frac{\bigl[\ln(BP_{k+m}) \bigr]^{\frac{1}{2}}}{
N_{k+m+1}^{(1-\epsilon_{k+m+1})/2}}.
\end{aligned} \label{est-Xi2} \end{equation}  
Combining (\ref{est-Xi1}) and (\ref{est-Xi2}) and using (\ref{boundedness}), we obtain 
\begin{align*} |\Xi_1| + |\Xi_2| &\leq 2B 2^{\frac{5}{2}(k+m+1)}\frac{\bigl[\ln(BP_{k+m}) \bigr]^{\frac{1}{2}}}{N_{k+m+1}^{(1 - \epsilon_{k+m+1})/2}}
\leq 2B\cdot 2^{-\frac{(k+m)\gamma}{2}}.
\end{align*}  
The conclusion (\ref{lim-cond-ineq}) follows upon summation in $m$.

\end{proof}


\section{The estimates for $\mathcal M$ and $\mathfrak M$}
\label{sec-choose}

In this section we prove those parts of Theorems \ref{thm-main} and \ref{thm-main3}
that concern the restricted maximal operators with $1<r<2$.  We will do this
by fixing the parameters $N_k$, $\epsilon_k$ of the random construction in
Section \ref{section-random} and showing that the conclusions of the theorems
hold for the sets $S_k$ with those choices of parameters. 
Specifically, the conclusions of Theorem \ref{thm-main} will hold for $S_k$
with
\begin{equation}\label{z-e1}
N_k=N^{k+1},\ \epsilon_k=\frac{1}{k+1},
\end{equation}
and the conclusions of Theorem \ref{thm-main3} will hold for $S_k$ with 
\begin{equation}\label{z-e2}
N_k=N^k,\ \epsilon_k=\epsilon,
\end{equation}
where $N$ is a large integer.

\begin{lemma}\label{z-lemma1}
Let $N_k,\epsilon_k$ be as above with $N$ sufficiently large.  Then:
\begin{enumerate}[(a)]
\item the set $S=\bigcap_{k=1}^\infty S_k$ has Hausdorff dimension $1$ if
(\ref{z-e1}) holds and $1-\epsilon$ if (\ref{z-e2}) holds,
\item assuming (\ref{z-e1}), (\ref{pseudo-restricted}) holds for all $q_0\geq 2$,
\item assuming (\ref{z-e2}), (\ref{pseudo-restricted}) holds for all $2\leq q_0<q_\epsilon$,
where $q_\epsilon=\frac{\epsilon+1}{2\epsilon}$ as in Theorem \ref{thm-main3},
\item assuming either (\ref{z-e1}) or (\ref{z-e2}), (\ref{boundedness}) holds with $\gamma=1$. \label{partd}
\end{enumerate} 
\end{lemma}

Lemma \ref{z-lemma1} will be proved in Subsections \ref{sub-dim1} and \ref{sub-smalldim} for
(\ref{z-e1}) and (\ref{z-e2}), respectively.

Assuming the lemma, the proof of the restricted maximal estimates is completed as follows.  
By parts (b) and (c) of the lemma, (\ref{pseudo-restricted}) holds with $q_0$ as above.
It follows by Corollary \ref{interpolation-cor} that
\begin{equation}\label{squirrel}
\|\calm_k f\|_{(q_0-1)p}\leq C 2^{-k\eta(p)}\|f\|_p,\ p>\frac{q_0}{q_0-1},
\end{equation}
for the same $q_0$.  

Consider first the case when (\ref{z-e1}) holds.  We claim that then
\begin{equation}\label{squirrel2}
\|\calm_k f\|_{q}\leq C 2^{-k\eta(p)}\|f\|_p
\end{equation}
for all $p,q\in (1,\infty)$.  Indeed, fix $p$ and $q$, and choose $q_0$ large enough so that
$\frac{q_0}{q_0-1}<p$ and $(q_0-1)p>q$.  Since $\calm_k f$ is supported on $[-4,0]$, 
we have
$$
\|\calm_k f\|_q\leq 5^{\frac{1}{q}-\frac{1}{(p-1)q_0}} \|\calm_k f\|_{(q_0-1)p}
$$
by H\"older's inequality.  Combining this with (\ref{squirrel}), we get (\ref{squirrel2}).

Summing up (\ref{squirrel2}) in $k$, we see that $\calm$ is bounded 
from $L^p[0,1]$ to $L^{q}[-4,0]$ for any $p,q\in(1,\infty)$.  
By Lemma \ref{spatial-lemma}, it follows that $\calm$ is bounded 
from $L^p(\rr)$ to $L^{q}(\rr)$ whenever $1<p\leq q<\infty$.

Assume now that (\ref{z-e2}) holds instead.  We claim that in this case (\ref{squirrel2})
holds whenever 
\begin{equation}\label{z-e11}
\frac{1+\epsilon}{1-\epsilon}<p<\infty \text{ and }1<q<\frac{1-\epsilon}{2\epsilon}p.
\end{equation}
Indeed, fix such $p$ and $q$, then $p'<\frac{1+\epsilon}{2\epsilon}=q_\epsilon$.  Choose
$q_0$ so that $p'<q_0<q_\epsilon$, then (\ref{squirrel}) yields
(\ref{squirrel2}) with $q=(q_0-1)p$.  As in the first case, (\ref{squirrel2}) also holds for
$q<(q_0-1)p$ by H\"older's inequality.  Taking $q_0\to q_\epsilon$, we get 
(\ref{squirrel2}) for all $p'<q_\epsilon$ and $q<(q_\epsilon-1)p$, which is equivalent to
(\ref{z-e11}).   We now sum up (\ref{squirrel2}) in $k$ to obtain the boundedness of $\calm$
from $L^p[0,1]$ to $L^{q}[-4,0]$ for $p,q$ as in (\ref{z-e11}).  By (\ref{z-e10}),
$\calm$ is bounded from $L^p(\rr)$ to $L^q(\rr)$ whenever $p\leq q$ and 
(\ref{z-e11}) holds.  
Note that the range of $p,q$ is nonempty whenever $\epsilon<1/3$.

The same conclusions follow automatically for $\mathfrak M$, provided that the weak
limit $\mu$ of $\phi_k$ exists.  But thanks to Lemma \ref{z-lemma1}(d), 
(\ref{boundedness}) holds, hence the existence of $\mu$ follows from 
Theorem \ref{random-mainprop}(d)
for both (\ref{z-e1}) and (\ref{z-e2}).

\subsection{The 1-dimensional case}\label{sub-dim1}

Let $N_k,\epsilon_k$ be as in (\ref{z-e1}).  Then
$M_k=N^{\frac{k(k+3)}{2}}$ and, by Theorem \ref{random-mainprop}(a), 
$$
2^{-k}N^{\frac{k(k+1)}{2}} \leq P_k \leq 2^{k}N^{\frac{k(k+1)}{2}} .
$$
By Lemma \ref{Hd-lemma}(b), 
\begin{align*}
\dim_{\mathbb H}(S) 
&\geq \liminf_{k \rightarrow \infty} \log(P_{k}/N_k)/\log(M_{k-1})\\
&\geq \liminf_{k \rightarrow \infty}\frac{\log (2^{-k}N^{\frac{k(k+1)}{2}-(k+1)})}{\log
(N^{\frac{(k-1)(k+2)}{2}})}=1.
\end{align*}
Hence $S$ has dimension 1.  

To prove Lemma \ref{z-lemma1}(b), it suffices to show that for any $q_0\geq 2$
the right side of (\ref{critter-cor}) is bounded by $C(q_0) 2^{-\eta k}$ with $\eta
=\eta(q_0)>0$.
Suppose first that $q_0=n$ is an even integer.
Plugging our values of $N_j$ and $\epsilon_j$ into (\ref{critter-cor}), we see 
after some straightforward but cumbersome algebra that
\begin{equation*}
\begin{split}
\sup_{\Omega \subseteq [0,1]} \frac{\|\Phi_k^{\ast}\mathbf 1_{\Omega}\|_{n}}{|\Omega|^{\frac{n-1}{n}}} &
\leq 
C (n!\, B)^{1/n} 2^{k(1+\frac{3}{2n})}
N^{-\frac{k^2}{4n}+(1-\frac{5}{4n})k+1}
\\
&\quad \quad \times
\Bigl[ \ln(4^n n!\, B)+(k+1)(k+4)Ln\ln N  \Bigr]^{1/2n},
\end{split}
\end{equation*}  
which is bounded by 
$C(n)2^{-\eta(n)k}$ with $\eta(n)=\frac{1}{4n}>0$ for all even integers $n$. 
The estimate in (b) for all $q_0\geq 2$ (not necessarily an even integer) follows
by interpolation.

Finally, to prove (d) we estimate 
\[
2^{6k}\frac{\ln(M_k)}{N_{k+1}^{1-\epsilon_{k+1}}}
\leq  
\frac{2^{6k-1}\,k(k+3)\ln N}{N^{k+1}}<\frac{1}{32}
\]
for all $k$, provided that $N$ is large enough.

\subsection{The lower-dimensional case}\label{sub-smalldim}
Let $N_k,\epsilon_k$ be as in (\ref{z-e2}).  
Then $M_k=N^{\frac{k(k+1)}{2}}$ and
by Theorem \ref{random-mainprop}(a), 
$$
2^{-k}N^{\frac{k(k+1)}{2}(1-\epsilon)} \leq P_k \leq 2^{k}N^{\frac{k(k+1)}{2}(1-\epsilon)} .
$$
By Lemma \ref{Hd-lemma}(a), 
\begin{align*}
\dim_{\mathbb H}(S) &\leq \liminf_{k \rightarrow \infty} \log(P_k)/\log(M_k)\\
&\leq \liminf_{k \rightarrow \infty} 
\frac{\log(2^k N^{\frac{k(k+1)}{2}(1-\epsilon)})}{\log (N^{\frac{k(k+1)}{2}})}
=1-\epsilon,
\end{align*}
whereas by Lemma \ref{Hd-lemma}(b),
\begin{align*}
\dim_{\mathbb H}(S) 
&\geq \liminf_{k \rightarrow \infty} \log(P_{k}/N_k)/\log(M_{k-1})\\
&\geq \liminf_{k \rightarrow \infty}
\frac{\log (2^{-k}N^{\frac{k(k+1)}{2}(1-\epsilon)-k})}{\log
(N^{\frac{k(k-1)}{2}})}=1-\epsilon.
\end{align*}
Hence $S$ has dimension $1-\epsilon$.

Next, we verify Lemma \ref{z-lemma1}(c).
Plugging (\ref{z-e2}) into (\ref{critter-cor}), we see 
after some more algebra that
\begin{equation*}
\begin{split}
\sup_{\Omega \subseteq [0,1]} \frac{\|\Phi_k^{\ast}\mathbf 1_{\Omega}\|_{n}}{|\Omega|^{\frac{n-1}{n}}} &
\leq 
C (n!\, B)^{1/n} 2^{k(1+\frac{3}{2n})}
N^{\frac{k(k+1)}{2n}(-\half +\epsilon(n-\half))+(k+1)\epsilon}
\\
&\quad \quad \times
\Bigl[ \ln(4^n n!\, B)+(k+1)(k+2)Ln\ln N  \Bigr]^{1/2n}.
\end{split}
\end{equation*}  
This is majorized by $C(n)2^{-\eta(n)k}$ with $\eta(n)=\frac{1+\epsilon}{2n}-\epsilon$.
Note that $\eta(n)>0$ if and only if $\epsilon(n-\half)<\half$, i.e.  
\begin{equation}\label{z-e6}
\epsilon<\frac{1}{2n-1},\hbox{ or }n<\half + \frac{1}{2\epsilon}=q_\epsilon.
\end{equation}
Let $n_1=n_1(\epsilon)$ be the largest even integer such that (\ref{z-e6}) holds, and
let $n_2=n_1+2$.  Interpolating between the estimates for $n_1$ and $n_2$, we
get that 
$$
\sup_{\Omega \subseteq [0,1]} \frac{\|\Phi_k^{\ast}\mathbf 1_{\Omega}\|_{q_0}}{|\Omega|^{\frac{q_0-1}{q_0}}} 
\leq C(q_0)2^{-\eta(q_0)k}
$$
with $\eta(q_0)>0$ for all $q_0<q_\epsilon$.

For part (\ref{partd}), we check as before that 
$$
2^{6k}\frac{\ln (M_k)}{N_{k+1}^{1-\epsilon_{k+1}}}
\leq \frac{ 2^{6k+1}\,k(k+1)\ln N}{N^{(k+1)(1-\epsilon)}}
\leq \frac{1}{32}
$$
for all $k$, if $N$ was chosen large enough. This proves (\ref{partd}) and establishes the existence of $\mu$.


\section{Extension to the unrestricted operator} \label{sec-scales}
It remains to prove the statements for the unrestricted maximal operators $\tilde{\mathcal M}^a$ and $\tilde{\mathfrak M}^a$ claimed in Theorem \ref{thm-main} (\ref{thm-main-unrestricted1}) and Theorem \ref{thm-main3}(\ref{thm-main3-unrestricted1}). Obtaining bounds for global maximal operators using known bounds for single-scale ones is a common theme in the harmonic analysis literature, often involving interpolation and scaling. In this section we present these arguments with the necessary modifications for our problem. The proof naturally splits into two cases $q \geq 2$ and $q < 2$, which are handled in Propositions \ref{scales-prop5} and \ref{p-less-2} respectively. The former follows an approach closely related to \cite{bourg-86}, \cite{schlag-thesis}. The proof for $p = q < 2$ is due to Andreas Seeger, who also indicated to us prior work in this direction \cite{nagel-stein-wainger78}, \cite{christ88-unpublished}. Proposition \ref{p-less-2} combines his argument with interpolation techniques used in a similar setting in \cite{gsw99-pams}. 

We remark that the scaling arguments below are quite general and apply
to any sequence $S_k$ as described in Section 2 subject to the bounds on $\calm_k$ and
(in Lemma \ref{scales-lemma2}) the
subexponential growth of $N_k$.  In other words, we will not be invoking the
probabilistic arguments of Section \ref{section-random}.

Recall the definitions (\ref{max-e2}), (\ref{max-e1}), (\ref{max-e102}), (\ref{max-e104}) and (\ref{max-e101}) of $\tilde{\mathcal M}$, $\tilde{\mathfrak M}$, $\tilde{\mathcal M}^a$, $\tilde{\mathfrak M}^a$  and $\mathcal M_k$ respectively.   
Denote by $A_r[k]$ the averaging operator associated to $\phi_k$:
\begin{equation}\label{average-def}
A_r[k]f(x) = \int f(x + ry) \phi_k(y) \, dy, \quad \text{ where } \quad \phi_k = \frac{1}{|S_k|} \mathbf 1_{S_k}. 
\end{equation}
The main results in this section are the following.

\begin{proposition}\label{scales-prop5}
Fix two exponents $p,q$ satisfying $1 < p \leq q < \infty$, $q \geq 2$. Assume that for some $C>0$ and $\eta_0>0$ we have the estimate
\begin{equation}\label{scales-e3} 
\|\calm_k f\|_q \leq C 2^{-\eta_0 k}\|f\|_p 
\end{equation}
for all $f$ supported on $[0,1]$.  Assume furthermore that $N_k$ have been chosen as
in (\ref{z-e1}) or (\ref{z-e2}).  Then     
$\tilde{\calm}^a$ is bounded from $L^p(\rr)$ to $L^q(\rr)$, with $a = \frac{1}{p} - \frac{1}{q}$. 
\end{proposition}

\begin{proposition} \label{p-less-2}
Suppose that there exists $\epsilon \in [0, \frac{1}{3})$ such that (\ref{scales-e3}) holds for all functions $f$ supported in $[0,1]$ and all exponents $(p,q)$ satisfying 
\begin{equation} \label{pq-conditions}
1 < p \leq q \leq 2, \quad \frac{1+\epsilon}{1 - \epsilon} < p < \infty, \quad 1 < q < \frac{1 - \epsilon}{2 \epsilon}p. 
\end{equation}  
Then $\tilde{\mathcal M}^a$ is bounded from $L^p(\mathbb R)$ to $L^q(\mathbb R)$ for all such $(p,q)$, with $a = \frac{1}{p} - \frac{1}{q}$.
\end{proposition} 
{\em{Remark: }} Despite the formal similarity, it is worth noting the distinction between the statements of the two propositions. In Proposition \ref{scales-prop5}, the assumption (\ref{scales-e3}) is for a {\em{fixed}} $(p,q)$, and the conclusion is the estimate for the global operator with the same $(p,q)$. In contrast, for Proposition \ref{p-less-2}, the hypothesis (\ref{scales-e3}) is for {\em{all}} $(p,q)$ in the domain (\ref{pq-conditions}). 

\begin{proof}[Conclusion of the proofs of Theorems \ref{thm-main}(\ref{thm-main-unrestricted1}) and \ref{thm-main3}(\ref{thm-main3-unrestricted1})]
Assuming the two propositions, the unrestricted maximal bounds are proved as follows.
It suffices to prove the bounds on $\tilde{\calm}^a$.  Suppose first that we are
in the one-dimensional case (\ref{z-e1}).  Then (\ref{squirrel2}) asserts that
the hypotheses of both Propositions \ref{scales-prop5} and \ref{p-less-2} hold (the latter with $\epsilon = 0$), hence so do the conclusions. In the lower-dimensional case (\ref{z-e2}), the same argument shows that $\tilde{\mathcal M}^a$ is bounded from $L^p(\mathbb R)$ to $L^q(\mathbb R)$ whenever $p,q$ obey (\ref{z-e11}) with $p\leq q$.
\end{proof}

\subsection{Scaling arguments}
The proofs of both Propositions \ref{scales-prop5} and \ref{p-less-2} use the Haar decomposition of a function $f$ and the relation between the averaging operators $A_r[k]$ for various scales of the dilation parameter $r$. We record the necessary facts in the following sequence of lemmas. Following \cite{bourg-86}, we denote by $\mathcal D_s$ the $\sigma$-algebra generated by dyadic intervals of length $2^{-s}$, and by $\mathbb E_s$ the corresponding conditional expectation operators, i.e., $\mathbb E_s(f) = \mathbb E(f|\mathcal D_s)$. We also set \begin{equation} \Delta_s f = \mathbb E_{s+1}(f) - \mathbb E_s(f). \label{def-Deltas} \end{equation}

\begin{lemma}\label{scales-lemma3}
Let $1<p, q<\infty$ and $\eta_0>0$. 
Suppose that (\ref{scales-e3}) holds for all $f$ supported on $[0,1]$.
Then there exists $\eta > 0$ such that 
\begin{equation}\label{scales-e4}
\| \mathcal M f \|_q \leq C2^{-\eta \sqrt{s}} \|f\|_p 
\end{equation}  
for all functions $f \in L^p(\mathbb R)$ satisfying $\mathbb E_s(f) = 0$, $s \geq 0$.  
\end{lemma}  
\begin{proof}
For any $f$ supported in $[0,1]$, 
$$
|A_r[k](f)|\leq \mathcal N f(x)+ \sum_{m=0}^{k-1} \calm_m |f|,
\quad \text{ hence } \quad \mathcal M f \leq \mathcal N f+ \sum_{k=0}^{\infty} \mathcal M_k|f|, 
$$
where $\mathcal N$ was defined at the beginning of Subsection \ref{subsec-lin-disc}.
Therefore \begin{equation}\label{scales-e5}
\|\mathcal M f \|_q\leq \sum_{k=0}^\infty \|\calm_kf\|_q.
\end{equation}
The right side is clearly summable by (\ref{scales-e3}).  To obtain decay as required in (\ref{scales-e4}),
we will use the assumption that $\mathbb E_s f=0$ to improve the estimate on the terms with
$k\leq k_0$, where $k_0$ will be determined shortly.  We have
\begin{align*}
\calm_kf(x)&= \sup_{1 < r < 2}\Big|\int f(x+ry)(\phi_{k+1}(y)-\phi_k(y))dy\Big|\\
&\leq \Big|\int f(x+ry)\phi_k(y)dy\Big|+\Big|\int f(x+ry)\phi_{k+1}(y)dy\Big|.
\end{align*}
Suppose that $2^{-s}<\delta_{k+1}$, and consider the term with $\phi_k$ first.
Each of the $\delta_{k}$-intervals $\{ I_{k}(\mathbf i) : \kappa_{k}(\mathbf i) = 1\}$ 
in the support of $\phi_{k}$ can be written as a union of 
some number of dyadic $2^{-s}$-intervals together with two intervals $J_1(\mathbf i, s)$ 
and $J_2(\mathbf i, s)$ of length at most $2^{-s}$, one at each end of $I_k(\mathbf i)$.  
Since $f$ integrates to 0 on each dyadic interval, the only non-zero contribution comes from
the intervals $J_j(\mathbf i, s)$.
By H\"older's inequality, we see that
\begin{align*}
\Bigl|\int f(x+ry) \phi_{k}(y)\, dy \Bigr| &= \frac{1}{P_{k}\delta_{k}} \Bigl| \sum_{j=1}^2 \sum_{\kappa_k(\mathbf i)=1} \int_{J_j(\mathbf i, s)}f(ry)dy \Bigr| \\
&\leq\frac{1}{P_k\delta_k}\|f\|_p(2 P_k\cdot 2^{-s})^{\frac{1}{p'}}\\
&=\frac{1}{P_k^{1/p}\delta_k}2^{1 -\frac{s}{p'}}\|f\|_p
\leq M_k 2^{1 -\frac{s}{p'}}\|f\|_p
\ .\\
\end{align*}
The term with $\phi_{k+1}$ is estimated similarly.  Taking the $L^q$ norm of the left side and 
using the fixed compact support of $\mathcal M_k f$, we see that
\begin{align*}
\|\calm_k f\|_q\leq M_k 2^{1 -\frac{s}{p'}}\|f\|_p\leq  C2^{-\frac{s}{p'}}N^{k(k+3)/2} \|f\|_p,
\end{align*}
where we used (\ref{z-e1}) and (\ref{z-e2}) at the last step. 
Let $k_0\approx c\sqrt{s}$ with a small enough constant, then for all $k\leq k_0$ we have
$N^{k(k+3)/2}2^{-{s}/(2p')} \leq C$, so that 
\[ \|\calm_k f\|_q  \leq C2^{-s/(2p')}\|f\|_p \quad \text{ for $k\leq k_0$.} \]
We now use this along with (\ref{scales-e3}) to estimate the right side of (\ref{scales-e5}):
\begin{align*}
\sum_{k=1}^\infty \|\calm_kf\|_q&=\sum_{k=1}^{k_0} \|\calm_kf\|_q+\sum_{k>k_0} \|\calm_kf\|_q\\
&\leq C k_02^{-s/(2p')}\|f\|_p+ C \sum_{k>k_0}2^{-\eta_0 k}\|f\|_p\\
& \leq C 2^{-s/(4p')}\|f\|_p+ C 2^{-c \eta_0 \sqrt{s}}\|f\|_p \leq C 2^{-\eta \sqrt{s}} \|f\|_p,
\end{align*}
as claimed in (\ref{scales-e4}). This proves the result for functions $f$ supported in [0,1]. The extension to a general $f$ is achieved by a ``disjointness of support''  argument identical to the one given in Lemma \ref{spatial-lemma} and is left to the reader.     
\end{proof}

We will also need the following rescaled version of (\ref{scales-e4}).

\begin{lemma}\label{scales-lemma2}
Suppose that (\ref{scales-e4}) holds for all functions $f \in L^p(\mathbb R)$
satisfying $\mathbb E_s(f) = 0$ for some $s \geq 0$. Then for any $m \in \mathbb Z$ and all $f \in L^p(\rr)$, 
\begin{equation}\label{scales-e1}
\Big\|\sup_{\begin{subarray}{c} k \geq 1 \\ 1 \leq r2^m \leq 2 \end{subarray}} 
\left| A_r[k](\Delta_{s+m}f) \right| \Big\|_q
\leq C\cdot 2^{ma-\eta \sqrt{s}}\|\Delta_{s+m}f\|_p.
\end{equation}
Here $\Delta_s f$ is as in (\ref{def-Deltas}).
\end{lemma}
\begin{proof}
Let $u=r2^m$, so that  $1\leq u\leq 2$. We have
\begin{equation}\label{scales-e2}
\begin{split}
A_r[k]f(x)&=\int f(x+ry)\phi_k(y)dy\\
&=\int f(x+2^{-m}uy)\phi_k(y)dy\\
&=\int f(2^{-m}(2^mx+uy))\phi_k(y)dy\\
&=A_u[k](f^{(m)})(2^mx),
\end{split}
\end{equation}
where $f^{(m)}(\cdot)=f(2^{-m}\cdot)$.  Note also that $(\Delta_{s+m}f)^{(m)}=\Delta_{s+m}(2^{-m}\cdot)$
is constant on dyadic $2^{-s}$-intervals, i.e. $\mathbb E_s((\Delta_{s+m}f)^{(m)})  =0$.
By (\ref{scales-e4}), we have
\begin{equation*}
\begin{split}
\Bigl\|
\sup_{\begin{subarray}{c}k \geq 1\\ 1\leq r 2^m\leq 2 \end{subarray}} 
\left|A_r[k](\Delta_{s+m}f)\right|\Bigr\|_q
&= 
\Bigl\|\sup_{\begin{subarray}{c}k \geq 1\\ 1\leq u\leq 2 \end{subarray}} 
\left|\big(A_u[k](\Delta_{s+m}f)^{(m)}\big)(2^m\cdot)\right|\Bigr\|_q\\
&=  2^{-m/q} \|\mathcal M (\Delta_{s+m} f)^{(m)}\|_q \\
&\leq C
2^{-m/q}\, 2^{-\eta \sqrt{s}} \left\|(\Delta_{s+m}f)^{(m)}\right\|_p\\
&= C
2^{-m/q}\, 2^{-\eta \sqrt{s}}\, 2^{m/p} \left\|\Delta_{s+m}f\right\|_p\\
&= C
2^{ma}\, 2^{-\eta \sqrt{s}}\, \left\|\Delta_{s+m}f\right\|_p\ .\\
\end{split}
\end{equation*}

\end{proof}

Finally, we need a technical lemma.

\begin{lemma} Given any $0 \leq a < 1$, there is a constant $C = C(a)$ such that for any $m \in \mathbb Z$ and all $f \in L^p(\mathbb R)$, 
\begin{align*}\label{cranberry}
& \sup_{\begin{subarray}{c}k \geq 1\\ 1 \leq r2^m \leq 2 \end{subarray}} 
r^a \bigl| A_r[k]\mathbb E_m f(x) \bigr| \leq Cf^{\ast}(x), \text{ where } \\
& f^{\ast}(x) := \sup_{r>0}r^{a-1}\int_{|y|\leq r}|f(x-y)|dy.
\end{align*} 
The mapping $f \mapsto f^{\ast}$ is bounded from $L^p(\mathbb R) \rightarrow L^q(\mathbb R)$ for all $1 < p \leq q \leq \infty$ for which $a=\frac{1}{p} - \frac{1}{q}$. \label{first-term-lemma}
\end{lemma} 
\begin{proof} 
Since $S_k \subseteq [1,2]$ and $r \leq 2^{-m+1}$, the set $x+ r S_k$ is contained
in an interval $J$ centered at $x$ of length $2^{-m+3}$.  Observe that
$J$ can be covered by at most 10 dyadic $2^{-m}$-intervals $J_i$.  On each
$J_i$, we have ${\mathbb E}_m(f)\equiv \lambda_i$, where $\lambda_i$ is the average
of $f$ on $J_i$.  Since $A_r[k]\mathbb E_m f(x)$ is a convex linear combination of the
$\lambda_i$-s, it suffices to prove that $r^a \lambda_i\leq f^*(x)$.
But this follows from
$$
r^a \lambda_i=\frac{r^a}{|J_i|} \int_{J_i} |f(y)| dy
\leq \frac{10r^a}{|J|} \int_{J'} |f(y)| dy
\leq \frac{C}{|J'|^{1-a}} \int_{J'} |f(y)| dy
\leq f^\ast (x),
$$
where $J'$ is an interval of length $2|J|$ centered at $x$ so that 
$J\subset \bigcup J_i\subset J'$.

If $p=q$, then $a=0$ and $f^{\ast}$ is simply the Hardy-Littlewood maximal function of $f$, which is 
bounded  on all $L^p$ for $p>1$.  If on the other hand $1<p<q \leq \infty$, then $0<a<1$ and
$$
f^{\ast}(x)=\sup_{r>0}r^{a-1}\int_{|x-z|\leq r}|f(z)|dz
\leq\int_{-\infty}^\infty \frac{|f(z)|}{|x-z|^{1-a}}dz.
$$
Since $f\in L^p(\rr)$ and $|z|^{a-1}$ is in weak $L^{\frac{1}{1-a}}(\rr)$, it follows by Young's inequality that the mapping $f\to f^{\ast}$ is bounded from $L^p(\mathbb R)$ to $L^q(\mathbb R)$ with $1+\frac{1}{q}=\frac{1}{p} +(1-a)$, as claimed.
\end{proof}

\subsection{Proof of Proposition \ref{scales-prop5}}
Given $m \in \mathbb Z$ such that $2^{-m} \leq r \leq 2^{-m+1}$, we write $f = \mathbb E_m(f) + \sum_{s \geq m} \Delta_s(f)$, where  
$\Delta_s(f)$ is defined as in (\ref{def-Deltas}). Therefore 
\begin{equation}  \label{haar-decomp}
A_r[k](f) = A_r[k](\mathbb E_m f) + \sum_{s \geq m} A_r[k](\Delta_s f),  
\end{equation}
so that 
\begin{equation} \tilde{\mathcal M}^af \leq \sup_{m \in \mathbb Z}\sup_{\begin{subarray}{c} k \geq 1 \\ 1 \leq r2^m  \leq 2  \end{subarray}}r^a \Biggl[ | A_r[k]\mathbb E_m(f)(x)| +   \Bigl| \sum_{s \geq m} A_r[k](\Delta_s f) \Bigr| \Biggr]. \label{At}  \end{equation}
The first term is bounded from $L^p \rightarrow L^q$ by Lemma \ref{first-term-lemma}. Turning our attention to the second term of (\ref{At}), 
it suffices to prove that
\begin{equation}\label{scales-e10}
\Biggl\|\sup_{m\in\zz}\sup_{\begin{subarray}{c}k \geq 1\\ 1\leq r 2^m\leq 2 \end{subarray}} 
2^{-ma}\Bigl| \sum_{s \geq m} A_r[k](\Delta_s f)\Bigr| \Biggr\|_q\leq C\|f\|_p\ .
\end{equation}
We write
\begin{equation*}
\begin{split}
\sup_{m\in\zz}\sup_{\begin{subarray}{c}k \geq 1\\ 1\leq r2^m\leq 2 \end{subarray}} 
2^{-ma}\Bigl| \sum_{s \geq m} A_r[k](\Delta_s f)\Bigr|
&\leq \Biggl[\sum_{m\in\zz}2^{-maq}
\sup_{\begin{subarray}{c}k \geq 1\\ 1\leq r2^m\leq 2 \end{subarray}} \Bigl|\sum_{s\geq m} 
A_r[k](\Delta_s f)\Bigr|^q\Biggr]^{\frac 1q}\\
&\leq \Biggl[\sum_{m\in\zz}2^{-maq}\Bigl(\sum_{s\geq m} 
\sup_{\begin{subarray}{c}k \geq 1\\ 1\leq r2^m\leq 2 \end{subarray}} 
\left|A_r[k](\Delta_s f)\right|\Bigr)^q\Biggr]^{\frac 1q}.
\end{split}
\end{equation*}
Taking the $L^q$-norms of both sides, then using  Lemma \ref{scales-lemma2} (whose hypothesis in turn is true by Lemma \ref{scales-lemma3}),
we see that the left side of (\ref{scales-e10}) is
bounded by
\begin{equation*}
\begin{split}
\Biggl(\sum_{m\in\zz}2^{-maq} &\Bigg\| \sum_{s\geq m} 
\sup_{\begin{subarray}{c}k \geq 1\\ 1\leq r2^m\leq 2 \end{subarray}} 
|A_r[k](\Delta_s f)|\Biggr\|_q^q\Biggr)^{\frac 1q} \\
&\leq \Biggl(\sum_{m\in\zz}2^{-maq}\Biggl[\sum_{s\geq m} 
\Bigl\|\sup_{\begin{subarray}{c}k \geq 1\\ 1\leq r2^m\leq 2 \end{subarray}} 
\bigl|A_r[k](\Delta_{s}f)\bigr|\Bigr\|_q\Biggr]^q\Biggr)^{\frac 1q}\\
&\leq C \Bigl(\sum_{m\in\zz}\Bigl[\sum_{s\geq m} 2^{-\eta \sqrt{s-m}}
\left\|\Delta_s f\right\|_p\Bigr]^q\Bigr)^{\frac 1q}.
\end{split}
\end{equation*}
The last line is the $\ell^q$-norm of the convolution of the discrete functions ${\bf 1}_{m\geq 0}
2^{-\eta \sqrt{m}}$ and $\|\Delta_m f\|_p$.  Applying Young's inequality with $s=\max(p,2)$ and
$\frac{1}{s}+\frac{1}{r}=1+\frac{1}{q}$, we bound it by
\begin{equation}\label{scales-e20}
\Bigl(\sum_{m \geq 0} 2^{-\eta \sqrt{m} r}\Bigr)^{\frac 1r}
\Bigl(\sum_{m \in \mathbb Z} \|\Delta_m f \|_p^s\Bigr)^{\frac 1s}
\leq C \Bigl(\sum_m\|\Delta_m f\|_p^s\Bigr)^{\frac 1s}\ .
\end{equation}
It remains to show that 
\begin{equation}\label{scales-e21}
\Bigl(\sum_m\|\Delta_m f \|_p^s\Bigr)^{\frac 1s} \leq C \|f\|_p\ .
\end{equation}
Suppose first that $p\geq 2$, so that $s=p$.  Then the claim is trivial for $p=\infty$, and 
for $p=2$ it follows from the orthogonality of $\Delta_m f$.  By interpolation, this implies (\ref{scales-e21}) for all $p\in[2,\infty)$. Assume next that $1<p<2$, so that $s=2$.  Then  
\begin{align*}
\Bigl(\sum_{m \in \mathbb Z} \|\Delta_m f\|_p^2\Bigr)^{\frac 12}
\leq \Bigl \| \bigl( \sum_{m \in \mathbb Z} |\Delta_m f|^2 \bigr)^{\frac{1}{2}} \Bigr \|_p \leq C_p ||f||_p,
\end{align*}
where the first step follows from the generalized Minkowski inequality and the second from Littlewood-Paley theory. This proves the claim (\ref{scales-e21}).

\qed

\subsection{Proof of Proposition \ref{p-less-2}}
As indicated in the remark following Proposition \ref{p-less-2}, the conclusion is immediate from Proposition \ref{scales-prop5} if $q = 2$. Fix $\epsilon \in [0,\frac{1}{3})$ and exponents $(p,q)$, $q < 2$ satisfying (\ref{pq-conditions}). We denote by $C(p,q;R)$ the norm of the linear operator  
\begin{equation} 
f \mapsto \bigl\{2^{-ma} A_{r2^{-m}}[k]f : -R \leq m \leq R, \, 1 \leq r \leq 2, \,  k \geq 1 \bigr\}, 
\end{equation} 
mapping $L^p(\mathbb R)$ to $L^q(\ell_m^{\infty}L_r^{\infty} \ell_k^{\infty})$, where $a = \frac{1}{p} - \frac{1}{q}$. In other words, $C(p,q;R)$ is the best constant such that the following inequality holds for all $f$: 
\begin{equation} \label{best-constant}
\bigl\| \sup_{-R \leq m \leq R} \sup_{1 \leq r \leq 2} \sup_{k \geq 1} 2^{-ma} \bigl|A_{r2^{-m}}[k]f \bigr| \bigr\|_q \leq C(p,q;R) \|f\|_p. 
\end{equation} 
We first ensure that $C(p,q;R)$ is well-defined. The hypothesis (\ref{scales-e3}) implies (\ref{z-e10}) after summing in $k$, hence the inequality in (\ref{z-e10}) continues to hold for all $f \in L^p(\mathbb R)$ by Lemma \ref{spatial-lemma}. By the scaling argument in (\ref{scales-e2}), this implies that for every fixed $m \in \mathbb Z$, 
\begin{equation} \bigl\| 2^{-ma} \sup_{1 \leq r \leq 2} \sup_{k \geq 1} \bigl|A_{r2^{-m}}[k]f \bigr| \bigr\|_q \leq \|\mathcal M\|_{p \rightarrow q} \|f\|_p, \quad f \in L^p(\rr). \label{fixedm} \end{equation}
Thus we already have the trivial bound $C(p,q;R) \leq R \| \mathcal M \|_{p \rightarrow q}$. Our goal is to show that for each $p,q$ in the indicated range, $C(p,q;R)$ is bounded 
uniformly in $R$:
\begin{equation} \label{goal} C(p,q;R) = O_{p,q}(1).\end{equation}
This would imply the conclusion of the proposition, since the left hand side of (\ref{best-constant}) converges as $R \rightarrow \infty$ to a limit that is bounded above and below by positive constant multiples of $\| \tilde{\mathcal M}^a f \|_q$. The convergence is justified by the monotone convergence theorem, which applies because the the operators $A_r[k]$ are non-negative and the functions $f$ can be chosen to be non-negative.   

In order to prove (\ref{goal}) we fix two other auxiliary exponents $(p_1, q_1)$ and $(p_2, q_2)$ obeying (\ref{pq-conditions}), such that $p_1 < p < p_2$, $q_2 = 2$, and the points $\{(\frac{1}{p}, \frac{1}{q}),(\frac{1}{p_1}, \frac{1}{q_1}),(\frac{1}{p_2}, \frac{1}{2})\}$ are collinear. The following lemma provides an essential interpolation ingredient of the proof.
\begin{lemma} \label{seeger-interpolation-lemma}
Given any sequence of functions $\{g_m : -R \leq m \leq R \}$, define 
\[ \mathcal T_a(\{ g_m \}) = \bigl\{2^{-ma} A_{r2^{-m}}[k]g_m : -R \leq m \leq R, \, 1 \leq r \leq 2, \, k \geq 1 \bigr\}. \]
\begin{enumerate}[(a)]
\item For any $(p,q)$ obeying (\ref{pq-conditions}), the operator $\mathcal T_{a(p,q)} : L^{p}_x \ell_m^{\infty} \rightarrow L_x^q \ell_m^{\infty} L_r^{\infty} \ell_k^{\infty}$ has norm bounded by $C(p,q;R)$, with $a(p,q) = \frac{1}{p} - \frac{1}{q}$.   \label{seeger-parta}
\item For any $(p_1, q_1)$ obeying (\ref{pq-conditions}), there is a constant $K_1 = K_1(p_1, q_1)$ independent of $R$ such that the operator $\mathcal T_{a(p_1,q_1)} : L^{p_1}_x \ell_m^{p_1} \rightarrow L_x^{q_1} \ell_m^{p_1} L_r^{\infty} \ell_k^{\infty}$ is bounded with norm $\leq K_1$.  \label{seeger-partb}
\item If $p_1 < p$, the norm of the operator $\mathcal T_{a(p_3,q_3)} : L^{p_3}_x \ell_m^{2} \rightarrow L_x^{q_3} \ell_m^{2} L_r^{\infty} \ell_k^{\infty}$ is bounded by $K_1^{\frac{p_1}{2}} C(p, q;R)^{1 - \frac{p_1}{2}}$; i.e.,  \label{seeger-partc}
\begin{multline} \label{interp-p3}
\Bigl\| \Bigl(\sum_{m=-R}^{R} \bigl[ 2^{-ma} \sup_{1 \leq r \leq 2} \sup_{k \geq 1} \bigl|A_{r2^{-m}}[k]g_m \bigr| \bigr]^2  \Bigr)^{\frac{1}{2}} \Bigr\|_{q_3} \\ \leq K_1^{\frac{p_1}{2}} C(p, q; R)^{1 - \frac{p_1}{2}} \Bigl\| \Bigl( \sum_m |g_m|^2 \Bigr)^{\frac{1}{2}}\Bigr\|_{p_3}.
\end{multline} 
Here $\frac{1}{p_3} = \frac{1}{p} + \frac{p_1}{2}(\frac{1}{p_1} - \frac{1}{p})$, and $(\frac{1}{p_i}, \frac{1}{q_i})$, $i=1,2, 3$ are collinear.  
\end{enumerate}
\end{lemma} 
\begin{proof}
Part (\ref{seeger-parta}) is a consequence of the non-negativity of $A_r[k]$ combined with (\ref{best-constant}):
\begin{multline*}
\Bigl\| \sup_{-R \leq m \leq R} \sup_{1 \leq r \leq 2} \sup_{k \geq 1} 2^{-ma}\bigl| A_{r2^{-m}}[k]g_m \bigr| \Bigr\|_q \\ \leq  \Bigl\| \sup_{-R \leq m \leq R} \sup_{1 \leq r \leq 2} \sup_{k \geq 1} 2^{-ma}\bigl| A_{r2^{-m}}[k] \bigl|\sup_{j} g_j \bigr| \bigr| \Bigr\|_q \leq C(p,q;R) \bigl\| \sup_m|g_m| \bigr\|_p.
\end{multline*} 
For part (\ref{seeger-partb}), by the triangle inequality in $L^{q_1/p_1}$ applied to 
functions $|G_m|^{p_1}$ we have
\[ \bigl\| \bigl( \sum_m |G_m|^{p_1}\bigr)^{\frac{1}{p_1}}\bigr\|_{q_1} \leq  \Bigl(\sum_m \bigl\| G_m \bigr\|_{q_1}^{p_1} \Bigr)^{\frac{1}{p_1}} \quad \text{ since } p_1 \leq q_1.\]
Using this with $G_m = 2^{-ma} \sup_{1 \leq r \leq 2} \sup_{k \geq 1} A_{r2^{-m}}[k]g_m$, we find 
\begin{align*}
\Bigl\| \Bigl(\sum_m \Bigl| 2^{-ma} \sup_{1 \leq r \leq 2} \sup_{k \geq 1}
&\bigl|A_{r2^{-m}}[k]g_m \bigr| \Bigr|^{p_1} \Bigr)^{\frac{1}{p_1}} \Bigr\|_{q_1} \\
&\leq \Bigl( \sum_m \Bigl\| 2^{-ma} \sup_{1 \leq r \leq 2} \sup_{k \geq 1}
A_{r2^{-m}}[k]g_m \Bigr\|_{q_1}^{p_1}\Bigr)^{\frac{1}{p_1}} \\
&\leq \|\mathcal M\|_{p_1 \rightarrow q_1} \Bigl(\sum_m \| g_m \|_{p_1}^{p_1} \Bigr)^{\frac{1}{p_1}} \\ &= K_1 \Big\| \bigl(\sum_m |g_m|^{p_1} \bigr)^{\frac{1}{p_1}}\Bigr\|_{p_1}, 
\end{align*}   
where we have used (\ref{fixedm}) with $(p_1, q_1)$ at the second step. This gives the conclusion with $K_1 = \| \mathcal M \|_{p_1 \rightarrow q_1}$. Part (\ref{seeger-partc}) now follows by complex interpolation of the family of operators $\mathcal T_a$ between the spaces in parts (\ref{seeger-parta}) and (\ref{seeger-partb}). The interpolation works because $p_1 < 2$, so that $\ell^2$ is intermediate between $\ell^{p_1}$ and $\ell^{\infty}$.  
\end{proof} 
\begin{proof}[Conclusion of the proof of Proposition \ref{p-less-2}]
In order to prove (\ref{goal}), we start again with the Haar decomposition of the function $f$, so that (\ref{haar-decomp}) holds. Thus 
\begin{equation} \label{Cpqr-est} 
\begin{aligned}
\sup_{-R \leq m \leq R} \sup_{1 \leq r \leq 2} \sup_{k \geq 1} &2^{-ma} \bigl|A_{r2^{-m}}[k]f \bigr| \\ &\leq \sup_{-R \leq m \leq R} \sup_{1 \leq r \leq 2} \sup_{k \geq 1} 2^{-ma} \bigl|A_{r2^{-m}}[k] \mathbb E_m f \bigr| \\ &+ \sum_{s \geq 1} \sup_{-R \leq m \leq R} \sup_{1 \leq r \leq 2} \sup_{k \geq 1} 2^{-ma} \bigl| A_{r2^{-m}}[k]\bigl( \Delta_{s+m}f\bigr)\bigr|.
\end{aligned}  
\end{equation}
As before, the first term on the right is bounded pointwise by $f^{\ast}$, and therefore bounded from $L^p \rightarrow L^q$ with norm independent of $R$ by Lemma \ref{first-term-lemma}. We estimate the $L^q$ norms of the summands in (\ref{Cpqr-est}) as follows. On one hand, (\ref{interp-p3}) with $g_m = \Delta_{s+m}f$ implies 
\begin{equation} \label{sq-p3}
\begin{aligned}
\Bigl\| \Bigl(\sum_{m=-R}^{R} \bigl[ 2^{-ma} \sup_{1 \leq r \leq 2} &\sup_{k \geq 1} \bigl|A_{r2^{-m}}[k](\Delta_{s+m}f) \bigr| \bigr]^2  \Bigr)^{\frac{1}{2}} \Bigr\|_{q_3} \\ &\leq K_1^{\frac{p_1}{2}} C(p_1, q_1; R)^{1 - \frac{p_1}{2}} \Bigl\| \Bigl( \sum_m |\Delta_{s+m}f|^2 \Bigr)^{\frac{1}{2}}\Bigr\|_{p_3} \\ &\leq K_1^{\frac{p_1}{2}} C(p_1, q_1; R)^{1 - \frac{p_1}{2}} \|f\|_{p_3},
\end{aligned} 
\end{equation}
where the last step is a consequence of the Littlewood-Paley inequality. 
On the other hand, for all $ s \geq 1$, 
\begin{equation} \label{sq-p2}
\begin{aligned}
\Bigl\| \Bigl(\sum_{m=-R}^{R} \bigl[ 2^{-ma} \sup_{1 \leq r \leq 2} &\sup_{k \geq 1} \bigl|A_{r2^{-m}}[k](\Delta_{s+m}f) \bigr| \bigr]^2  \Bigr)^{\frac{1}{2}} \Bigr\|_{2} \\ &= \Bigl(\sum_m \bigl\| 2^{-ma} \sup_{1 \leq r \leq 2} \sup_{k \geq 1} \bigl|A_{r2^{-m}}(\Delta_{s+m}f) \bigr| \bigr\|_2^2 \Bigr)^{\frac{1}{2}} \\ &\leq C 2^{-\eta \sqrt{s}} \Bigl[\sum_{m} \|\Delta_{s+m}f \|_{p_2}^2 \Bigr]^{\frac{1}{2}} \\ &\leq C 2^{-\eta \sqrt{s}} \Bigl\| \Bigl( \sum_m |\Delta_{m+s}f|^2\Bigr)^{\frac{1}{2}}\Bigr\|_{p_2} \\ &\leq C 2^{-\eta \sqrt{s}} \|f \|_{p_2}. 
\end{aligned}
\end{equation} 
Here $\eta$ is a positive constant (independent of $m$) whose existence is guaranteed by Lemma \ref{scales-lemma2}. The third step above uses the generalized Minkowski inequality (since $p_2 \leq 2$) and the fourth follows from Littlewood-Paley theory. 

Since $p_3 < p < p_2$ and $\{(\frac{1}{p_3}, \frac{1}{q_3}), (\frac{1}{p}, \frac{1}{q}), (\frac{1}{p_2}, \frac{1}{q_2}) \}$ are collinear, we can interpolate between (\ref{sq-p3}) and (\ref{sq-p2}) to obtain $0 < \theta < 1$ such that 
 \begin{align*}
\Bigl\|\sup_{-R \leq m \leq R} \sup_{1 \leq r \leq 2} \sup_{k \geq 1} &2^{-ma} \bigl| A_{r2^{-m}}\bigl( \Delta_{s+m}f\bigr)\bigr| \Bigr\|_q  \\ &\leq \Bigl\| \Bigl(\sum_{m=-R}^{R} \bigl[ 2^{-ma} \sup_{1 \leq r \leq 2} \sup_{k \geq 1} \bigl|A_{r2^{-m}}[k]\Delta_{s+m}f \bigr| \bigr]^2  \Bigr)^{\frac{1}{2}} \Bigr\|_{q} \\ &\leq \bigl(K_1^{\frac{p_1}{2}} C(p, q; R)^{1 - \frac{p_1}{2}} \bigr)^{\theta} \bigl( C 2^{-\eta \sqrt{s}}\bigr)^{1 - \theta} \|f \|_{p}.
\end{align*} 
The right hand side is summable in $s$. In summary, we have obtained the following estimate for  the $L^q$ norm of the left hand side of (\ref{Cpqr-est}): there is a large constant $K$ and $0 < \rho < 1$ such that
\[ \big\| \sup_{-R \leq m \leq R} \sup_{1 \leq r \leq 2} \sup_{k \geq 1} 2^{-ma} \bigl|A_{r2^{-m}}[k]f \bigr| \bigr\|_q \leq K(1 + C(p,q;R)^{\rho}) \|f\|_p. \] 
In view of the definition (\ref{best-constant}) of $C(p,q;R)$, we obtain $C(p,q;R) \leq C(1 + C(p,q;R)^{\rho})$. But this implies that $C(p,q;R)$ is bounded above by a constant depending only on $K,p,q$, but not on $R$, which is the desired conclusion (\ref{goal}). 
\end{proof}

\section{Differentiation results}\label{AP-section}

\subsection{Proof of Theorem \ref{AP-diff} and Theorem \ref{thm-main3}(\ref{lower-dim-diff})}

Assume that $\{S_k \}$ is a sequence of sets for which the maximal operator $\tilde{\mathcal M}$ is bounded on $L^p(\rr)$ for some $p \in (1, \infty)$. We claim that in this case $\{rS_k \}$ differentiates $L^p$ in the sense that (\ref{diff-e10}) holds.

Let $f\in L^p[0,1]$.  We need to prove that
\begin{equation}\label{ap-e0} \lim_{r\to 0} \sup_{k\geq 1} |A_r[k]f(x) - f(x)|=0
\end{equation}
for almost all $x$, where the averages $A_r[k]$ are defined as in (\ref{average-def}).
In other words, it suffices to show that for any $\lambda > 0$ 
\begin{equation} \label{ap-e1}
\Bigl|\Bigl\{x : \lim_{r \rightarrow 0} \sup_{k \geq 1} \bigl| A_r[k]f(x) - f(x) \bigr| > \lambda \Bigr\} \Bigr| = 0. 
\end{equation} 
To this end, fix $t > 0$ and a continuous function $f_{t}$ on $[0,1]$ such that $\|f - f_{t}\|_p < \epsilon$. Since (\ref{ap-e0}) holds for all $x$ for continuous functions, 
\begin{align*}
\Bigl|\Bigl\{x : \lim_{r \rightarrow 0} &\sup_{k \geq 1} \bigl| A_r[k]f(x) - f(x) \bigr| > \lambda \Bigr\} \Bigr| \\ &= \Bigl|\Bigl\{x : \lim_{r \rightarrow 0} \sup_{k \geq 1} \Bigl| A_r[k](f- f_{t})(x) - (f - f_{t})(x) \Bigr| > \lambda \Bigr\} \Bigr| \\ &\leq \Bigl|\Bigl\{x : \tilde{\mathcal M}(f - f_{t})(x) > \frac{\lambda}{2} \Bigr\} \Bigr| + \Bigl|\Bigl\{x : |f-f_{t}|(x) > \frac{\lambda}{2} \Bigr\} \Bigr| \\ &\leq \frac{2^p \|\tilde{\mathcal M}(f-f_{t})\|_p^p}{\lambda^p} + \frac{2^p \|f-f_{t}\|_p^p}{\lambda^p} \leq C_p \lambda^{-p} t^p, 
\end{align*} 
where the last step uses the boundedness of $\tilde{\mathcal M}$ on $L^p$. Since $t$ was arbitrary, (\ref{ap-e1}) and hence (\ref{ap-e0}) are proved. 

The proof of (\ref{diff-e-meas}) is similar, except that we use the bounds on the maximal
operator $\tilde{\mathfrak M}$ instead of $\tilde\calm$.  The details are left to the interested reader.

\subsection{The $L^1$ case}\label{preiss-example}

The following proposition, due to David Preiss (private communication), shows that
(\ref{diff-e-meas}) cannot hold for all $f\in L^1(\rr)$ if $\mu$ is a probability measure singular
with respect to Lebesgue.  
\begin{proposition}\label{preiss-prop}
Suppose that $\mu$ is a probability measure on $\rr$ such that its restriction
to $\rr\setminus\{0\}$ is not absolutely continuous with respect to the Lebesgue
measure.  Then there is a function $f\in L^1(\rr)$ such that for every $x\in\rr$ the set
$$
Z_x=\Big\{r\in(0,\infty):\ \int f(x+ry)d\mu(y)=\infty\Big\}
$$
is dense in $(0,\infty)$.
\end{proposition}

\begin{proof}
We may choose an $x_0\neq 0$ such that 
$\mu(x_0-r,x_0+r)/(2r)\to\infty$ as $r \searrow 0$
(see \cite[Theorem 7.15]{rudin}).
In particular, there is a $\rho_0>0$ and a continuous function 
$\eta:(0,\infty)\to [0,\infty)$ such that
$\eta(r)\to\infty$ as $r \searrow 0$, $\eta\equiv 0$ on $[\rho_0,\infty)$, $\eta$ is
strictly decreasing on $(0,\rho_0)$, and 
\begin{equation}\label{mulberry-e1}
\frac{\mu(x_0-r,x_0+r)}{2r}\geq\eta(r)\text{ for all } r \in( 0,\rho_0).
\end{equation}
Let $g\in L^1(0,\infty)$ be a continuous, nonnegative and strictly decreasing function 
such that 
$$\int_0^\infty g(y)\eta(\lambda g(y))dy=\infty \text{ for any }\lambda >0.$$
Let $h(x)=g^{-1}(|x|)$ for $0<|x|<\rho_0$ and $h(x)=0$ for $|x|>\rho_0$. Define
$$
f(x)=\sum_{j=1}^\infty 2^{-j}h(x-x_j),
$$
where the sequence $\{x_j\}_{j=1}^\infty$ is dense in $\rr$.  Then $f\in L^1(\rr)$, since
$$
\int_\rr f(x)dx=\int_\rr h(x)dx=\int_0^\infty|\{x:\ h(x)>t\}|dt
=2\int_0^\infty g(t)dt<\infty.
$$ 
We must prove that for any 
$x\in\rr$ and $a<b$, the interval $(a,b)$ contains a point of $Z_x$.
Indeed, by the density of $\{x_j\}$ there is a $j\geq 1$ such that 
$r:=(x_j-x)/x_0\in(a,b).$
Then 
\begin{equation}\label{mulberry-e2}
\begin{aligned}
\int h(x-x_j+ry)d\mu(y)&=\int h(r(y-x_0))d\mu(y)\\
&=\int_0^\infty \mu(\{y:\ h(r(y-x_0))>t\})dt\\
&=\int_0^\infty \mu\Big(x_0-\frac{g(t)}{r},x_0+\frac{g(t)}{r}\Big)dt\\
&\geq \int_0^\infty \frac{g(t)}{r}\eta\Big(\frac{g(t)}{r}\Big)dt=\infty,\\
\end{aligned}
\end{equation}
so that
\begin{equation}\label{mulberry-e3}
\int f(x+ry)d\mu(y)\geq \sum_{j=1}^\infty 2^{-j}\int h(x-x_j+ry)d\mu(y)=\infty
\end{equation}
as required.

\end{proof}

\textit{Remark.}
The above argument can be adapted to show that (\ref{diff-e10}) cannot hold for all $f\in L^1(\rr)$
if $\{E_k\}$ is a decreasing sequence of subsets
of $[1,2]$ (or any other interval separated from zero) with $|E_k|\to 0$.  Namely, fix 
any such sequence $\{E_k\}$ and let $\phi_k={\bf 1}_{E_k}/|E_k|$ as before.  Then
there is a subsequence $\{\phi_{j_k}\}_{k=1}^\infty$ converging weakly to a probability
measure $\mu$ supported on a set $E$ of measure 0.  Without loss of generality we may
assume that $j_k=k$. 
Let also $\mu_k$ be the absolutely continuous measure with density
$\phi_k$.  We claim that 
\begin{equation}\label{mulberry-e4}
\lim_{k\to\infty} \int f(x+ry)d\mu_k(y)dy=\infty \text{ for all }r\in Z_x,\ x\in\rr.
\end{equation}
To prove (\ref{mulberry-e4}), we first ask the reader to verify that (\ref{mulberry-e1}) implies the following 
statement: for every $\rho_1>0$ there is a $K=K(\rho_1)$ such that
\begin{equation*}
\frac{\mu_k(x_0- \rho,x_0+\rho)}{2\rho}\geq \frac{1}{4}\eta(\rho)\text{ for all }
\rho_1<\rho<\rho_0,\ k>K(\rho_1).
\end{equation*}
With $g,h,f$ as above, we then have as in (\ref{mulberry-e2})  
\begin{equation*}
\begin{aligned}
\int h(x-x_j+ry)d\mu_k(y)
&=\int_0^\infty \mu_k\Big(x_0-\frac{g(t)}{r},x_0+\frac{g(t)}{r}\Big)dt\\
&\geq \int_0^{R} \frac{g(t)}{2r}\eta\Big(\frac{g(t)}{r}\Big)dt,\\
\end{aligned}
\end{equation*}
for any $R>0$,
provided that $k>K(g(R)/r)$.  Since the last integral can be made arbitrarily large as
$R\to\infty$, (\ref{mulberry-e4}) follows as in (\ref{mulberry-e3}).

}

\bibliographystyle{amsplain}

\noindent{\sc Department of Mathematics, University of British Columbia, Vancouver,
B.C. V6T 1Z2, Canada}
                                                                                     
\noindent{\it ilaba@math.ubc.ca, malabika@math.ubc.ca}

\end{document}